\documentclass[12pt, a4paper, reqno]{amsart}
\usepackage{amsmath}
\usepackage{amsthm}
\usepackage{mathtools}
\usepackage{amsfonts}
\usepackage{amssymb}
\usepackage{MnSymbol}
\usepackage{hyperref}
\hypersetup{colorlinks=true,linkcolor=blue,filecolor=magenta,urlcolor=cyan,}
\usepackage{cleveref}
\usepackage{xcolor}
\usepackage{enumitem}
\usepackage{tikz}
\usepackage{tikz-cd}
\usetikzlibrary{shapes}
\usetikzlibrary{calc}

\usepackage[margin=2cm,footskip=0.25in]{geometry}
\usepackage[framemethod=TikZ]{mdframed}
\mdfdefinestyle{MyFrame}{%
linecolor=black,
outerlinewidth=1pt,
roundcorner=10pt,
innertopmargin=7pt,
innerbottommargin=7pt,
innerrightmargin=20pt,
innerleftmargin=20pt,
backgroundcolor=white}

\usepackage{amscd, amssymb}
\usepackage{amsthm}
\usepackage{tikz}
\usetikzlibrary{arrows,positioning} 
\usepackage{mathrsfs}
\usepackage{bbm}
\usepackage{ stmaryrd }

\def\C{\mathbb{C}}

\def\R{\mathbb{R}}

\def\Z{\mathbb{Z}}

\def\AA{\mathcal{A}}
\def\TT{\mathcal{T}}

\def\UU{\mathcal{U}}

\def\FF{\mathcal{F}}

\def\NN{\mathcal{N}}

\def\SS{\mathcal{S}}
\def\CC{\mathcal{C}}

\def\n{\mathfrak{n}}
\def\e{\mathfrak{e}}
\def\p{\mathfrak{p}}
\def\q{\mathfrak{q}}
\def\g{\mathfrak{g}}
\def\h{\mathfrak{h}}
\def\r{\mathfrak{r}}

\def\l{\mathfrak{l}}
\def\s{\mathfrak{s}}
\def\o{\mathfrak{o}}

\def\ol{\overline}
\def\rk{\operatorname{rk}}
\def\d{\partial}

{\newtheorem{theorem}{Theorem} 
	\newtheorem{lemma}{Lemma}
	\newtheorem{proposition}{Proposition}
	\newtheorem{corollary}{Corollary}
	
	\theoremstyle{definition}
	\newtheorem{definition}{Definition}
	\theoremstyle{definition}
	\newtheorem{remark}{Remark}
	\theoremstyle{definition}
	\newtheorem{example}{Example}}
\numberwithin{theorem}{section}
\numberwithin{lemma}{section}
\numberwithin{remark}{section}
\numberwithin{definition}{section}
\numberwithin{corollary}{section}
\numberwithin{proposition}{section}
\numberwithin{example}{section}

\crefname{conjecture}{conjecture}{conjectures}

\usepackage{relsize}
\usepackage{fancyhdr}
\usepackage[all,cmtip]{xy}

\begin{document}

	\title{A queer Kac--Moody Construction}
	
	\author[A.~Sherman, L.~Silberberg]{Alexander Sherman, Lior Silberberg}
	
	\begin{abstract} We introduce a new, Kac--Moody-flavoured construction for Lie superalgebras, which incorporates phenomena of the type Q (queer) Lie superalgebra.  This is done by replacing a maximal even torus by the most general possible Cartan subalgebra for Lie superalgebras, which is a maximal quasitoral subalgebra.  The theory is remarkably rigid but nevertheless unveils a new natural class of Lie superalgebras, which we call Type Q Kac--Moody (QKM) algebras.  We classify finite-growth type Q Kac--Moody algebras, and obtain in a novel way the $d=2$, $\NN=1,2,3,4$ twisted superconformal algebras, along with three other new, finite growth Lie superalgebras.  Our work also gives a new perspective on the distinctiveness of the Lie superalgebra $\q(n)$.  
	\end{abstract}
	
	\maketitle
	\pagestyle{plain}
	
	\section{Introduction}
	
	In 1968, Kac \cite{K3} and Moody \cite{Moody} introduced a class of Lie algebras that are now known as Kac--Moody algebras.  On the surface level, the theory of such algebras is a simple yet beautiful extension of the theory of semisimple Lie algebras as developed by Lie, Killing, and Cartan.   Namely, from the data of an $n\times n$ complex matrix $A$ (often assumed to satisfy certain conditions), one produces a Lie algebra $\g(A)$, and if $A$ is nice enough (i.e.~symmetrizable, generalized Cartan) then $\g(A)$ admits a tidy presentation in terms of generators (the Chevalley generators) and relations (the Chevalley-Serre relations).   
	
	Although the extent of their applications were not immediately apparent, in the 70s and 80s it dawned that Kac--Moody algebras, and in particular the class of affine Lie algebras, have rich connections to physics, number theory, modular forms, theta functions, differential equations, the Virasoro algebra, and more.  We refer to \cite{K} for a comprehensive study of such algebras and their applications to other fields, as well as a wealth of references to other works.
	
	Somewhat in parallel, the theory of Lie superalgebras had a start, with a major milestone being Kac's seminal 1977 paper \cite{K2} in which, among other things, the simple finite-dimensional Lie superalgebras were classified.  It was realized already in \cite{K2} that the theory of Kac--Moody Lie algebras admits an obvious extension to the super setting, where one starts with an $n\times n$ matrix $A$ along with a parity function $p:\{1,\dots,n\}\to\Z/2\Z$.  One obtains (up to central quotients and derived subalgebras) all simple, basic (admitting an even invariant form) classical Lie superalgebras from this approach.  Finite growth Kac--Moody superalgebras were classified later on in the works of Kac (\cite{K4}) (the case with no isotropic simple roots), Van-de Leur (\cite{L}) (the symmetrizable case), and finished by Hoyt and Serganova (see \cite{C} and \cite{CS}).  The theory of affine Lie superalgebras has connections once again to physics and number theory, amongst other fields, see for instance \cite{KW1}, \cite{KW2}, and \cite{KW3}.   See \cite{GHS} for a more recent approach to Kac--Moody superalgebras, which nicely clarifies the role of the Weyl groupoid.
	
	\subsection{The type Q case} Amidst the progress in understanding Kac--Moody Lie superalgebras, a certain Lie superalgebra avoided the Kac-Moody fold, despite being contragredient and simple (in an appropriate sense).  This is the type Q Lie superalgebra $\q(n)$.  In the theory of finite-dimensional associative superalgebras over the complex numbers, there are two families of central simple superalgebras: $\operatorname{End}(\C^{m|n})$ and $Q(n)$, where the latter is the type Q associative superalgebra, and it can be viewed as the subalgebra of $\operatorname{End}(\C^{n|n})$ consisting of endomorphisms commuting with a particular odd automorphism.  One may present $Q(n)$ as matrices of the form
	\[
	\begin{bmatrix}
		A & B\\ B & A
	\end{bmatrix}
	\]
	where $A,B$ are arbitrary $n\times n$ matrices.  Thus one arrives at the Lie superalgebra $\q(n)$, which is the natural Lie superalgebra obtained from $Q(n)$ via the supercommutator.  
	
	The Lie superalgebra $\q(n)$ has even part $\g\l(n)$ and odd part the adjoint representation of $\g\l(n)$.  It is distinctive in that its Cartan subalgebra, 
	\[
	\h=\begin{bmatrix}
		D & D'\\ D' & D
	\end{bmatrix},
	\]
	where $D,D'$ are diagonal, is \emph{not} purely even or commutative.  It is instead what we call \emph{quasitoral}, i.e.~it satisfies $[\h_{\ol{0}},\h]=0$.  This property prevents it from ever being described in classical Kac--Moody terms, because in that setup one always begins with a self-normalizing, purely even torus.  In fact the root system of $\q(n)$ is nothing but the classical $A_{n-1}$ root system.   Thus it is clear that to have a Kac--Moody type construction that gives rise to $\q(n)$, one needs to begin with a self-normalizing quasitoral subalgebra.
	
	On the flip side, quasireductive Lie superalgebras $\g$ (i.e.~ones for which $\g_{\ol{0}}$ is reductive and acts $\operatorname{ad}$-semisimply on $\g$) are a main focus of super representation theory, as their representation categories are Frobenius (see \cite{S2}).  The Cartan subalgebra of a quasireductive superalgebra $\g$ must be quasitoral, meaning that quasitoral superalgebras are the correct generalization of tori to the super setting.  This motivates a Kac--Moody construction which begins with a self-normalizing quasitoral subalgebra, which we now explain.
	
	\subsection{The classical construction} Without going into too many details, we now describe the general construction that is developed in the paper.  Let us start with the classical construction, which is a special case.  Starting with an $n\times n$ matrix $A=(a_{ij})$ and parity function $p:\{1,\dots,n\}\to\Z/2\Z$, one begins by finding a finite-dimensional abelian Lie algebra $\h$ along with linearly independent sets 
	\[
	\Pi=\{\alpha_1,\dots,\alpha_n\}\subseteq\h^*, \ \text{ and } \ \Pi^\vee=\{h_1,\dots,h_n\}\subseteq\h, \text{ satisfying } \alpha_j(h_i)=a_{ij}.
	\]
	One then takes the Lie algebra generated by $\h$ and Chevalley generators $e_1,\dots,e_n,f_1,\dots,f_n$, where the parity of $e_i$ and $f_i$ is $p(i)$, with imposed relations
	\[
	[h,e_{i}]=\alpha_i(h)e_i, \ \ \ [h,f_i]=-\alpha_i(h)f_i, \ \ \ [e_i,f_j]=\delta_{ij}h_i.
	\]
	Then one defines $\g(A)$ to be the quotient of this algebra by the maximal set of relations that don't kill any elements in $\h$.  If $A$ is generalized Cartan and symmetrizable, then these relations are generated by the Serre relations (\cite{GK}) (see \cite{LS} for the super case).
	
	To understand our generalization, one should view the choices of $\alpha_i\in\h^*$ and $p(i)\in\Z/2\Z$ as choices of irreducible representations of $\h$, which are necessarily one-dimensional: indeed, $\alpha_i$ tells you the action, and $p(i)$ tells you the parity of the representation.  Thus the Chevalley generators are just nonzero elements (i.e.~bases) of corresponding irreducible representations of $\h$.
	
	\subsection{Type Q construction} Let $\h$ be a quasi-toral Lie superalgebra, and set $\mathfrak{t}:=\h_{\ol{0}}$ (which is necessarily abelian).  The irreducible representations of $\h$ are in bijection (up to parity) with $\mathfrak{t}^*$; however even when $\mathfrak{t}$ acts semisimply, which we assume, the category of $\h$-modules is not semisimple, which adds both a complication and richness to the theory. 
	
	Let $\Pi=\{\alpha_1,\dots,\alpha_n\}\subseteq\mathfrak{t}^*$ be linearly independent, and choose irreducible representations 
	\[
	\g_{\alpha_1},\dots,\g_{\alpha_n},\g_{-\alpha_1},\dots,\g_{-\alpha_n}
	\] 
	of $\h$ with specified weights $\pm\alpha_1,\dots,\pm\alpha_n$.  These irreducible representations correspond to a set of Chevalley generators.  Finally, choose (nontrivial) $\h$-equivariant maps $[-,-]_{i}:\g_{\alpha_{i}}\otimes\g_{-\alpha_i}\to\h$ (these correspond to coroots).
	
	We call the package of data specified above a Cartan datum, $\AA$, and from it one may construct a Lie superalgebra $\g(\AA)$ in a completely analogous fashion as in the classical Kac--Moody setup.  In particular $\g(\AA)$ will contain $\h$ as a self-normalizing subalgebra, will have a root decomposition, will contain $\g_{\pm\alpha_i}$ for all $i$, and it will admit a natural triangular decomposition $\g(\AA)=\n^{-}\oplus\h\oplus\n$.
	
	We note that in the recent preprint \cite{APS}, Kac--Moody approaches similar to ours were explored in the setting of an arbitrary symmetric tensor category.
	
	\subsection{Clifford Kac--Moody algebras} The above setup is very general (in fact one doesn't need the $\h$-modules $\g_{\alpha_i}$ to be irreducible), and to get something with more structure it makes sense to impose the condition of integrability (see Definition \ref{definition-integrable}).  This gives us the class of what we call Clifford Kac--Moody algebras.  Here the name Clifford signifies that the simple root spaces are representations of Clifford algebras.  
	
	It is natural to then classify Clifford Kac--Moody algebras with one simple root, and determine how the roots can `interact' with one another.  We have done this in Sections 5 and 6, and the remarkable answer we obtain is depicted in the diagram below.  To explain, the `nodes' of the diagram represent possible Clifford Kac--Moody algebras with one simple root $\alpha$, and we have drawn an arrow $\alpha\to\beta$ if it is possible (under our integrability assumption) that $[\h_{\alpha},\g_{\beta}]\neq0$, where $\h_{\alpha}=[\g_{\alpha},\g_{-\alpha}]$ is the space of $\alpha$-coroots.  This is the analogue to the condition that $\beta(h_\alpha)\neq0$ in a Kac--Moody algebra.  In particular, a loop at a node means that it is possible for two simple roots of that type to interact.
	
	Here the node `Super KM' contains the three possible root subalgebras $\s\l(2),\o\s\p(1|2)$, and $\s\l(1|1)$.  These subalgebras can interact with each other in every possible way, so to speak, which is what the loop there represents.  They give rise to classical Kac--Moody superalgebras.
	
	Similarly, the node `Type Q KM' represents the three possible root subalgebras which appear in our definition of a \emph{type Q Kac--Moody} algebra: they are (up to central quotient) odd loop superalgebras $\mathfrak{t}\s:=\s\otimes\C[\xi]\oplus\C\langle c\rangle$, where $\C[\xi]=\C\langle1,\xi\rangle$ is a supersymmetric polynomial algebra on one odd variable, $c$ is central, and $\s=\s\l(2),\o\s\p(1|2),$ or $\s\l(1|1)$.  The formula for the bracket is given in Example \ref{example-takiff}, and is entirely analogous to the loop algebra construction.  Such an odd affinization is what we refer to as a `Takiff' construction, following an established convention, and in recognition of Takiff's work \cite{T}.
	
	Thus we see that type Q Kac--Moody (QKM) nodes are obtained from super Kac--Moody nodes by applying the Takiff construction.  Further, the diagram shows that QKM nodes can interact with each other in interesting ways.

	\begin{figure}[h] 
	\[
	\xymatrix{
		&&&\text{Super KM}  \ar@(ul,ur) \ar@/_3pc/[ddddlll] \ar@/^3pc/[ddddrrr]\ar[ddr]\ar[ddl] \ar@/^2pc/[ddddl]\ar@/_2pc/[ddddr]&&&\\
		&&&&&&\\
		&&\text{Type Q KM}\ar[dd] \ar[ddrr] \ar@(ul,ur)\ar@/^2pc/@{-->}[rr] &&\ar[ll] \ar[dd]\ar[lldd]H_1 &&\\
		&&&&&&\\
		\h\e(0) & & \h\e(2)^{\Pi} && H_n,\ n\geq 3 & & \h\e(0)^{\Pi}
	}
	\]
	\caption{Clifford Kac--Moody algebras}\label{Fig1a}
\end{figure}
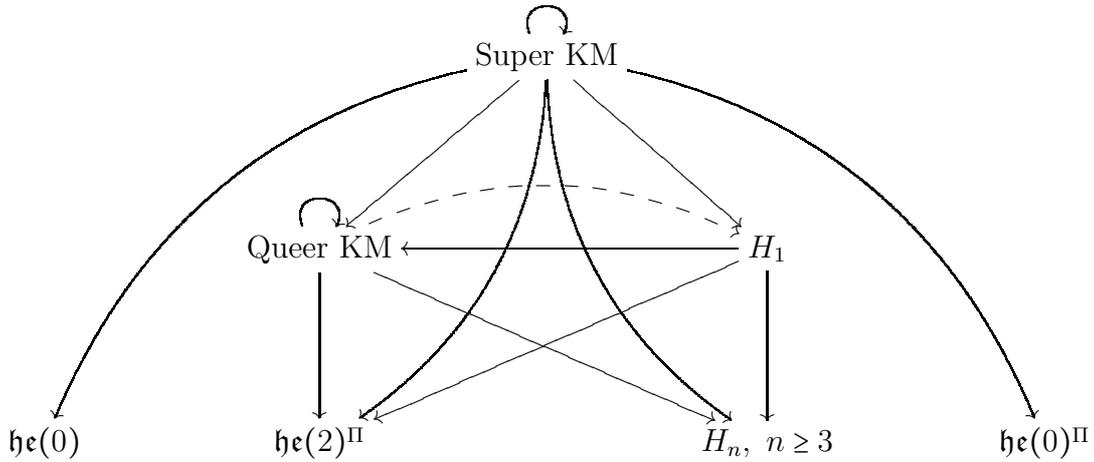

	All other nodes are of `Heisenberg' type, which means that $[\h_{\alpha},\g_{\alpha}]=0$.  The diagram illustrates that they cannot interact in interesting ways with themselves or other nodes under our integrability assumption; in particular their simple root spaces will generate nontrivial ideals inside of Clifford Kac--Moody superalgebras.  (See Section \ref{section-connectivity} for further explanation and for the meaning of the dashed arrow).

	\subsection{Type Q Kac--Moody algebras} As stated above, type Q Kac--Moody (QKM) algebras are those constructed from simple roots of type $\mathfrak{t}\s$, where $\s=\s\l(2),\o\s\p(1|2)$, or $\s\l(1|1)$.  Note also that $\mathfrak{t}\s\l(2)\cong\mathfrak{sq}(2)=[\q(2),\q(2)]$.  
	
	Thus arises the problem of classifying such superalgebras, and in particular the classification of the finite growth QKM algebras.  One simple way to produce a QKM algebra is to start with a Kac--Moody superalgebra $\s$ and apply the Takiff construction to obtain $\s\otimes\C[\xi]\oplus\langle c\rangle$.  Although this is a beautiful and important superalgebra (and in the finite-dimensional case a slight variant of it was studied in \cite{CC}), we view the Takiff construction as somewhat degenerate: in particular every simple root shares the central element $c$ as a coroot.  We thus call such algebras `completely coupled' (or $Y$-coupled to be more precise).
	
	More generally we define notions of being completely $X$-coupled, completely $Y$-coupled, and completely uncoupled, the latter being the least degenerate case (see Definition \ref{definition_coupling}).  In Section \ref{section_coupled_v_uncoupled} we prove the following:
	
	\begin{theorem}
		An indecomposable, finite growth QKM algebra is either completely $X$-coupled, completely $Y$-coupled, or completely uncoupled.  
	\end{theorem}
	
	This allows us to separate our study into the three separate classes.  For the following, we note that QKM algebras have Dynkin diagrams with vertices  \begin{tikzpicture}	\node[diamond,aspect=1,scale=0.6,draw] (A) {};	\end{tikzpicture} for $\mathfrak{t}\s\l(2)$ simple roots, \begin{tikzpicture}	\node[diamond,aspect=1,scale=0.6,draw,fill=black] (A) {};	\end{tikzpicture} for $\mathfrak{t}\o\s\p(1|2)$ simple roots, and $\diamondtimes$ for $\mathfrak{tsl}(1|1)$ simple roots.  We put labels on edges to indicate values of pairings of roots with coroots (see Section \ref{section_dynkin}).
	
	\begin{theorem}\label{thm intro type X}
		Let $\g$ be an indecomposable, completely $X$-coupled QKM algebra with more than one simple root, without any assumptions on growth conditions.  Then $\g$ is of finite growth with GK dimension 1, and is one of three possible superalgebras with the following Dynkin diagrams:
		
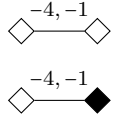
\begin{center}
			\begin{tikzpicture}
				\node (C) at (-5,0) {$\q^X(A_1^{(1)})$};
				\node[diamond,aspect=1,scale=0.6,draw] (A) at (-0.5,0) {};
				\node[diamond,aspect=1,scale=0.6,draw] (B) at (0.5,0) {};
				\draw (A) edge node[above,font=\tiny] {\(-2,-2\)} (B);
			\end{tikzpicture}
		\end{center}
		\begin{center}
			\begin{tikzpicture}
				\node (C) at (-5,0) {$\q^X(A_{2}^{(2)})$};
				\node[diamond,aspect=1,scale=0.6,draw] (A) at (-0.5,0) {};
				\node[diamond,aspect=1,scale=0.6,draw] (B) at (0.5,0) {};
				\draw (A) edge node[above,font=\tiny] {\(-4,-1\)} (B);
			\end{tikzpicture}
		\end{center}
		\begin{center}
			\begin{tikzpicture}
				\node (C) at (-5,0) {$\q_{\bullet}^X(A_2^{(2)})$};
				\node[diamond,aspect=1,scale=0.6,draw] (A) at (-0.5,0) {};
				\node[diamond,aspect=1,scale=0.6,draw,fill=black] (B) at (0.5,0) {};
				\draw (A) edge node[above,font=\tiny] {\(-2,-1\)} (B);
			\end{tikzpicture}
		\end{center}
	\end{theorem}

For the next theorem, let $\Theta$ denote the map which takes a Dynkin diagram of a QKM algebra and turns the diamonds into circles, thus giving the Dynkin diagram of Kac--Moody superalgebra.
	
	\begin{theorem}\label{thm_intro_Y_coupled}
		Suppose that $\g$ is a $Y$-coupled, QKM algebra with Dynkin diagram $D$.   Let $\s$ be the Kac--Moody superalgebra obtained from the Dynkin diagram $\Theta(D)$. Then $\g$ is constructed from $\s$ via the Takiff construction (see Theorem \ref{theorem-Takiffs} for a precise statement).
	\end{theorem}
	
	\subsection{Completely uncoupled QKM algebras} Theorem \ref{thm_intro_Y_coupled} tells us that any possible finite-growth Dynkin diagram can appear in our construction.  However if we pass to the completely uncoupled situation, things become more rigid.  Namely we prove the following (Corollary \ref{corollary-fg-cc-QKM}):
	
	\begin{theorem}\label{theorem_intro_completely_uncoupled}
		The possible Dynkin diagrams of an indecomposable, completely uncoupled QKM algebra of finite growth with more than one simple root are the following:		
		\begin{enumerate}
			\item Type \(A(n)\) for \(n \in \mathbb{Z}_{\geq 2}\), giving rise to $\q(n+1)$:
			
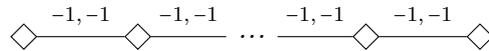
\begin{center}
				\begin{tikzpicture}
					\node[diamond,aspect=1,scale=0.6,draw] (A) at (-3,0) {};
					\node[diamond,aspect=1,scale=0.6,draw] (B) at (-1.5,0) {};
					\node (C) at (0,0) {...};
					\node[diamond,aspect=1,scale=0.6,draw] (D) at (1.5,0) {};
					\node[diamond,aspect=1,scale=0.6,draw] (E) at (3,0) {};
					\draw (A) edge node[above,font=\tiny] {\(-1,-1\)} (B);
					\draw (B) edge node[above,font=\tiny] {\(-1,-1\)} (C) ;
					\draw (C) edge node[above,font=\tiny] {\(-1,-1\)} (D) ;
					\draw (D) edge node[above,font=\tiny] {\(-1,-1\)} (E) ;
				\end{tikzpicture}
			\end{center}
			with \(n\) vertices.
			\item Type \(A(n)^{(1)}\) for \(n \in \mathbb{Z}_{\geq 2}\), giving rise to $\q(n+1)^{(1)}$, an affinization of $\q(n)$:
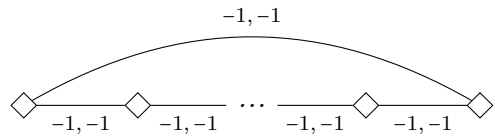
			\begin{center}
				\begin{tikzpicture}
					\node[diamond,aspect=1,scale=0.6,draw] (A) at (-3,0) {};
					\node[diamond,aspect=1,scale=0.6,draw] (B) at (-1.5,0) {};
					\node (C) at (0,0) {...};
					\node[diamond,aspect=1,scale=0.6,draw] (D) at (1.5,0) {};
					\node[diamond,aspect=1,scale=0.6,draw] (E) at (3,0) {};
					\draw (A) edge node[below,font=\tiny] {\(-1,-1\)} (B);
					\draw (B) edge node[below,font=\tiny] {\(-1,-1\)} (C) ;
					\draw (C) edge node[below,font=\tiny] {\(-1,-1\)} (D) ;
					\draw (D) edge node[below,font=\tiny] {\(-1,-1\)} (E) ;
					\draw (A) edge [bend left] node[above,font=\tiny] {\(-1,-1\)} (E);
				\end{tikzpicture}
			\end{center}
			with \(n+1\) vertices.
			\item Type \(A(1)^{(1)}\), giving rise to two superalgebras $\q_{(2,2)}^{\pm}$:
			\begin{center}
				\begin{tikzpicture}
					\node[diamond,aspect=1,scale=0.6,draw] (A) at (-0.75,0) {};
					\node[diamond,aspect=1,scale=0.6,draw] (B) at (0.75,0) {};
					\draw (A) edge node[above,font=\tiny] {\(-2,-2\)} (B) ;
				\end{tikzpicture}
			\end{center}
		\end{enumerate}
	\end{theorem}	
	One outcome of Theorem \ref{theorem_intro_completely_uncoupled} is that the only finite-type, completely uncoupled QKM root system we can obtain is $A(n)$, and ultimately the only algebra we can get (up to central extensions/quotients and derivations/commutator subalgebras) is $\q(n)$.  Thus it shows that there cannot be a type Q version of other simple root systems.
	
	Another outcome is the existence of two nontrivial completely uncoupled finite-growth Lie superalgebras $\q^{\pm}_{(2,2)}$.  These Lie superalgebras both have root system of type $A^{(1)}_{1}$, and contain $\widehat{\s\l}_{2}$ as a subalgebra. The following will be shown in future work:
	\begin{theorem}\label{thm superconformal}
			The Lie superalgebra $\q_{(2,2)}^+$ is isomorphic to the $d=2$, $\NN=3$ twisted superconformal algebra, and $\q_{(2,2)}^-$ is isomorphic to the $d=2$, $\NN=4$ twisted superconformal algebra.
	\end{theorem} 

  In the language of \cite{KL}, these superalgebras are $\s\o(3)$ and $\s\o(4)$ superconformal algebras.  We can also naturally produce the twisted, $d=2$, $\NN=1,2$ superconformal algebras via our construction; see Section \ref{section ramond}.
	
	\subsection{Serre relations} Establishing Serre relations is important in the study of Kac--Moody superalgebras, in particular for quantizations of superalgebras.  It is natural to ask when QKM algebras admit nice presentations.  We have the following, which is proven in \cite{Si}.
	
	\begin{theorem}
		Let $\g$ be a finite growth, completely uncoupled QKM algebra with simple roots $\{\alpha_1,\dots,\alpha_n\}$, where $n>1$. Let $A=(a_{ij})$ be its Cartan matrix (see Section \ref{section_dynkin}).  Then $\g$ is generated by $\h$ and simple root spaces $\g_{\pm\alpha_1},\dots,\g_{\pm\alpha_i}$, subject to the following relations:
		\begin{enumerate}
			\item Chevalley-type relations (those defining $\g(\AA)$, as listed in Section \ref{section_cartan_datum}); and
			\item Serre relations: for $i\neq j$ we have
			\[
			\operatorname{ad}(\g_{\alpha_i})^{1-a_{ij}}(\g_{\alpha_j})=0,  \ \ \ \ \operatorname{ad}(\g_{-\alpha_i})^{1-a_{ij}}(\g_{-\alpha_j})=0.
			\]
		\end{enumerate}
	\end{theorem}
	
	\subsection{Connections to physics}  The $d=2$, $\NN=1,2,3,$ and $4$ superconformal algebras play important roles in conformal field theory and string theory.  Their representation theory, in particular the classification of unitary representations, is of importance (see the recent work \cite{KFP}).  It is possible that our novel presentation of these superalgebras may lead to new insights on their representation theory, using techniques from Kac--Moody algebras.  Further, it would be interesting to know if the three finite-growth algebras $\q^X(A_{1}^{(1)})$, $\q^X(A_{2}^{(2)})$, and $\q_{\bullet}^X(A_{2}^{(2)})$ from Theorem \ref{thm intro type X} have known realizations or connections to physics.  
	
	\subsection{Future work} In addition to finding realizations for the algebras in Theorems \ref{thm intro type X}, we plan to study of representation theory of QKM algebras.  One can ask for descriptions of integrable representations, character formulas, denominator identities, Shapovalov determinants (to extend the work of \cite{G}), and more.  In particular we hope that nice identities can be found in the Grothendieck ring of $\h$, such as was begun for $\q(n)$ in \cite{GSS}.

	\subsection{Acknowledgements}  The authors would like to thank Maria Gorelik for many helpful discussions and suggestions.  We thank Vera Serganova for numerous suggestions and comments, including those which led to the realization of the $\NN=3$ superconformal algebra.  We thank David Ridout for comments that led to the realization of the $\NN=4$ superconformal algebra.  The first author was partially supported by ARC grant DP210100251, ISF grant 711/18 and NSF-BSF grant 2019694.  The second author was partially supported by the funding received from the MINERVA Stiftung with the funds from the BMBF of the Federal Republic of Germany.
	
	\tableofcontents

	\subsection{List of notation}
	
	\begin{itemize}
		\item $\CC\ell(V,B)$ the Clifford superalgebra on a vector space $V$ with symmetric bilinear form $B$;
		\item $\SS(V)$ the supersymmetric algebra on $V$;
		\item $\h$ a finite-dimensional quasi-toral Lie superalgebra (Sec.~\ref{section-irreducible-h-modules});
		\item $\mathfrak{t}=\h_{\ol{0}}$ the even part of a quasitoral Lie superalgebra;
		\item $\FF(\h)$ the category of finite-dimensional $\h$-modules with semisimple action of $\mathfrak{t}$ (Sec.~\ref{section-irreducible-h-modules});
		\item $B_{\lambda}$ the symmetric bilinear form on $\h_{\ol{1}}$ given by $B_{\lambda}(x,y)=\lambda([x,y])$ (Sec.~\ref{section-irreducible-h-modules});
		\item $C_{\lambda}$ an irreducible representation of $\h$ of weight $\lambda$ (Sec.~\ref{section-irreducible-h-modules});
		\item $V^\vee:=V^{\omega_{\h}^{-1}}$ the twist of an $\h$-module $V$ by $\omega_{\h}^{-1}$ (Sec.~\ref{section_dualities});
		\item $\AA$ a Cartan datum (Def.~\ref{definition-datum});
		\item $\Pi=\{\alpha_1,\dots,\alpha_n\}\subseteq\mathfrak{t}^*$ simple roots (Sec.~\ref{section_cartan_datum});
		\item $\h_{\alpha}:=[\g_{\alpha},\g_{-\alpha}]$ the coroot space of $\alpha$ (Def.~\ref{definition_coroot});
		\item $\tilde{\g}(\AA)$ the Lie superalgebra constructed in Sec.~\ref{section_cartan_datum};
		\item $\g(\AA):=\tilde{\g}(\AA)/\r$, see Sec.~\ref{section-g(A)};
		\item $\omega$ the Chevalley automorphism on $\g(\AA)$, see Cor.~\ref{corollary-Chevalley};
		\item $\g\langle\alpha\rangle$ the subalgebra generated by $\g_{\alpha}$ and $\g_{-\alpha}$ (Sec.~\ref{section-one-root});
		\item $T\s$, $\mathfrak{t}\s$, $\mathfrak{pt}\s$ are Takiff constructions, see Ex.~\ref{example-takiff};
		\item $\h\e(n)^{(\Pi)}$ a Heisenberg superalgebra determined by a rank $n$ simple root, see Sec.~\ref{section_heisenberg};
		\item $Tak(\s)$ a type for a simple root, see the table at the start of Sec.~\ref{section-connectivity};
		\item $e_i,E_i,f_i,F_i$ are `Chevalley generators' in a QKM algebra (Sec.~\ref{section-notation-QKM});
		\item $H_i,h_i,c_i$ are the pure coroots of $\alpha_i$ in QKM algebra (Sec.~\ref{section-notation-QKM});
		\item $x_{ij},y_{ij}$ satisfy $[H_i,e_j+E_j]=x_{ij}E_j+y_{ij}e_j$ (Sec.~\ref{section-notation-QKM}).
	\end{itemize}
	
	\section{Quasi-toral Lie superalgebras}\label{section_quasi_toral}
	
	In this section we review some known results on the representation theory of quasi-toral Lie superalgebras (see Definition \ref{definition-quasi-toral}), which will be used frequently in our work. For more details, we refer the reader to \cite{G} and \cite{GSS}.
	
	\subsection{Clifford superalgebras}\label{section_cliff_algs}
	Let \(V\) be a finite-dimensional vector space and let \(B\) be a symmetric (not necessarily non-degenerate) bilinear form on \(V\). Let \(\mathcal{C}\ell(V,B)\) denote the corresponding Clifford superalgebra, with the elements of \(V\) being odd.  Explicitly, $\CC\ell(V,B)$ is the quotient of the tensor superalgebra $\TT V:=\bigoplus\limits_{n\geq0}V^{\otimes n}$ by the ideal generated by expressions of the form $v\otimes w+w\otimes v-B(v,w)$, where $v,w\in V$.  Because this ideal is homogeneous (in particular it is generated by even elements), the quotient inherits the structure of an associative superalgebra.
	
	The form \(B\) induces a non-degenerate symmetric bilinear form on \(V/\ker B\), which we shall also denote by \(B\). We have a well-known, non-canonical isomorphism of superalgebras
	\begin{align}\label{eq_iso_cliff}
		\mathcal{C}\ell(V,B) \simeq \mathcal{C}\ell(V/\ker B,B) \otimes \mathcal{S} (\ker B),
	\end{align}
	where \(\mathcal{S}(\ker B)\) is the superalgebra of supercommutative polynomials in elements of \(\ker B\), a purely odd vector space. In particular $\mathcal{S}(\ker B)$ is isomorphic as an algebra to the Grassmann algebra on $\ker B$ viewed as an even vector space, and thus is finite-dimensional.
	
	Let $m$ denote the dimension of \(V\), and suppose that \(B\) is non-degenerate.  For the following facts we refer to \cite{CW}, Exercise 3.11.
	
	\begin{itemize}
		\item \(\mathcal{C}\ell(V,B)\) is always a simple superalgebra;
		\item all modules over \(\mathcal{C}\ell(V,B)\) are completely reducible;
		\item if \(m\) is odd, \(\mathcal{C}\ell(V,B)\) admits a unique, parity invariant, irreducible module;
		\item if \(m\) is even, \(\mathcal{C}\ell(V,B)\) admits two irreducible modules which differ by parity. 
		\item if \(m \neq 0\) and \(E\) is an irreducible \(\mathcal{C}\ell(V,B)\)-module, then \(\dim E_{\ol{0}} = \dim E_{\ol{1}} = 2^{\lfloor \frac{m-1}{2} \rfloor}\), where $\lfloor-\rfloor$ denotes the floor function.
	\end{itemize} 
	
	\subsection{Irreducible \(\mathfrak{h}\)-modules}\label{section-irreducible-h-modules}
	
	For general background on Lie superalgebras, we refer to \cite{CW} or \cite{M}.
	\begin{definition}\label{definition-quasi-toral}
		A Lie superalgebra \(\mathfrak{h}\) is said to be quasi-toral if \([\mathfrak{h}_{\ol{0}},\mathfrak{h}] = 0\), i.e.~$\h_{\ol{0}}$ is central in $\h$.
	\end{definition}
	
	\begin{example}\label{example of h_n}
		For $n\geq0$, let $\h(n)$ be the Lie superalgebra with $\h(n)_{\ol{0}}=\C\langle c\rangle$, and $\h(n)_{\ol{1}}=\C^n$.  We impose that $c$ is a nonzero, central element of $\h(n)$, and for $x,y\in\h(n)_{\ol{1}}$ we set $[x,y]=(x,y)c$, where $(-,-)$ is a fixed, nondegenerate symmetric bilinear form on $\C^n$.  Then $\h(n)$ is a quasitoral Lie superalgebra.
	\end{example}
	
	For \(\mathfrak{h}\) a finite-dimensional quasi-toral Lie superalgebra, we write \(\mathfrak{t} \coloneq \mathfrak{h}_{\ol{0}}\). Throughout the paper, we will consider only \(\mathfrak{h}\)-modules with semisimple \(\mathfrak{t}\)-action.  Write $\FF(\h)$ for the category of finite-dimensional $\h$-modules with semisimple action of $\mathfrak{t}$.
	
	To a weight \(\lambda \in \mathfrak{t}^*\), we associate a symmetric bilinear form \(B_{\lambda}\) on \(\mathfrak{h}_{\ol{1}}\), given by \(B_{\lambda}(H_1,H_2) = \lambda([H_1,H_2])\).
	\begin{definition}\label{definition-rank}
		For a weight \(\lambda \in \mathfrak{t}^*\), we define the rank of $\lambda$ to be \(\operatorname{rk}\lambda:=\rk B_{\lambda} \in \mathbb{Z}_{\geq 0 }\), where $\mathrm{rk}B_{\lambda}$ is the rank of the bilinear form $B_{\lambda}$ on $\h_{\ol{1}}$.  Observe that $\rk c\lambda=\rk\lambda$ for $c\in\C^\times$.
	\end{definition}
	We consider the universal enveloping algebra \(\mathcal{U}(\mathfrak{h})\) as an algebra over the central polynomial subalgebra \(\UU(\mathfrak{t})=\mathcal{S}(\mathfrak{t})\). Let \(\lambda \in \mathfrak{t}^*\) and let \(I(\lambda)\) be the kernel of the algebra homomorphism \(\mathrm{ev}_{\lambda} : \mathcal{S}(\mathfrak{t})\rightarrow \mathbb{C}\), induced by evaluation at \(\lambda\). We consider the Clifford superalgebra
	\begin{align*}
		\mathcal{C}\ell(\lambda) \coloneq \mathcal{C}\ell(\mathfrak{h}_{\ol{1}},B_{\lambda}) = \mathcal{U}(\mathfrak{h})/I(\lambda)\mathcal{U}(\mathfrak{h}).
	\end{align*}
	Then we have a non-canonical isomorphism of superalgebras (see (\ref{eq_iso_cliff}))
	\begin{align*}
		\mathcal{C}\ell(\lambda) \simeq \mathcal{C}\ell(\mathfrak{h}_{\ol{1}}/\ker B_\lambda,B_{\lambda}) \otimes \mathcal{S} (\ker B_\lambda),
	\end{align*}
	where \(B_{\lambda}\), by abuse of notation, also denotes the induced form on \(\mathfrak{h}_{\ol{1}}/\ker B_\lambda\). \par
	We have a decomposition of $\FF(\h)$ according to central character given by
	\[
	\FF(\h)=\bigoplus\limits_{\lambda\in\mathfrak{t}^*}\FF_{\lambda},
	\]
	where $\FF_{\lambda}$ consists of those modules on which $\mathfrak{t}$ acts by $\lambda$.  Further, $\FF_{\lambda}$ is equivalent to the category of $\CC\ell(\lambda)$-modules.  Thus all irreducible \(\mathfrak{h}\)-modules arise as  \(\mathcal{C}\ell(\lambda)\)-modules, for some \(\lambda \in \mathfrak{t}^*\). We denote by \(C_\lambda\) a choice of unique (up to parity) irreducible \(\mathfrak{h}\)-module on which \(\mathfrak{t}\) acts via \(\lambda\), where we assume that $C_0$ is the trivial module $\C$.   Recall from Section \ref{section_cliff_algs} that if $m=\rk\lambda>0$, then $\dim (C_{\lambda})_{\ol{0}}= \dim(C_{\lambda})_{\ol{1}}=2^{\lfloor \frac{m-1}{2} \rfloor}$.
	
	Note that the blocks of \(\mathcal{F}(\mathfrak{h})\) are more finely parameterized in the following way. If \(B_\lambda\) is degenerate, then $\FF_{\lambda}$ is a block, and as stated above, this block is equivalent to the category of finite-dimensional modules over \(\CC\ell(\lambda)\). If \(B_{\lambda}\) is non-degenerate, then $\FF_{\lambda}$ is semisimple; if \(\rk \lambda\) is odd there is exactly one simple module in $\FF_{\lambda}$, and if \(\rk \lambda\) is even there are two. 
	
	\subsection{Dualities on $\FF(\h)$}\label{section_dualities}
	The category \(\mathcal{F}(\mathfrak{h})\) admits the usual duality \((-)^*\) which is induced by the anti-automorphism of \(\mathcal{U}(\mathfrak{h})\) defined by \(-\mathrm{id}_{\mathfrak{h}}\). We have another duality on \(\mathcal{F}(\mathfrak{h})\) which we denote by \((-)^\#\). It is induced by the anti-automorphism of \(\mathcal{U}(\mathfrak{h})\) defined by \(\mathrm{id}_{\mathfrak{h}_{\ol{0}}}\oplus(\sqrt{-1}\cdot\mathrm{id}_{\mathfrak{h}_{\ol{1}}})\). We have the formula
	\[
	C_{\lambda}^\#\cong \begin{cases}
		C_{\lambda} & \text{ if }\rk\lambda \text{ is odd or }\rk\lambda\equiv0\ \pmod4,\\
		\Pi C_{\lambda} & \text{ if }\rk\lambda\equiv 2\ \pmod4.
	\end{cases}
	\] 
	
	The quasi-toral Lie superalgebra \(\mathfrak{h}\) admits an automorphism \(\omega_{\mathfrak{h}} = (-\mathrm{id}_{\mathfrak{h}_{\ol{0}}})\oplus(\sqrt{-1}\cdot\mathrm{id}_{\mathfrak{h}_{\ol{1}}})\). Given an irreducible \(\mathfrak{h}\)-module \(C_{\lambda}\), we define the twisted module 
	\[
	C_{\lambda}^{\vee} \coloneq C_{\lambda}^{\omega_{\mathfrak{h}}^{-1}}
	\]
	(note the inverse - it will simplify notation later).  On $\FF(\h)$, we have a natural isomorphism of functors \((-)^\vee \cong ((-)^*)^\#\).
	
	From Theorem 3.6.4 of \cite{GSS} we obtain the following isomorphisms of $\h$-modules which we use later: if \(\rk\lambda \) is odd, then 
	\begin{align}\label{equation-2Z+1}
		C_{\lambda}\otimes C_{\lambda}^\vee \simeq \mathcal{S}(\mathfrak{h}_{\ol{1}}/\ker B_{\lambda})\oplus \Pi \mathcal{S}(\mathfrak{h}_{\ol{1}}/\ker B_{\lambda}),
	\end{align}
	where $\SS(\mathfrak{h}_{\ol{1}}/\ker B_{\lambda})$ is the Grassmann algebra on $\mathfrak{h}_{\ol{1}}/\ker B_{\lambda}$, and has the natural action by $\CC\ell(0)\cong\SS(\h_{\ol{1}})$ from left multiplication in the quotient.  In particular $\mathfrak{t}=\h_{\ol{0}}$ will act trivially.  Note that $\SS(\mathfrak{h}_{\ol{1}}/\ker B_{\lambda})$ has a simple socle and top, and thus in particular is indecomposable.
	
	If \(\rk\lambda \in 4\mathbb{Z}\), then we have
	\begin{align}\label{equation-4Z}
		C_{\lambda}\otimes C_{\lambda}^\vee \simeq \mathcal{S}(\mathfrak{h}_{\ol{1}}/\ker B_{\lambda}),
	\end{align}
	while if \(\rk\lambda \in 4\mathbb{Z}+2\), then
	\begin{align}\label{equation-4Z+2}
		C_{\lambda}\otimes C_{\lambda}^\vee \simeq \Pi\mathcal{S}(\mathfrak{h}_{\ol{1}}/\ker B_{\lambda}).
	\end{align}
	We now obtain the following corollary:
	
	\begin{corollary}\label{corollary-tensor-of-irreducibles}
		Let \(\lambda \in \mathfrak{t}^*\) and let \(\phi : C_{\lambda}\otimes C_{\lambda}^\vee \rightarrow \mathfrak{h}\) be an \(\mathfrak{h}\)-module homomorphism.
		\begin{enumerate}
			\item If \(\rk\lambda \in 4\mathbb{Z}\), the $\operatorname{Im}\phi$ is purely even and at most one-dimensional.
			\item If \(\rk\lambda \in 4\mathbb{Z}+2\), then the odd part of the image of \(\phi\) is at most one-dimensional, and it generates $\operatorname{Im}\phi$ as an \(\mathfrak{h}\)-module.
		\end{enumerate}
	\end{corollary}
	\begin{proof}
		This follows immediately from $[\h_{\ol{0}},\h]=0$ and the structure of $C_{\lambda}\otimes C_{\lambda}^\vee$ given in \eqref{equation-4Z} and \eqref{equation-4Z+2}.
	\end{proof}

	\subsection{Nondegenerate morphisms}\label{section_nondegenerate}
	
	\begin{definition}
		Let $\h$ be a quasitoral Lie superalgebra, and let $U,V,$ and $W$ be $\h$-modules.  We say a morphism of $\h$-modules $\phi:U\otimes V\to W$ is nondegenerate if for any submodules $U'\subseteq U$ and $V'\subseteq V$ we have both $\phi(U'\otimes V)\neq0$ and $\phi(U\otimes V')\neq0$.
	\end{definition}

	Observe that if $U$ and $V$ are simple $\h$-modules, then a map $\phi:U\otimes V\to W$ is nondegenerate if and only if $\phi\neq0$.
	
	\subsection{Realization of \(C_{\lambda}\)}\label{subsection-realization-irreducible}
	
	We consider the polynomial superalgebra \(\mathbb{C}[\xi_1,...,\xi_m]\), where $\xi_1,\dots,\xi_m$ are odd, supercommuting variables, i.e.~ \(\xi_i\xi_j = -\xi_j\xi_i\) for all \(i,j\). In particular \(\xi_i^2 = 0\). For a set \(I = \{i_1,...,i_k\} \subseteq \{1,...,m\}\), we denote \(\xi_I \coloneq \xi_{i_1}\cdots \xi_{i_k}\) for \(i_1<...<i_k\) (in particular \(\xi_\emptyset = 1\)). We now explain how to concretely realize the irreducible \(\mathfrak{h}\)-modules \(C_{\lambda}\) as polynomial superalgebras, depending on the rank of \(\lambda\).

	If \(\rk\lambda = 2m\) for some \(m \in \mathbb{Z}_{\geq0}\), then there exist linearly independent vectors \linebreak \(H_1,...,H_m,\bar{H}_1,...,\bar{H}_m\) of \(\mathfrak{h}_{\ol{1}}\) which satisfy 
	\begin{align*}
		\begin{array}{ccc}
			B_\lambda(H_i,\bar{H_j}) = \delta_{ij}, & B_\lambda(H_i,H_j) = 0, & B_\lambda(\bar{H}_i,\bar{H}_j)=0,
		\end{array}
	\end{align*}
	for \(1\leq i,j \leq m\). We can identify \(C_\lambda\) (up to parity) with \(\mathbb{C}[\xi_1,...,\xi_m]\), subject to the following action of \(\mathfrak{h}_{\ol{1}}\): \(H_i\) acts by multiplication by \(\xi_i\) and \(\bar{H}_i\) acts by the odd derivation \(\partial_{\xi_i}\). \par
	If, on the other hand, \(\rk\lambda = 2m+1\) for some \(m \in \mathbb{Z}_{\geq 0}\), then there exist linearly independent vectors \(H_1,...,H_m,\bar{H}_1,...,\bar{H}_m,\tilde{H}\) of \(\mathfrak{h}_{\ol{1}}\) which satisfy 
	\begin{align*}
		\begin{array}{ccc}
			B_\lambda(H_i,\bar{H_j}) = \delta_{ij}, & B_\lambda(H_i,H_j) = 0, & B_\lambda(\bar{H}_i,\bar{H}_j)=0, \\
			B_\lambda(\tilde{H},\tilde{H}) = 2, & B_\lambda(H_i,\tilde{H}) = 0, & B_\lambda(\bar{H}_i,\tilde{H}) = 0,
		\end{array}
	\end{align*}
	for \(1\leq i,j \leq m\). The \(\mathfrak{h}\)-module \(C_\lambda\) can be realized as \(\mathbb{C}[\xi_1,...,\xi_{m+1}]\) with the following action of \(\mathfrak{h}_{\ol{1}}\): \(H_i\) acts by multiplication by \(\xi_i\), \(\bar{H}_i\) acts by the odd derivation \(\partial_{\xi_i}\), and \(\tilde{H}\) acts by \(\xi_{m+1}+\partial_{\xi_{m+1}}\).

	\section{General Construction of $\g(\AA)$}

	In this section we give a construction, in the spirit of the Kac--Moody approach, for Lie superalgebras admitting a quasi-toral Cartan subalgebra. It is a generalization of the constructions given in \cite{K} and \cite{K3} for Lie algebras, and \cite{K2} for Lie superalgebras.
	
	\subsection{Cartan datum and $\tilde{\g}(\AA)$}\label{section_cartan_datum}
	\begin{definition}\label{definition-datum} 
		A Cartan datum \(\mathcal{A}\) consists of the following information:
		\begin{enumerate}
			\item A finite-dimensional quasi-toral Lie superalgebra \(\mathfrak{h}\). We write \(\mathfrak{t} \coloneq \mathfrak{h}_{\ol{0}}\) and \([-,-]_{\mathfrak{h}}\) for the Lie bracket on \(\mathfrak{h}\);
			\item A linearly independent subset \(\Pi = \{\alpha_1,...,\alpha_n\}\subseteq\mathfrak{t}^*\);
			\item For each $\alpha \in \pm \Pi$, an \(\mathfrak{h}\)-module \(\mathfrak{g}_{\alpha}\) in $\FF_{\alpha}$; i.e., as a \(\mathfrak{t}\)-module, $\g_{\alpha}$ is a weight space of weight \(\alpha\). We write \(m_{\alpha} : \mathfrak{h} \times \mathfrak{g}_{\alpha} \rightarrow \mathfrak{g}_{\alpha}\) for the action of \(\mathfrak{h}\) on \(\mathfrak{g}_{\alpha}\);
			\item\label{definition-datum-step-4} For each \(\alpha \in \Pi\), a nondegenerate morphism (see Section \ref{section_nondegenerate}) of \(\mathfrak{h}\)-modules \newline \([-,-]_{\alpha} : \mathfrak{g}_{\alpha}\otimes \mathfrak{g}_{-\alpha} \rightarrow \mathfrak{h}\).
		\end{enumerate}
	\end{definition}
	
	\begin{remark}\label{remark on vanishing}
		Observe that for any $\alpha\in\Pi$, the morphism $[-,-]_{\alpha}$ must vanish on the submodule generated by odd elements in the radical (as an $\h$-module) of $\g_{\alpha}\otimes\g_{-\alpha}$.
	\end{remark}
	
	To a Cartan datum \(\mathcal{A}\) we associate a Lie superalgebra \(\tilde{\mathfrak{g}}(\mathcal{A})\), which is generated by the super vector space
	\begin{align}\label{generating-set}	\mathcal{C} = \bigoplus_{i=1}^n \mathfrak{g}_{-\alpha_i} \oplus \mathfrak{h} \oplus \bigoplus_{i=1}^n \mathfrak{g}_{\alpha_i},
	\end{align}
	subject to the following relations:
	\begin{enumerate}
		\item For any \(h_1,h_2 \in \mathfrak{h}\) we have \([h_1,h_2] = [h_1,h_2]_{\mathfrak{h}}\).
		\item For any \(h \in \mathfrak{h}\), \(\alpha \in \pm\Pi\) and \(x \in \mathfrak{g}_{\alpha}\), we have \([h,x] = m_{\alpha}(h,x)\).
		\item For any \(\alpha,\beta \in \Pi\) with \(x \in \mathfrak{g}_{\alpha}\) and \(y \in \mathfrak{g}_{-\beta}\), we have \([x,y] = [x,y]_{\alpha}\) if \(\alpha = \beta\) and \([x,y] = 0\) otherwise.
	\end{enumerate}
	We denote by \(\tilde{\mathfrak{n}}^+\) and \(\tilde{\mathfrak{n}}^-\) the subalgebras of \(\tilde{\mathfrak{g}}(\mathcal{A})\) generated by \(\bigoplus_{i=1}^n \mathfrak{g}_{\alpha_i}\) and \(\bigoplus_{i=1}^n \mathfrak{g}_{-\alpha_i}\), respectively. We set \(Q_+ \coloneq \mathbb{Z}_{\geq 0} \Pi\). We obtain the following theorem, whose statement and proof are direct generalizations of Theorem 1.2 in \cite{K}.
	\begin{theorem}\label{theorem-half-baked-structure}
		Let \(\tilde{\mathfrak{g}} \coloneq \tilde{\mathfrak{g}}(\mathcal{A})\) and \(\tilde{\mathfrak{n}}^\pm\) be as above. Then
		\begin{enumerate}[label=(\roman*)]
			\item\label{bullet1-half-baked} The subalgebras \(\tilde{\mathfrak{n}}^+\) and \(\tilde{\mathfrak{n}}^-\) are freely generated as Lie superalgebras by \(\bigoplus_{i=1}^n \mathfrak{g}_{\alpha_i}\) and \(\bigoplus_{i=1}^n \mathfrak{g}_{-\alpha_i}\), respectively;
			\item\label{bullet2-half-baked} \(\tilde{\mathfrak{g}} = \tilde{\mathfrak{n}}^- \oplus \mathfrak{h} \oplus \tilde{\mathfrak{n}}^+\) as super vector spaces;
			\item\label{bullet3-half-baked} As a \(\mathfrak{t}\)-module, \(\tilde{\mathfrak{g}}\) admits a root space decomposition:
			\begin{align*}
				\tilde{\mathfrak{g}} = \bigoplus_{\substack{\alpha \in Q_+ \\ \alpha \neq 0}} \tilde{\mathfrak{g}}_{-\alpha} \oplus \mathfrak{h} \oplus \bigoplus_{\substack{\alpha \in Q_+ \\ \alpha \neq 0}} \tilde{\mathfrak{g}}_{\alpha}.
			\end{align*}
			In particular, \(\mathfrak{h}\) is the centralizer of \(\mathfrak{t}\) in \(\tilde{\mathfrak{g}}\), and is self-normalizing;
			\item\label{bullet4-half-baked} \(\tilde{\mathfrak{g}}\) contains a unique maximal ideal \(\mathfrak{r}\) that intersect \(\mathfrak{h}\) trivially; it satisfies \(\mathfrak{r} = (\mathfrak{r}\cap \tilde{\mathfrak{n}}^+)\oplus (\mathfrak{r}\cap \tilde{\mathfrak{n}}^-)\).
		\end{enumerate}
	\end{theorem}
	
	\begin{proof}
		Let \(W\) be any representation of \(\mathfrak{h}\), and consider the super vector space \linebreak \(V = \TT(\oplus_{i=1}^n \mathfrak{g}_{\alpha_i})\otimes W\). The space \(V\) is naturally graded, with \(V_s = (\oplus_{i=1}^n \mathfrak{g}_{\alpha_i})^{\otimes s}\otimes W\). We define a representation \(\pi_W : \tilde{\mathfrak{g}} \rightarrow \operatorname{End}(V)\) by the following action of the generators of \(\tilde{\mathfrak{g}}\). Let \(e \in \oplus_{i=1}^n \mathfrak{g}_{\alpha_i}\), \(f \in \oplus_{i=1}^n \mathfrak{g}_{-\alpha_i}\), and \(h\in \mathfrak{h}\) be homogeneous elements. We let \(f\) act on \(V_0\) trivially and \(h\) act on \(V_0=W\) by the existing action of \(\mathfrak{h}\) on $W$. For \(a\in V\), we define inductively:
		\begin{align*}
			& e(a) = e\otimes a, \\
			& h(e\otimes a) = [h,e]\otimes a + (-1)^{\bar{h}\bar{e}} e\otimes h(a) \\ 
			& f(e\otimes a) = [f,e]\otimes a + (-1)^{\bar{f}\bar{e}} e\otimes f(a). 
		\end{align*}
		To show this defines an action of \(\tilde{\mathfrak{g}}\) on \(V\), we need to check it respects the defining relations imposed on \(\tilde{\mathfrak{g}}\), which is an immediate yet technical generalization of the proof of Theorem 1.2 in \cite{K}.  We demonstrate the general argument of the proof by showing, inductively, that \([h,f](a) = (hf-(-1)^{\bar{h}\bar{f}}fh)(a)\) for all \(a \in V\). Assuming this equality holds for \(a \in V_s\), then
		\begin{align*}
			& (hf-(-1)^{\bar{h}\bar{f}}fh)(e\otimes a) \\ 
			& = h[f,e](a) + (-1)^{\bar{f}\bar{e}}h(e\otimes f(a)) -(-1)^{\bar{h}\bar{f}}f([h,e]\otimes a) - (-1)^{\bar{h}(\bar{f}+\bar{e})}f(e\otimes h(a)) \\ & = h[f,e](a) + (-1)^{\bar{f}\bar{e}}[h,e]\otimes f(a) + (-1)^{(\bar{f}+\bar{h})\bar{e}}e\otimes hf(a) - (-1)^{\bar{h}\bar{f}} [f,[h,e]](a) \\ 
			& - (-1)^{\bar{e}\bar{f}}[h,e]\otimes f(a) - (-1)^{\bar{h}(\bar{f}+\bar{e})}[f,e]h(a) - (-1)^{\bar{e}(\bar{f}+\bar{h})+\bar{h}\bar{f}}e\otimes fh(a) \\
			& = [h,[f,e]](a) - (-1)^{\bar{h}\bar{f}}[f,[h,e]](a) + (-1)^{\bar{e}(\bar{f}+\bar{h})}e\otimes(hf-(-1)^{\bar{h}\bar{f}}fh)(a) \\ 
			& = [[h,f],e](a) + (-1)^{\bar{e}(\bar{f}+\bar{h})}e\otimes[h,f](a) = [h,f](e\otimes a).
		\end{align*}
		The equality \([h,f](a) = (hf-(-1)^{\bar{h}\bar{f}}fh)(a)\) for \(a \in V\) follows by induction. \par
		As the representation \(\pi_W\) is defined for every \(\mathfrak{h}\)-module \(W\), we obtain that \(\mathfrak{h}\) injects into \(\tilde{\mathfrak{g}}\). Moreover, the action of \(\bigoplus_{i=1}^n \mathfrak{g}_{\alpha_i}\) via \(\pi_W\) shows that  \(\bigoplus_{i=1}^n \mathfrak{g}_{\alpha_i}\) injects into \(\tilde{\mathfrak{g}}\). \par
		Let \(W_0\) denote the trivial \(\mathfrak{h}\)-module. The map \(\phi:\tilde{\mathfrak{n}}^+ \rightarrow \TT(\bigoplus_{i=1}^n \mathfrak{g}_{\alpha_i})\) given by \(n \mapsto \pi_{W_0}(n)(1)\) establishes \(\TT(\bigoplus_{i=1}^n \mathfrak{g}_{\alpha_i})\) as an enveloping algebra of \(\tilde{\mathfrak{n}}^+\), and is easily verified to be its universal enveloping algebra. In particular, \(\tilde{\mathfrak{n}}^+\) is freely generated by \(\bigoplus_{i=1}^n \mathfrak{g}_{\alpha_i}\). The result for \(\tilde{\mathfrak{n}}^-\) is analogous. This proves \ref{bullet1-half-baked}. \par
		We certainly have \(\tilde{\mathfrak{g}} = \tilde{\mathfrak{n}}^- + \mathfrak{h} + \tilde{\mathfrak{n}}^+\). Suppose \(n^- + h + n^+ = 0\) for some \(n^\pm \in \tilde{\mathfrak{n}}^{\pm}\) and \(h\in \mathfrak{h}\). Then \(0 = \pi_W(n^-+h+n^+)(a)\) for \(a \in V_0 = W\). For \(W = W_0\) we obtain \(\phi(n^+) = \pi_{W_0}(n^-+h+n^+)(1) = 0\), so \(n^+ = 0\). Now \(0 = \pi_W(n^-+h)(a) = \pi_W(h)(a)\) for all \(a \in V_0 = W\), which implies \(h = 0\). We conclude that \(n^- = 0\), which establishes \ref{bullet2-half-baked}. Part \ref{bullet3-half-baked} is an immediate consequence of \ref{bullet2-half-baked}. \par
		Finally, a standard linear algebraic argument gives that any ideal \(I\) of \(\tilde{\mathfrak{g}}\) decomposes as a \(\mathfrak{t}\)-module in the following way:
		\begin{align}\label{equation-decomposition-of-ideals}
			I = \bigoplus_{\substack{\alpha \in Q_+ \\ \alpha \neq 0}}(I\cap\tilde{\mathfrak{g}}_{-\alpha}) \oplus (I\cap \mathfrak{h}) \oplus \bigoplus_{\substack{\alpha \in Q_+ \\ \alpha \neq 0}}(I\cap\tilde{\mathfrak{g}}_{\alpha}) = (I\cap \tilde{\mathfrak{n}}^-) \oplus (I\cap \mathfrak{h}) \oplus (I\cap \tilde{\mathfrak{n}}^+).
		\end{align}
		Therefore, if \(I\) intersects \(\mathfrak{h}\) trivially then \(I = (I\cap \tilde{\mathfrak{n}}^-) \oplus (I\cap \tilde{\mathfrak{n}}^+) \). Summing over all such ideals of \(\tilde{\mathfrak{g}}\), we get a unique maximal \(\mathfrak{r}\) intersecting \(\mathfrak{h}\) trivially, which satisfies \(\mathfrak{r} = (\mathfrak{r}\cap\tilde{\mathfrak{n}}^-)\oplus(\mathfrak{r}\cap\tilde{\mathfrak{n}}^+)\), according to \eqref{equation-decomposition-of-ideals}. This proves part \ref{bullet4-half-baked}.
	\end{proof}
	
	\begin{remark}
		By the same argument, Theorem \ref{theorem-half-baked-structure} also holds in the case that \(\mathfrak{g}_{\pm \alpha_i}\) are generalized weight spaces of weight \(\pm \alpha_i\) as \(\mathfrak{t}\)-modules.
	\end{remark}

	The following lemma is a direct generalization of Lemma 1.5 in \cite{K}, with the same proof.
	
	\begin{lemma}\label{lemma-maximal-ideal}
		If \(x \in \tilde{\mathfrak{n}}^+\) is such that \([\mathfrak{g}_{-\alpha},x] \in \mathfrak{r}\) for all \(\alpha \in \Pi\), then \(x \in \mathfrak{r}\). Analogously, if \(y \in \tilde{\mathfrak{n}}^-\) satisfies \([\mathfrak{g}_{\alpha},y] \in \mathfrak{r}\) for all \(\alpha \in \Pi\), then \(y \in \mathfrak{r}\).
	\end{lemma}
		
	\subsection{Chevalley automorphism of \(\tilde{\mathfrak{g}}(\mathcal{A})\)}\label{subsection-Chevalley}
	
	Recall the automorphism $\omega_{\h}$ of $\h$ defined in Section \ref{section-irreducible-h-modules}.
	\begin{definition}
		Let $\AA$ be a Cartan datum; a Chevalley automorphism of $\tilde{\g}(\AA)$ is an extension $\tilde{\omega}$ of $\omega_{\h}$ to an automorphism of all of $\tilde{\g}(\AA)$.  
	\end{definition}

	In contrast with the usual Kac--Moody construction, the Lie superalgebra \(\tilde{\mathfrak{g}}(\mathcal{A})\) is not guaranteed to admit a Chevalley automorphism; in particular, a necessary condition is that $\g_{-\alpha}\cong\g_{\alpha}^\vee$ for all roots $\alpha$, where \(\mathfrak{g}_{\alpha}^\vee\) is the twist by the automorphism \(\omega_{\mathfrak{h}}^{-1}\) (see Section \ref{section-irreducible-h-modules}).  We now describe a condition which ensures the existence of such an automorphism.
	
	Suppose \(\mathcal{A}\) is a Cartan datum for which we have identifications \(\mathfrak{g}_{-\alpha}=\mathfrak{g}_{\alpha}^\vee\) for all \(\alpha \in \Pi\). For each $\alpha\in\Pi$, let \(\omega_{\alpha}:\mathfrak{g}_{\alpha}\rightarrow \mathfrak{g}_{\alpha}^\vee\) be the identity map of the underlying super vector spaces. We will also write \(x^\vee \coloneq \omega_{\alpha}(x)\) for \(x \in \mathfrak{g}_{\alpha}\). We set \(\omega_{-\alpha} \coloneq \delta_{\mathfrak{g}_{\alpha}}\circ\omega_{\alpha}^{-1}:\mathfrak{g}_{\alpha}^\vee \rightarrow \mathfrak{g}_{\alpha}\), where \(\delta_{\mathfrak{g}_{\alpha}}(v)=(-1)^{\ol{v}}v\) is the grading operator on \(\mathfrak{g}_{\alpha}\). It follows immediately that for \(\alpha\in \Pi\), \(x \in \mathfrak{g}_{\alpha}\), \(y \in \mathfrak{g}_{-\alpha}\) and \(h \in \mathfrak{h}\) we have
	\begin{align*}
		\begin{array}{cc}
			\omega_{\alpha}([h,x]) = [\omega_{\mathfrak{h}}(h),\omega_{\alpha}(x)], & \omega_{-\alpha}([h,y]) = [\omega_{\mathfrak{h}}(h),\omega_{-\alpha}(y)].
		\end{array}
	\end{align*}
	\begin{theorem}\label{theorem-half-baked-Chevalley}
		Let \(\mathcal{A}\) be a Cartan datum satisfying \(\mathfrak{g}_{-\alpha}=\mathfrak{g}_{\alpha}^\vee\) for all \(\alpha \in \Pi\). Assume \(\omega_{\mathfrak{h}}([x,y]) = [\omega_{\alpha}(x),\omega_{-\alpha}(y)]\) for all \(x\in \mathfrak{g}_{\alpha}\) and \(y \in \mathfrak{g}_{-\alpha}\), for all \(\alpha \in \Pi\). Then there exists a Chevalley automorphism \(\tilde{\omega}\) of \(\tilde{\mathfrak{g}}(\mathcal{A})\) of order \(4\), which preserves \(\mathfrak{r}\) and satisfies \(\tilde{\omega}|_{\mathfrak{h}} = \omega_{\mathfrak{h}}\) and \(\tilde{\omega}|_{\mathfrak{g}_{\alpha}} = \omega_{\alpha}\) for \(\alpha \in \pm \Pi\).
	\end{theorem}
	
	\begin{proof}
		Let \(\omega_{\mathcal{C}} \coloneq \bigoplus_{i=1}^n \omega_{-\alpha_i} \oplus \omega_{\mathfrak{h}} \oplus \bigoplus_{i=1}^n \omega_{\alpha_i}\) be the endomorphism of the super vector space \(\mathcal{C}\) defined in \eqref{generating-set}. It is immediate that \(\omega_{\mathcal{C}}\) is an automorphism and that \(\omega_{\mathcal{C}}^2 = \delta_{\mathcal{C}}\), where \(\delta_{\mathcal{C}}\) is the grading operator of \(\mathcal{C}\). We wish to extend \(\omega_{\mathcal{C}}\) to a Lie superalgebra endomorphism \(\tilde{\omega}\) of \(\tilde{\mathfrak{g}}(\mathcal{A})\). By assumption we know that \(\omega_{\mathcal{C}}([x,y]) = [\omega_{\mathcal{C}}(x),\omega_{\mathcal{C}}(y)]\) whenever \(x,y,[x,y] \in \mathcal{C}\). Moreover, because \(\omega_\mathcal{C}\) is an even linear map, we have
		\begin{align*}
			[\omega_{\mathcal{C}}(x),\omega_{\mathcal{C}}(y)] + (-1)^{\bar{x}\bar{y}}[\omega_{\mathcal{C}}(y),\omega_{\mathcal{C}}(x)] = 0
		\end{align*} 
		and
		\begin{align*}
			[[\omega_{\mathcal{C}}(x),\omega_{\mathcal{C}}(y)],\omega_{\mathcal{C}}(z)] = [\omega_{\mathcal{C}}(x),[\omega_{\mathcal{C}}(y),\omega_{\mathcal{C}}(z)]] + (-1)^{\bar{x}\bar{y}} [\omega_{\mathcal{C}}(y),[\omega_{\mathcal{C}}(x),\omega_{\mathcal{C}}(z)]]
		\end{align*}
		for homogeneous \(x,y,z \in \mathcal{C}\). It follows that we may extend \(\tilde{\omega}\) to $\tilde{\g}(\AA)$ by requiring that \(\tilde{\omega}([x,y]) = [\tilde{\omega}(x),\tilde{\omega}(y)]\) for \(x,y \in \tilde{\mathfrak{g}}(\mathcal{A})\) and \(\tilde{\omega}|_{\mathcal{C}} \coloneq \omega_{\mathcal{C}}\). From the definition of \(\omega_{\mathcal{C}}\) follows that \(\tilde{\omega}^2 = \delta_{\tilde{\mathfrak{g}}(\mathcal{A})}\), where \(\delta_{\tilde{\mathfrak{g}}(\mathcal{A})}\) is the grading operator of \(\tilde{\mathfrak{g}}(\mathcal{A})\). In particular, we obtain that \(\tilde{\omega}\) is an automorphism of order \(4\). It is immediate that \(\tilde{\omega}(\tilde{\mathfrak{g}}_{\alpha}) = \tilde{\mathfrak{g}}_{-\alpha}\) for any \(\alpha \in \pm Q_+\). This implies \(\tilde{\omega}(\mathfrak{r})\cap \mathfrak{h} = 0\), so \(\tilde{\omega}(\mathfrak{r}) \subseteq \mathfrak{r}\). A similar argument gives \(\tilde{\omega}^{-1}(\mathfrak{r}) \subseteq \mathfrak{r}\). Altogether, we obtain \(\tilde{\omega}(\mathfrak{r}) = \mathfrak{r}\).
	\end{proof}
	
	\begin{proposition}\label{proposition-rank-2-chevalley}
		Let \(\mathcal{A}\) be a Cartan datum such that \(\mathfrak{g}_{-\alpha} = \mathfrak{g}_{\alpha}^\vee\) for all \(\alpha \in \Pi\). Assume, moreover, that \(\mathfrak{g}_{\alpha}\) is irreducible and that \(\rk\alpha \leq 2\) for all \(\alpha \in \Pi\). Then the conditions of Theorem \ref{theorem-half-baked-Chevalley} hold, and thus $\tilde{\g}(\AA)$ admits a Chevalley automorphism.
	\end{proposition}
	
	\begin{proof}
		Fix a simple root \(\alpha \in \Pi\). If \(\rk\alpha = 0\), then \(\mathfrak{g}_{\alpha}\) is \(1\)-dimensional with basis \(\{v\}\). Then \(\{v^\vee\}\) is a basis for \(\mathfrak{g}_{-\alpha}\) (where \(v^\vee = \omega_{\alpha}(v)\)) and \([v,v^\vee] \in \mathfrak{t}\). Therefore
		\begin{align*}
			\omega_{\mathfrak{h}}([v,v^\vee]) = -[v,v^\vee] = - [(-1)^{\bar{v}}\omega_{-\alpha}(v^\vee),\omega_{\alpha}(v)] = [\omega_{\alpha}(v),\omega_{-\alpha}(v^\vee)].
		\end{align*}
		If \(\rk\alpha=1\) or $2$, then \(\mathfrak{g}_{\alpha}\) is \((1|1)\)-dimensional. Let \(\{v,w\}\) be a homogeneous basis for \(\mathfrak{g}_{\alpha}\). From the irreducibility of \(\mathfrak{g}_{\alpha}\), there exists \(H \in \mathfrak{h}_{\ol{1}}\) such that \(H\cdot v = w\). As \([v,v^\vee] \in \mathfrak{t}\), we have \([\mathfrak{h},[v,v^\vee]] = 0\). So
		\begin{align*}
			0 = [H,[v,v^\vee]] = [H\cdot v, v^\vee] + (-1)^{\bar{v}} [v,H\cdot v^\vee] = [w,v^\vee] - \sqrt{-1}\cdot (-1)^{\bar{v}}[v,w^\vee].
		\end{align*}
		As \([w,v^\vee] \in \mathfrak{h}_{\ol{1}}\), we must have
		\begin{align*}
			\omega_{\mathfrak{h}}([w,v^\vee]) = \sqrt{-1}[w,v^\vee] = -(-1)^{\bar{v}}[v,w^\vee] = - [\omega_{-\alpha}(v^\vee),\omega_{\alpha}(w)] = [\omega_{\alpha}(w),\omega_{-\alpha}(v^\vee)].
		\end{align*}
		A similar computation gives \(\omega_{\mathfrak{h}}([v,w^\vee]) = [\omega_{\alpha}(v),\omega_{-\alpha}(w^\vee)]\). Finally, the computation in the \(\rk\alpha = 0\) case gives \(\omega_{\mathfrak{h}}([v,v^\vee]) = [\omega_{\alpha}(v),\omega_{-\alpha}(v^\vee)] \)  
		and \(\omega_{\mathfrak{h}}([w,w^\vee]) = [\omega_{\alpha}(w),\omega_{-\alpha}(w^\vee)]\). The claim follows by linearity. 
	\end{proof}
	
	\subsection{The Lie superalgebra \(\mathfrak{g}(\mathcal{A})\)}\label{section-g(A)}
	
	We define \(\mathfrak{g}(\mathcal{A}) \coloneq \tilde{\mathfrak{g}}(\mathcal{A})/\mathfrak{r}\), the quotient of \(\tilde{\mathfrak{g}}(\mathcal{A})\) by the maximal ideal intersecting \(\mathfrak{h}\) trivially, as in Theorem \ref{theorem-half-baked-structure}. We say \(\mathfrak{g}(\mathcal{A})\) is the Lie superalgebra associated to the Cartan datum \(\mathcal{A}\). We also write \(\mathfrak{g}\) instead of \(\mathfrak{g}(\mathcal{A})\) when clear from context. 
	
	\begin{lemma}
		The natural map 
		\[
		\mathcal{C} = \bigoplus_{i=1}^n \mathfrak{g}_{-\alpha_i} \oplus \mathfrak{h} \oplus \bigoplus_{i=1}^n \mathfrak{g}_{\alpha_i}\to\g(\AA)
		\]
		is an embedding of super vector spaces.  Thus we may identify $\h$ and $\g_{\alpha}$ for $\alpha\in\pm\Pi$ with their images in $\g(\AA)$.
	\end{lemma}

	\begin{proof}
		This immediately follows from the definition of $\r$ and the assumption of nondegeneracy on our maps $\g_{\alpha}\otimes\g_{-\alpha}\to\h$.
\end{proof}
	
	Since $\mathfrak{t}$ acts on $\g$ semisimply, we have a weight space decomposition $\g=\bigoplus\limits_{\alpha\in\mathfrak{t}^*}\g_{\alpha}$.  We call \(\alpha \in \mathfrak{t}^*\setminus\{0\}\) a root of \(\mathfrak{g}\) if \(\mathfrak{g}_{\alpha} \neq 0\). We denote by \(\Delta\) the set of roots of \(\mathfrak{g}\), and set \(\Delta_+ \coloneq \Delta \cap Q_+\), where $Q_+=\sum\limits_{i=1}^n\mathbb{N}\alpha_i\subseteq\mathfrak{t}^*$. If \(\alpha = \sum_{i=1}^{s} k_i\alpha_i\) is a root for some \(k_i \in \mathbb{Z}_{\geq 0}\), we say \(\alpha\) is of height \(\sum_{i=1}^sk_i\). 
	
	We denote by \(\mathfrak{n}^+\) and \(\mathfrak{n}^-\) the images of \(\tilde{\mathfrak{n}}^+\) and \(\tilde{\mathfrak{n}}^-\) in \(\mathfrak{g}\), respectively. Theorem \ref{theorem-half-baked-structure} implies the triangular decomposition \(\mathfrak{g} = \mathfrak{n}^-\oplus \mathfrak{h} \oplus \mathfrak{n}^+\) as a super vector space and the decomposition
	\begin{align*}
		\mathfrak{g} = \bigoplus_{\alpha\in\Delta_+} \mathfrak{g}_{-\alpha} \oplus \mathfrak{h} \oplus \bigoplus_{\alpha \in \Delta_+} \mathfrak{g}_{\alpha}
	\end{align*}
	as a \(\mathfrak{t}\)-module. We notice that \(\mathfrak{g}_{\alpha}\) is an \(\mathfrak{h}\)-module for any \(\alpha \in Q_+\), so the decomposition also holds as an \(\mathfrak{h}\)-module.  Further, $\h$ is the centralizer of $\mathfrak{t}$ in $\g(\AA)$, and is self-normalizing.
	
	\begin{definition}\label{definition_coroot}
		For a root \(\alpha \in \Delta\), we call \(\h_{\alpha}:=[\mathfrak{g}_{\alpha},\mathfrak{g}_{-\alpha}]\) the \emph{coroot-space} corresponding to \(\alpha\); it is an ideal of \(\mathfrak{h}\). An element of $\h_{\alpha}$ shall be called an \emph{\(\alpha\)-coroot}, or a coroot corresponding to \(\alpha\). We say an \(\alpha\)-coroot \(h\) is \emph{pure} if \(h = [x,y]\) for homogeneous non-zero \(x \in \mathfrak{g}_{\alpha}\) and \(y \in \mathfrak{g}_{-\alpha}\).
	\end{definition}

	As a generalization of Lemma 1.6 in \cite{K} we have:

	\begin{proposition}\label{proposition-roots-connectivity}
	Let \(\Pi_1,\Pi_2 \subseteq \Pi\) be disjoint subsets such that
	\begin{align*}
		[\h_{\alpha},\mathfrak{g}_{\beta}] = [\h_{\beta},\mathfrak{g}_{\alpha}] = 0
	\end{align*} 
	for any \(\alpha \in \Pi_1\), \(\beta \in \Pi_2\). Let \(Q^+_s = \mathbb{Z}_{\geq 0}\Pi_s\) for \(s=1,2\). If \(\theta \in \Delta\) is a root of \(\mathfrak{g}\) that satisfies \(\theta \in Q^+_1 + Q^+_2\), then necessarily \(\theta \in Q^+_1\) or \(\theta \in Q^+_2\).
\end{proposition}

\begin{proof}
	Let us show that \([\mathfrak{g}_{\alpha},\mathfrak{g}_{\beta}] = 0\) for \(\alpha \in \Pi_1\), \(\beta \in \Pi_2\). The Jacobi identity implies
	\begin{align*}
		[\mathfrak{g}_{-\alpha},[\mathfrak{g}_{\alpha},\mathfrak{g}_{\beta}]] \subseteq [[\mathfrak{g}_{-\alpha},\mathfrak{g}_{\alpha}],\mathfrak{g}_{\beta}] + [\mathfrak{g}_{\alpha},[\mathfrak{g}_{-\alpha},\mathfrak{g}_{\beta}]].
	\end{align*}
	But \([[\mathfrak{g}_{-\alpha},\mathfrak{g}_{\alpha}],\mathfrak{g}_{\beta}] = 0\) by assumption and \([\mathfrak{g}_{-\alpha},\mathfrak{g}_{\beta}]=0\) by the defining relations of \(\mathfrak{g}\). So \([\mathfrak{g}_{-\alpha},[\mathfrak{g}_{\alpha},\mathfrak{g}_{\beta}]] = 0\). A similar argument shows \([\mathfrak{g}_{-\beta},[\mathfrak{g}_{\alpha},\mathfrak{g}_{\beta}]] = 0\). We certainly have \([\mathfrak{g}_{-\gamma},[\mathfrak{g}_{\alpha},\mathfrak{g}_{\beta}]] = 0\) for any \(\gamma \in \Pi \setminus \{\alpha,\beta\}\). As \([\mathfrak{g}_{\alpha},\mathfrak{g}_{\beta}] \subseteq \mathfrak{n}^+\) and \([\mathfrak{g}_{-\gamma},[\mathfrak{g}_{\alpha},\mathfrak{g}_{\beta}]]=0\) for any \(\gamma \in \Pi\), Corollary \ref{corollary-maximal-ideal} implies \([\mathfrak{g}_{\alpha},\mathfrak{g}_{\beta}] = 0\). \par
	Let \(\mathfrak{g}^{(s)}\) be the subalgebra of \(\mathfrak{g}\) generated by \(\mathfrak{g}_{\alpha}\) and \(\mathfrak{g}_{-\alpha}\) for \(\alpha \in \Pi_{s}\). Then what we have shown so far implies \([\mathfrak{g}^{(1)},\mathfrak{g}^{(2)}] = 0\). Because \(\theta \in Q^+_1 + Q^+_2\) we must have that \(\mathfrak{g}_{\theta}\) is contained in the algebra generated by \(\mathfrak{g}^{(1)}\) and \(\mathfrak{g}^{(2)}\). But then \(\mathfrak{g}_{\theta}\) is contained in either \(\mathfrak{g}^{(1)}\) or \(\mathfrak{g}^{(2)}\), which means \(\theta \in Q^+_1\) or \(\theta \in Q^+_2\).
\end{proof}

	We now state several obvious corollaries:
	
	\begin{corollary}\label{corollary-center}
		Let \(\mathfrak{c}\) be the center of \(\mathfrak{g}\). Then \(\mathfrak{c} \subseteq \bigcap_{i=1}^n (\mathrm{Ann}_{\mathfrak{h}} \mathfrak{g}_{\alpha_i}\cap \mathrm{Ann}_{\mathfrak{h}} \mathfrak{g}_{-\alpha_i})\) and \newline \(\mathfrak{c}_{\ol{0}} = \bigcap_{i=1}^n \ker \alpha_i \subseteq \mathfrak{t}\).
	\end{corollary}
	
	Lemma \ref{lemma-maximal-ideal} immediately implies the following corollary:
	\begin{corollary}\label{corollary-maximal-ideal}
		If \(x \in \mathfrak{n}^+\) satisfies \([\mathfrak{g}_{-\alpha},x] = 0\) for all \(\alpha \in \Pi\), then \(x = 0\). Similarly, if \(y \in \mathfrak{n}^-\) is such that \([\mathfrak{g}_{\alpha},y] = 0\) for all \(\alpha \in \Pi\), then \(y=0\). 
	\end{corollary}
	
	Theorem \ref{theorem-half-baked-Chevalley} and Proposition \ref{proposition-rank-2-chevalley} give the following corollary:
	
	\begin{corollary}\label{corollary-Chevalley}
		Let \(\mathcal{A}\) be a Cartan datum satisfying \(\mathfrak{g}_{-\alpha} = \mathfrak{g}_{\alpha}^\vee\) for all \(\alpha \in \Pi\). Assume
		\begin{align}\label{chevalley-equation}
			\omega_{\mathfrak{h}}([x,y]) = [\omega_{\alpha}(x),\omega_{-\alpha}(y)]
		\end{align} 
		for all \(x \in \mathfrak{g}_{\alpha}\) and \(y \in \mathfrak{g}_{-\alpha}\), for all \(\alpha \in \Pi\) (see Section \ref{subsection-Chevalley} for notation). Then \(\mathfrak{g}(\mathcal{A})\) admits a Chevalley automorphism $\omega$ with $\omega^2=\delta$, where $\delta(v)=(-1)^{\ol{v}}v$.
		
		In particular, if the \(\mathfrak{h}\)-modules \(\mathfrak{g}_{\alpha}\) are irreducible and \(\rk\alpha \leq 2\) for all \(\alpha \in \Pi\), then \eqref{chevalley-equation} holds.
	\end{corollary}

	\subsection{The Ramond superalgebra and the twisted $d=2$, $\NN=2$ superconformal algebra}\label{section ramond} We give here two examples with connections to physics. Let $\AA_{R}$ be the following Cartan datum: let $\h$ be the $(2|1)$ dimensional super vector space with odd generator $G_0$ and even generators $L_0,c$, with $[G_0,G_0]=2L_0-c/12$.  Let $\alpha$ be the weight with $\alpha(L_0)=-1$, $\alpha(c)=0$.  Let $\g_{\alpha}$ have basis $L_1,G_1$, and $\g_{-\alpha}$ basis $L_{-1},G_{-1}$, where $L_{\pm1}$ are even, $G_{\pm1}$ are odd, and we have actions
	\[
	[G_0,L_{\pm1}]=\mp G_{\pm1}/2, \ \ \ [G_0,G_{\pm1}]=2L_{\pm1}.
	\]
	Finally, set $[L_1,L_{-1}]=2L_0$, $[G_{1},G_{-1}]=2L_0+c/4$, and $[L_1,G_{-1}]=3G_0/2=[G_1,L_{-1}]$
	
	Next, define $\AA_{\NN=2}$ to be the Cartan datum with the same $\h$,  but now we take for simple root $\alpha/2$, so that $(\alpha/2)(L_0)=-1/2$, and $(\alpha/2)(c)=0$.  Then let $\g_{\alpha/2}$ have basis $J_{1/2}$, $H_{1/2}$, and $\g_{-\alpha/2}$ basis $J_{-1/2}$, $H_{-1/2}$, where $J_{\pm1/2}$ are even and $H_{\pm1/2}$ are odd.  We set the actions to be:
	\[
	[G_0,J_{\pm1/2}]=H_{\pm 1/2}, \ \ \ [G_0,H_{\pm1/2}]=\mp J_{\pm1/2}/2.
	\]
	Finally, set 
	\[
	[J_{1/2},J_{-1/2}]=c/6, \ \ \ [J_{1/2},H_{-1/2}]=-2G_0, \ \ \ [H_{1/2},H_{-1/2}]=2L_0.
	\]
	
	\begin{proposition}
		\begin{enumerate}
			\item $\g(\AA_R)$ is isomorphic to the Ramond algebra, i.e.~the $d=2$, $\NN=1$ superconformal algebra in the Ramond sector.  
			\item $\g(\AA_{\NN=2})$ is isomorphic to the twisted $d=2$, $\NN=2$ superconformal algebra.
		\end{enumerate}

	\end{proposition}
	\begin{proof}
		This follows from a straightforward comparison of relations, see for instance \cite{SS}.
	\end{proof}

	\subsection{Cartan subdata and morphisms}\label{section_subdatum}
	
	Let $\h'\subseteq\h$ be any Lie subalgebra of $\h$, and for each $\alpha\in\pm\Pi$ choose $\h'$-submodules $\g_{\alpha}'\subseteq\g_{\alpha}$ such that $[\g_{\alpha}',\g_{-\alpha}']\subseteq\h'$.  
	
	Then we obtain what we call a \emph{Cartan subdatum} $\AA'$ of $\AA$, which by definition consists of the above information, and constitutes a Cartan datum of its own apart from possibly failing the condition that the map $\g_{\alpha}'\otimes\g_{-\alpha}'\to\h'$ is nondegenerate.  Nevertheless, we may define $\tilde{\g}(\AA')$ in the natural way, and we will have an obvious map $\tilde{\g}(\AA')\to\g(\AA)$.
	
	There are two distinguished types of Cartan subdata.  Let $\Pi'\subseteq\Pi$ be a subset of simple roots.  Then we obtain a natural Cartan subdatum $\AA':=\AA_{\Pi'}$ of $\AA$ with $\h'=\h$, and $\g_{\alpha}'=\g_{\alpha}$ if $\pm\alpha\in\Pi'$, and $\g_{\alpha}'=0$ otherwise.  In this case $\AA'$ will be a Cartan datum in its own right.  We call this a \emph{full} Cartan subdatum.
	
	The other important case is when $\h'=\mathfrak{t}$, and we choose vectors $e_1,\dots,e_n,f_1,\dots,f_n$ with $e_i\in\g_{\alpha_i}$, $f_i\in\g_{-\alpha_i}$, $e_i,f_i$ of the same parity.  We call a Cartan subdatum of this form \emph{classical}. Setting $h_i:=[e_i,f_i]$, we obtain in this way a Cartan matrix $A'=(\alpha_j(h_i))$, and with it a natural map $\tilde{\g}'(A')\to\g(\AA)$, where $\tilde{\g}'(A')$ is the derived subalgebra of $\tilde{\g}(A')$.
	
	The following result will be of use later on.  
	\begin{corollary}
		If there exists a classical Cartan subdatum of $\AA$ with Cartan matrix $A'$ for which $\g(A')$ is not of finite growth, then $\g(\AA)$ is also not of finite growth.
	\end{corollary}
	
	\begin{proof}
		By the setup described above, we have a map $\tilde{\g}(A')\to\g(\AA)$.  The image of this map admits as a subquotient the Lie superalgebra $\g'(A')$ quotient by its center, where $\g'(A')$ is the derived subalgebra of $\g(A')$.  Since this is of finite growth if and only if $\g(A')$ is, we are done.
	\end{proof}
	
	\section{Clifford Kac--Moody algebras}
	
	\begin{definition}\label{definition-integrable}
		We say that a Cartan datum $\AA$ (and by extension, the Lie superalgebra \(\mathfrak{g}(\AA)\)) is integrable if for any \(\alpha,\beta \in \pm \Pi\) there exists \(n \in \mathbb{N}\) such that \((\mathrm{ad}\ \mathfrak{g}_{\alpha})^n \mathfrak{g}_{\beta} = 0\).
	\end{definition}
	
	\begin{definition}
		We call \(\mathfrak{g}(\mathcal{A})\) Clifford Kac--Moody if it satisfies the following properties:
		\begin{enumerate}
			\item The \(\mathfrak{h}\)-module \(\mathfrak{g}_{\alpha}\) is irreducible for any \(\alpha \in \pm\Pi\); and
			\item $\AA$ is integrable.
		\end{enumerate}
		In this situation, we will also say that the Cartan datum $\AA$ is Clifford Kac--Moody.
	\end{definition}
	
	\begin{remark}\label{remark-almostKM-subdatum}
		For any full Cartan subdatum $\AA'$ of $\AA$ (see Section \ref{section_subdatum}), $\AA'$ is Clifford Kac--Moody whenever $\AA$ is.
	\end{remark}
	
		

	\begin{remark}
		For any \(\alpha \in \Pi\), choose a non-zero pure even coroot \(h_{\alpha}\in\h_{\alpha}\) and set \(A \coloneq (\beta(h_{\alpha}))_{\alpha,\beta\in\Pi}\). Then the integrability of \(\mathfrak{g}(\mathcal{A})\) implies that if \(A\) is elemental as in Sec.~2 of \cite{CS}, then it satisfies the condition of Lem.~3.1 in \cite{CS}. 
	\end{remark}
	
	\subsection{Examples of Clifford Kac--Moody algebras}
	\begin{example}
		Let $\g=\q(n)$ denote the Lie subalgebra of $\g\l(n|n)$ consisting of matrices of the form
		\[
		\begin{bmatrix}
			A & B\\ B &A
		\end{bmatrix},
		\]
		where $A$ and $B$ are arbitrary.  Set $\h\subseteq\q(n)$ to be the subalgebra of matrices of the form
		\[
		\begin{bmatrix} D & D'\\ D' & D\end{bmatrix},
		\]
		where $D,D'$ are diagonal.  Then $\h$ is quasi-toral and self-normalizing in $\g$.  Let $\mathfrak{t}=\h_{\ol{0}}$, the subalgebra of block diagonal matrices, and let $\epsilon_1,\dots,\epsilon_n\in\mathfrak{t}^*$ be the coordinate projections.  Then if we set $\alpha_i:=\epsilon_i-\epsilon_{i+1}$, we obtain a Cartan datum from $\h$ and $\g_{\pm\alpha_1},\dots,\g_{\pm\alpha_{n-1}}$, and all roots are of rank 2.  The root spaces $\g_{\alpha_i}$ are irreducible, and since integrability is easily checked, we see that $\g$ is Clifford Kac--Moody.  
		
		Write $e_i:=e_{ii}+e_{n+i,n+i}$ and $E_{i}:=e_{i,n+i}+e_{n+i,i}$.  Then the pure coroots in $\h_{\alpha_i}$ are given by $h_i:=e_i-e_{i+1}$, $c_i:=e_i+e_{i+1},$ and $H_i:=E_{i}-E_{i+1}$.  It is interesting to note that the pure coroots are distinct for different simple roots.
		
		One should view $\q(n)$ as the primary, motivating example of a Clifford Kac--Moody algebra, and of our construction at large.  
		
	\end{example}
	
	\begin{example}\label{example_sqn}
		Let $\g=\s\q(n)=[\q(n),\q(n)]$.  Then $\g$ has Cartan subalgebra given by $\h$ the set of matrices of the form
		\[
		\begin{bmatrix}
			D & D'\\ D' & D
		\end{bmatrix},
		\] 
		where $\operatorname{tr}(D')=0$.  The rest of the Cartan datum remains unchanged from that of $\q(n)$, and $\s\q(n)$ is Clifford Kac--Moody whenever $n\geq 3$ (notice that $\g_{\alpha_1}$ is not irreducible in $\s\q(2)$).  
	\end{example}
	
	\begin{example}\label{example_psq(n)}
		Let $\g=\p\s\q(n):=\s\q(n)/\C I_{n|n}$.  Similarly to the previous example, one can verify that $\g$ is Clifford Kac--Moody if $n\geq3$.
	\end{example}

	\begin{example}[Takiff Superalgebras]\label{example-takiff}
		An interesting class of examples of Clifford Kac--Moody algebras is obtained via extensions of Takiff superalgebras, which can be described as follows (see also Example 5.1 in \cite{S2}).  Let \(\mathfrak{s}\) be a symmetrizable integrable (as in Sec.~3 of \cite{CS}) Kac--Moody superalgebra with invariant form \((-,-)\).  We define 
		\[
		T\mathfrak{s} \coloneq \mathfrak{s}\otimes\mathbb{C}[\xi]/(\xi^2)\oplus \mathbb{C}\langle\partial_\xi,c,z\rangle,
		\]
		where \(\xi\) and \(\partial_\xi\) are odd, \(c\) and \(z\) even and central, and we have the following bracket (note that there are signs missing in the formula for the bracket in \cite{S2}):
		\begin{align*}
			&[x\otimes p+ a\partial_\xi  , y\otimes q+ b\partial_\xi] = \\ & (-1)^{\ol{p}\ol{y}}[x,y]\otimes pq+(-1)^{\ol{y}}ay\otimes q'+(-1)^{\ol{p}}bx\otimes p'+(-1)^{\ol{y}}(x,y)Res(p'q)c + abz 
		\end{align*}		
		for \(x,y \in \mathfrak{s}\), \(a,b \in \mathbb{C}\), and $p,q\in\C[\xi]$.	Here we write e.g.~$p':=\d_{\xi}p$, and $Res(a+b\xi)=b$.
		
		To check that \(\g:=T\mathfrak{s}\) is Clifford Kac--Moody, let $\ol{\mathfrak{t}}\subseteq\s$ be a maximal torus of $\s$, and let $\h=\ol{\mathfrak{t}}\otimes\C[\xi]\oplus\C\langle\d_{\xi},c,z\rangle$.  Then $\h$ is quasi-toral and self-normalizing in $T\s$.  Write $\mathfrak{t}:=\h_{\ol{0}}$.  If $\ol{\Pi}=\{\ol{\alpha_1},\dots,\ol{\alpha_n}\}\subseteq(\ol{\mathfrak{t}})^{*}$ denotes the simple roots of $\s$, let $\alpha_i\in\mathfrak{t}^*$ denote the weight satisfying $\alpha_i(t)=\ol{\alpha_i}(t)$ for $t\in\ol{\mathfrak{t}}$, and $\alpha_i(c)=\alpha_i(z)=0$.  Then $\Pi=\{\alpha_1,\dots,\alpha_n\}\subseteq\mathfrak{t}^*$ is linearly independent, and $\h$ along with the root spaces $\g_{\pm\alpha_i}$ give rise to a Cartan datum.  
		
		The root spaces $\g_{\alpha_i}$ are spanned by $e_{i}\otimes1,e_i\otimes\xi$ where $e_i$ is the Chevalley generator of $\s$ for the root $\ol{\alpha_i}$, and it is easy to see this forms an irreducible $\h$-module (with the help of $\d_{\xi}$).  Finally, $\g$ will be integrable because the same is true of $\s$.  
		
		Observe that every simple root of $T\s$ is of rank 2.  The $\alpha$ coroot-space $\h_{\alpha}$ for a simple root $\alpha$ is given by $h_{\ol{\alpha}}\otimes \C[\xi]\oplus\C\langle c\rangle$, where $h_{\ol{\alpha}}$ denotes the coroot in $\s$ of the simple root $\ol{\alpha}$.  Thus the pure coroots are $h_{\ol{\alpha}}\otimes 1$, $h_{\ol{\alpha}}\otimes\xi$, and $c$; in particular every coroot space shares the pure coroot $c$, and this coroot is central.
		
		As a special case of the above construction, we have that 
		\[
		T\s\l(2)/(c-z)\cong\q(2).
		\]
		
		We will see that in some ways the property of two coroot spaces sharing a pure, central coroot is characteristic of the superalgebras $T\s$.  This construction represents the only known `general' construction of nontrivial Clifford Kac--Moody algebras, i.e.~ones with simple roots of rank bigger than 0.      
	\end{example}
	
	\begin{example}\label{example_takiff_non_symm}
		In the above example, we can drop the condition that $\mathfrak{s}$ be symmetrizable at the cost of removing the central extension $c$.  We will still obtain a Clifford Kac--Moody algebra in this way, and we will still call this superalgebra $T\s$.
	\end{example}
	
	\begin{example}\label{example_ts_pts}
		For examples of superalgebras that are not Clifford Kac--Moody, let $\s$ be a symmetrizable Kac--Moody Lie superalgebra, and write $\mathfrak{t}\s$ for the subalgebra of $T\s$ spanned by $\s\otimes\C[\xi]$ and $c$.  It $\ol{\mathfrak{t}}$ is a maximal torus of $\s$, then we set $\h=\ol{\mathfrak{t}}\otimes\C[\xi]\oplus\C\langle c\rangle$ to obtain a self-normalizing quasi-toral subalgebra.  However in this case the root spaces of simple roots, spanned by $e_i\otimes 1$ and $e_i\otimes\xi$ for a Chevalley generator $e_i$ of $\s$, will not be irreducible $\h$-modules because we no longer have the derivation $\d_{\xi}$.  Thus $\mathfrak{t}\s$ is not Clifford Kac--Moody (however it does arise naturally from our more general construction, as $\mathfrak{t}\s$ contains no nontrivial ideals that intersect $\h$ trivially).  
		
		Another example is obtained by considering $\p\mathfrak{t}\s:=\mathfrak{t}\s/(c)\cong\s\otimes\C[\xi]$; and indeed if $\mathfrak{s}$ is not symmetrizable then this is the natural superalgebra to consider.  This case is again clearly not Clifford Kac--Moody, but once again arises naturally from our construction.
		
		Finally, we note that 
		\[
		\mathfrak{t}\s\l(2)\cong\s\q(2) \ \text{ and } \ \mathfrak{pt}\s\l(2)\cong\p\s\q(2).
		\]
	\end{example}

	\section{Cartan datum with one simple root}\label{section-one-root}
	
	In light of Section \ref{section_subdatum} and Remark \ref{remark-almostKM-subdatum}, it is wise to begin our study with Clifford Kac--Moody algebras \(\mathfrak{g}(\mathcal{A})\) having one simple root. Thus let \(\mathcal{A}\) be a Cartan datum with \(\Pi = \{\alpha\}\) and irreducible \(\mathfrak{h}\)-modules \(\mathfrak{g}_{\alpha}\) and \(\mathfrak{g}_{-\alpha}\). The integrability condition on \(\mathfrak{g}(\mathcal{A})\) means that \((\mathrm{ad}\ \mathfrak{g}_{\alpha})^n \mathfrak{g}_{\alpha} = 0\) and \((\mathrm{ad}\ \mathfrak{g}_{-\alpha})^n \mathfrak{g}_{-\alpha} = 0\) for some \(n \in \mathbb{N}\).
	
	For a simple root $\alpha$ of any Cartan datum, write 
	\[
	\g\langle\alpha\rangle:=\h_{\alpha}\oplus\bigoplus\limits_{n\in\Z_{\neq0}}\g_{n\alpha}.
	\]
	Said otherwise, $\g\langle\alpha\rangle$ is the subalgebra generated by $\g_{\alpha}$ and $\g_{-\alpha}$.
	
	\subsection{Roots of Heisenberg type}\label{section_heisenberg}
	
	\begin{definition}
		We say that a simple root $\alpha$ is of Heisenberg type if any of the following equivalent conditions hold:
		\begin{enumerate}
			\item $[\h_{\alpha},\mathfrak{g}_{\alpha}]=0$;
			\item $\h_{\alpha}\subseteq\operatorname{Ann}_{\h}\g_{\alpha}$;
			\item $\h_{\alpha}$ is central in $\g\langle\alpha\rangle$.
		\end{enumerate}
	\end{definition}
	The one part of the above definition that may not be clear is why $\h_{\alpha}$ need be abelian; however if there exists $H_1,H_2\in(\h_{\alpha})_{\ol{1}}$, then $[H_1,H_2]\in[H_1,[\g_{\alpha},\g_{-\alpha}]]=0$.  
	
	\begin{lemma}
		If $\alpha$ is of Heisenberg type, then \(\Delta = \{\pm \alpha\}\).
	\end{lemma}
	\begin{proof}
		The Jacobi identity implies
		\begin{align*}
			[\mathfrak{g}_{-\alpha},[\mathfrak{g}_{\alpha},\mathfrak{g}_{\alpha}]] \subseteq [\mathfrak{g}_{\alpha},[\mathfrak{g}_{\alpha},\mathfrak{g}_{-\alpha}]].
		\end{align*}
		Hence our assumption \([\mathfrak{g}_{\alpha},[\mathfrak{g}_{\alpha},\mathfrak{g}_{-\alpha}]] = 0\) yields \([\mathfrak{g}_{-\alpha},[\mathfrak{g}_{\alpha},\mathfrak{g}_{\alpha}]] = 0\). From Corollary \ref{corollary-maximal-ideal} we obtain \([\mathfrak{g}_{\alpha},\mathfrak{g}_{\alpha}] = 0\). A similar argument shows \([\mathfrak{g}_{-\alpha},\mathfrak{g}_{-\alpha}] = 0\).
	\end{proof}
	
	\begin{definition}\label{definition he_n}
		For each  $n\geq0$, let $\h\e(n)$ denote the following Lie superalgebra: let $\h(n)$ be the Lie superalgebra constructed in Example \ref{example of h_n}, and let $C_1$ denote an irreducible representation of $\h(n)$ in which $c$ acts as the identity; assume that $C_1$ is purely even if $n=0$.  
		
		Now set $C_{-1}:=C_1^\vee$, and let $\h$ denote the quotient of the $\h(n)$-module $C_1\otimes C_{-1}$ by the submodule generated by odd elements lying in its radical (see Remark \ref{remark on vanishing}).  Then $C_{-1}\oplus \h\oplus\h(n)\oplus C_1$ is Clifford Kac--Moody, where $\h$ is abelian and acts trivially on $C_{\pm1}$, and the bracket map $[-,-]:C_{1}\otimes C_{-1}\to\h\oplus \h(n)$ is the quotient map onto $\h$.  We set $\h\e(n):=C_{-1}\oplus\h\oplus C_1$.  
		
		Similarly let $\h^{\Pi}$ denote the quotient of $C_{1}\otimes\Pi C_{-1}$ by the submodule generated by odd elements lying in the radical.  Then $\Pi C_{-1}\oplus\h^{\Pi}\oplus\h(n)\oplus C_1$ will be Clifford Kac--Moody in a similar way, and we set $\h\e(n)^{\Pi}:=\Pi C_{-1}\oplus \h^{\Pi}\oplus C_1$.
	\end{definition}
	\begin{example}\label{example_he(n)}
		If $n=0$ in the above, we obtain the purely even algebra $\h\e(0)=\C\langle e,h,f\rangle$ with $h$ central and $[e,f]=h$.  On the other hand, $\h\e(0)^{\Pi}=\C\langle e,h,f\rangle$ where $e,h$ are odd, $f$ is even, $h$ is central, and $[e,f]=h$.
		
		Finally we note that $\h\e(2)\cong\mathfrak{t}\s\l(1|1)\cong\mathfrak{t}\h\e(0)$.
	\end{example}
	
	The following lemma is straightforward, and mostly given for purposes of clarity.
	\begin{lemma}
		If $\alpha$ is a simple root of rank $n$ and of Heisenberg type, then $\g\langle\alpha\rangle$ is a quotient by a central ideal of $\h$ of one of the following:
		\begin{enumerate}
			\item if $n$ is odd, then is it a quotient of $\h\e(n)$;
			\item if $n\equiv 2$(mod 4), then it is a quotient of either $\h\e(n)$ or $\h\e(n)^{\Pi}$;
			\item if $n\equiv 0$(mod 4), then it is a quotient of either $\h\e(n),$ $\h\e(n)^{\Pi}$, or $\Pi C_{-1}\oplus \h\oplus\Pi C_1$ in the notation of Definition \ref{definition he_n}.
		\end{enumerate} 
	\end{lemma}

	\subsection{Simple roots of non-Heisenberg type}   The goal for the rest of this section is to prove the following theorem:
	
	\begin{theorem}\label{theorem-one-simple-root}
		Suppose \(\h_{\alpha} \nsubseteq \mathrm{Ann}_{\mathfrak{h}}\mathfrak{g}_{\alpha}\). Then \(\mathfrak{g}_{-\alpha} \simeq \mathfrak{g}_{\alpha}^\vee\), and one of the following happens:
		\begin{enumerate}[label=(\roman*)]
			\item \(\rk\alpha = 0\), \(\Delta = \{\pm \alpha\}\), and \(\mathfrak{g}\langle \alpha \rangle\) is isomorphic to \(\s\l(2)\).
			\item \(\rk\alpha = 0\), \(\Delta = \{\pm \alpha,\pm 2\alpha\}\), and \(\mathfrak{g}\langle \alpha \rangle\) is isomorphic to \(\mathfrak{osp}(1|2)\).
			\item \(\rk\alpha = 2\), \(\Delta = \{\pm \alpha\}\), and \(\mathfrak{g}\langle \alpha \rangle\) is isomorphic to either $\mathfrak{tsl}(2)\cong\s\q(2)$ or $\mathfrak{ptsl}(2)\cong\p\s\q(2)$ (see Examples \ref{example_sqn}, \ref{example_psq(n)}, and \ref{example_ts_pts}).
			\item \(\rk\alpha = 2\), \(\Delta = \{\pm \alpha,\pm 2\alpha\}\), and \(\mathfrak{g}\langle \alpha \rangle\) is isomorphic to either $\mathfrak{t}\o\s\p(1|2)$ or $\p\mathfrak{t}\o\s\p(1|2)$.
		\end{enumerate}
	\end{theorem}
	Stated more simply, the above theorem says that if $\alpha$ is not of Heisenberg type, then either it gives rise to $\s\l(2)$, $\o\s\p(1|2)$, or a Takiff construction on one of these.
	
	We deduce Theorem \ref{theorem-one-simple-root} from the following results. 
	\begin{lemma}\label{lemma-root-subalgebra}
		Suppose \(\h_{\alpha} \nsubseteq \mathrm{Ann}_{\mathfrak{h}}\mathfrak{g}_{\alpha}\). There exist \(e \in \mathfrak{g}_{\alpha}\) and \(f \in \mathfrak{g}_{-\alpha}\) such that the subalgebra they generate is isomorphic to either \(\s\l(2)\) or \(\mathfrak{osp}(1|2)\).
	\end{lemma}
	\begin{proof}
		First let us show that we can find an even \(\alpha\)-coroot \(h \in \h_{\alpha}\) such that \(\alpha(h) \neq 0\). By assumption, there exists a homogeneous \(\alpha\)-coroot that does not annihilate \(\mathfrak{g}_{\alpha}\). If it is even, take \(h\) to be this coroot. If it is odd denote it by \(H\). It follows from \([H,\mathfrak{g}_{\alpha}]\neq 0\) and from the irreducibility of \(\mathfrak{g}_{\alpha}\) that \(H \notin \ker B_\alpha\). So there exists \(K \in \mathfrak{h}_{\ol{1}}\) such that \(\alpha([K,H]) \neq 0\). Then \(h \coloneq [K,H] \in \mathfrak{h}_{\alpha}\) is the even coroot we are looking for. \par
		Now, by definition, there exists a finite index set \(I\) and homogeneous elements \(e_i \in \mathfrak{g}_{\alpha}\) and \(f_i \in \mathfrak{g}_{-\alpha}\) of the same parity for each \(i\in I\) such that \(h = \sum_{i\in I} [e_i,f_i]\). Hence \(\alpha(h) = \alpha(\sum_{i\in I} [e_i,f_i]) \neq 0\), which implies that \(\alpha([e_j,f_j]) \neq 0\) for some \(j \in I\). 
		
		Now by our integrability assumption, it is straightforward to check that $e:=e_j$, $f:=f_j$ generate a subalgebra isomorphic to either $\s\l(2)$ or $\o\s\p(1|2)$, depending on whether $e,f$ are even or odd.  
	\end{proof}
	We denote by \(\mathfrak{s}\) the subalgebra isomorphic to either $\s\l(2)$ or $\o\s\p(1|2)$ obtained in Lemma \ref{lemma-root-subalgebra}.
	\begin{lemma}\label{lemma-sl2-osp2-subalgebras}
		Suppose \(\h_{\alpha}\nsubseteq \mathrm{Ann}_{\mathfrak{h}}\mathfrak{g}_{\alpha}\), and let \(\mathfrak{s}\) be as above. 
		\begin{enumerate}[label=(\roman*)]
			\item If \(\mathfrak{s} \simeq \s\l(2)\), then \(\Delta = \{\pm \alpha\}\).
			\item If \(\mathfrak{s}\simeq \mathfrak{osp}(1|2)\), then \(\Delta = \{\pm \alpha,\pm 2\alpha\}\).
		\end{enumerate}
		Moreover, \(\rk\alpha \leq 2\) and \(\dim \mathfrak{g}_\beta = \dim \mathfrak{g}_{\alpha}\) for any \(\beta \in \Delta\). 
	\end{lemma}
	\begin{proof}
		The integrability of \(\mathfrak{g}(\mathcal{A})\) implies that it is finite-dimensional, and in particular a finite dimensional \(\mathfrak{s}\)-module. If \(e \in \mathfrak{s}\cap \mathfrak{g}_{\alpha}\) is non-zero, then by the finite-dimensional representation theory of \(\s\l(2)\) and \(\mathfrak{osp}(1|2)\), the map
		\begin{align}\label{equation-sl2-action}
			\mathrm{ad}\ e : \mathfrak{g}_{k\alpha} \rightarrow \mathfrak{g}_{(k+1)\alpha}
		\end{align}
		is surjective for all \(k \geq 0\). In particular \(\dim \mathfrak{g}_{k\alpha} \leq \dim \mathfrak{g}_{\alpha}\) for all \(k > 0\). Since \(\mathfrak{g}_{\alpha}\) is irreducible and \(\rk\alpha = \rk k\alpha\) for all \(k\neq 0\), we have \(\dim \mathfrak{g}_{k\alpha} \geq \dim \mathfrak{g}_{\alpha}\) whenever \(\mathfrak{g}_{k\alpha} \neq 0\). As a consequence, \(\mathrm{ad}\ e : \mathfrak{g}_{k\alpha} \rightarrow \mathfrak{g}_{(k+1)\alpha}\) is a linear isomorphism for \(k > 0\) whenever \(\mathfrak{g}_{(k+1)\alpha}\neq 0\). \par
		If \(\mathfrak{s}\simeq \s\l(2)\), then \([e,e] = 0\), hence \(\mathfrak{g}_{2\alpha} = 0\). If \(\mathfrak{s}\simeq \mathfrak{osp}(1|2)\), then \([e,e]\neq 0\) and \([e,[e,e]] = 0\), so \(\mathfrak{g}_{3\alpha} = 0\). The description of \(\Delta\) follows immediately. \par
		Finally, from \eqref{equation-sl2-action}, we have \([e,\mathfrak{h}] = \mathfrak{g}_{\alpha}\). As \([e,\mathfrak{t}] = \mathbb{C}\langle e\rangle\), we obtain that either \(\dim (\mathfrak{g}_\alpha)_{\ol{0}} = 1\) or \(\dim (\mathfrak{g}_\alpha)_{\ol{1}} = 1\). Therefore \(\rk\alpha \leq 2\) by our dimension formulas for irreducibles given in Section \ref{section_quasi_toral}.
	\end{proof}
	\begin{lemma}\label{lemma-rk0-rk2}
		If \(\mathfrak{h}_{\alpha} \nsubseteq \mathrm{Ann}_{\mathfrak{h}}\mathfrak{g}_{\alpha}\), then \(\rk\alpha = 0\) or \(\rk\alpha = 2\).
	\end{lemma}
	\begin{proof}
		Let us assume \(\rk\alpha \neq 0\). Then according to Lemma \ref{lemma-sl2-osp2-subalgebras}, we have \(1 \leq \rk\alpha \leq 2\). The irreducibility of \(\mathfrak{g}_{\alpha}\) as an \(\mathfrak{h}\)-module implies \(\dim \mathfrak{g}_{\alpha} = (1|1)\). Let \(k \in \mathbb{Z}\) be the maximal integer such that \(k\alpha \in \Delta\) and fix \(\beta = k\alpha\). Then \(\dim \mathfrak{g}_\beta = (1|1)\) from Lemma \ref{lemma-sl2-osp2-subalgebras}. Choose non-zero elements \(x \in (\mathfrak{g}_{\beta})_{\ol{0}}\), \(X \in (\mathfrak{g}_{\beta})_{\ol{1}}\) and \(y \in (\mathfrak{g}_{-\beta})_{\ol{0}}\), and set \(H \coloneq [X,y]\). We notice that \(x\) and \(y\) belong to an \(\s\l(2)\)-subalgebra, so \(\beta([x,y]) \neq 0\). We have
		\begin{align*}
			[H,x] = [[X,y],x] = [X,[y,x]] = \beta([x,y])X \neq 0,
		\end{align*}
		and
		\begin{align*}
			[H,X] = [[X,y],X] = [X,[y,X]] = [[y,X],X] = -[[X,y],X] = -[H,X].
		\end{align*}
		Thus \([H,X] = 0\), and we conclude that \([H,\mathfrak{g}_{\beta}] \neq 0\) and \(\beta([H,H]) = 0\). This is impossible if \(\rk\beta = 1\),and thus we must have \(\rk\beta=\rk\alpha = 2\).
	\end{proof}
	\begin{corollary}\label{corollary-twist}
		If \(\mathfrak{h}_{\alpha} \nsubseteq \mathrm{Ann}_{\mathfrak{h}}\mathfrak{g}_{\alpha}\), then \(\mathfrak{g}_{-\alpha} \simeq \mathfrak{g}_{\alpha}^\vee\). 
	\end{corollary}
	\begin{proof}
		If \(\rk\alpha = 0\), then \(\mathfrak{g}_{-\alpha} \simeq \mathfrak{g}_{\alpha}^\vee\) because \(\mathfrak{g}_{\alpha}\) and \(\mathfrak{g}_{-\alpha}\) are of the same parity. If \(\rk\alpha = 2\), then according to \eqref{equation-4Z+2} in Section \ref{section-irreducible-h-modules}, we have \(\mathfrak{g}_{\alpha}\otimes \mathfrak{g}_{-\alpha} \simeq \Pi \mathcal{S}(\mathfrak{h}_{\ol{1}}/\ker B_\alpha)\) if \(\mathfrak{g}_{-\alpha} \simeq \mathfrak{g}_{\alpha}^\vee\) and \(\mathfrak{g}_{\alpha}\otimes \mathfrak{g}_{-\alpha} \simeq \mathcal{S}(\mathfrak{h}_{\ol{1}}/\ker B_\alpha)\) otherwise. 
		
		However in the latter situation, the top of $\mathcal{S}(\mathfrak{h}_{\ol{1}}/\operatorname{ker} B_\alpha)$ as an $\h$-module is even, and thus $\h_{\alpha}=\operatorname{Im}([-,-]:\g_{\alpha}\otimes\g_{-\alpha}\to \h)$ would be purely even.  However the proof of Lemma \ref{lemma-rk0-rk2} has shown that $\h$ contains a non-trivial odd element. Our result follows.
	\end{proof}
	
	\subsection{Realization of  \(\mathfrak{g}\langle \alpha \rangle\) as in Theorem \ref{theorem-one-simple-root}}\label{section-realization}
	Again, suppose \(\mathfrak{h}_{\alpha} \nsubseteq \mathrm{Ann}_{\mathfrak{h}}\mathfrak{g}_{\alpha}\). We continue building on the results of Lemmas \ref{lemma-sl2-osp2-subalgebras} and \ref{lemma-rk0-rk2} along with Corollary \ref{corollary-twist}. If \(\rk\alpha =0\), then it is clear that \(\mathfrak{g}\langle \alpha \rangle\) is isomorphic to either \(\s\l(2)\) or \(\mathfrak{osp}(1|2)\). So assume \(\rk\alpha = 2\). In this case $\g_{\alpha}$ is $(1|1)$-dimensional.
	
	If $\Delta=\{\pm\alpha\}$, let \(e\) and \(E\) be even and odd basis vectors of $\g_{\alpha}$, respectively.  As \(\mathfrak{g}_{-\alpha} \simeq \mathfrak{g}_{\alpha}^\vee\), under this identification we may write \(f \coloneq e^\vee\) and \(F \coloneq \sqrt{-1}E^\vee\) as in Proposition \ref{proposition-rank-2-chevalley}, to form a (unique up to scalar) homogeneous basis for \(\mathfrak{g}_{-\alpha}\).
	
	If $\Delta=\{\pm\alpha,\pm2\alpha\}$, we instead let \(E\) and \(e\) be even and odd basis vectors of $\g_{\alpha}$ (note the change of order), and set $f=\sqrt{-1}e^{\vee}$ and $F=E^\vee$.
	
	Now in either case, we define the following pure coroots:
	\begin{align*}
		\begin{array}{ccc}
			h \coloneq [e,f], & c \coloneq [E,F], & H \coloneq [E,f].
		\end{array}
	\end{align*}
	
	If $\Delta=\{\pm\alpha\}$, then we normalize \(e\) in the  such that \(\alpha(h) = 2\), and if $\Delta=\{\pm\alpha,\pm2\alpha\}$ we normalize $e$ so that $\alpha(h)=1$.  Then the proof of Proposition \ref{proposition-rank-2-chevalley} and a direct computation give the following relations:
	\begin{align*}
		\begin{array}{ccccc}
			H = [e,F], & [H,e] =\alpha(h)E, & [H,E] = 0, & [H,H] = 2c,  & \alpha(c)=0.
		\end{array}
	\end{align*}
	It is now straightforward to check that $\g\langle\alpha\rangle$ is isomorphic to either $\mathfrak{t}\s$ or $\mathfrak{pt}\s$ where $\s=\s\l(2)$ if $\Delta=\{\pm\alpha\}$ and $\s=\o\s\p(1|2)$ if $\Delta=\{\pm\alpha,\pm2\alpha\}$.
	
	\begin{proof}[Proof of Theorem \ref{theorem-one-simple-root}]
		The theorem is an immediate consequence of Lemmas \ref{lemma-root-subalgebra} and \ref{lemma-sl2-osp2-subalgebras}, and \ref{lemma-rk0-rk2}, along with Corollary \ref{corollary-twist} and Section \ref{section-realization}.
	\end{proof}
	
	\section{On Connectivity of simple roots}\label{section-connectivity}
	
	Let \(\mathfrak{g}(\mathcal{A})\) be a Clifford Kac--Moody algebra. For a simple root \(\alpha \in \Pi\), let \(\mathfrak{g}\langle\alpha\rangle\) denote the subalgebra of \(\mathfrak{g}(\mathcal{A})\) generated by \(\mathfrak{g}_{\alpha}\) and \(\mathfrak{g}_{-\alpha}\). Using the results of Section \ref{section-one-root}, we define the root type of a simple root \(\alpha \in \Pi\) according to the following table: 
	
	\renewcommand{\arraystretch}{1.25}
	\begin{center}
		\begin{tabular}{ | c | c | }
			\hline
			root type  & the root \(\alpha\) satisfies \\ 
			\hline
			\hline
			$\s\l(2)$ & $\g\langle\alpha\rangle\cong\s\l(2)$ \\  
			\hline
			$\o\s\p(1|2)$ & $\g\langle\alpha\rangle\cong\o\s\p(1|2)$ \\  
			\hline
			$\s\l(1|1)$ & $\g\langle\alpha\rangle\cong\s\l(1|1)$ \\
			\hline 
			$\h\e(0)$ & $\g\langle\alpha\rangle$ a purely even Heisenberg of rank 0\\
			& (see Example \ref{example_he(n)})\\
			\hline
			$\h\e(0)^{\Pi}$ & a `mixed' parity Heisenberg of rank 0\\
			& (see Example \ref{example_he(n)})\\
			\hline
			$Tak(\s\l(2))$ & $\rk\alpha=2$ and $\g\langle\alpha\rangle$ a central quotient of $\mathfrak{t}\s\l(2)$ \\
			& (case (iii) of Theorem \ref{theorem-one-simple-root})\\
			\hline 
			$Tak(\o\s\p(1|2))$ & $\rk\alpha=2$ and $\g\langle\alpha\rangle$ a central quotient of $\mathfrak{t}\o\s\p(1|2)$ \\
			& (case (iv) of Theorem \ref{theorem-one-simple-root})\\
			\hline 
			$Tak(\s\l(1|1))$ & $\rk\alpha=2$ and $\g\langle\alpha\rangle$ a central quotient of $\mathfrak{t}\s\l(1|1)\cong\mathfrak{t}\h\e(0)$\\
			& (see Example \ref{example_ts_pts}) \\
			\hline
			$H_n$ & $\alpha$ is of Heisenberg type and $\rk\alpha=n$ \\
			& (see Section \ref{section_heisenberg})\\
			\hline
		\end{tabular}
	\end{center}
	
	\
	
	We emphasize that each simple root in a Clifford Kac--Moody algebra has a well-defined type. The above table hints that simple roots of rank \(0\) and rank \(2\) are distinguished among all simple roots. In this section, we indeed show that if we want `interesting' interactions between simple roots, i.e.~without any obvious ideals, we should ask that they are all of rank zero or rank 2.
	
	\subsection{Connectivity}\label{section_connectivity} The next natural question to address is when two simple roots can interact, according to their type in the above table.  To be precise, we are interested in the following question:
	
	\textbf{Question:} for which ordered pairs of simple roots $(\alpha,\beta)$ of given type in the above table do the following equivalent condition hold:
	\begin{enumerate}
		\item $[\h_{\alpha},\g_{\beta}]\neq0$;
		\item $[(\h_{\alpha})_{\ol{0}},\g_{\beta}]\neq0$;
		\item there exists a pure, even $\alpha$-coroot $h$ (see Section \ref{section-g(A)}) for which $\beta(h)\neq0$.
	\end{enumerate}    
	
	Notice that when the above question has a positive answer, we will have that $\g_{\alpha+\beta}\neq0$, i.e.~$\alpha+\beta\in\Delta$.
	
	Equivalently, the above question asks when it is a necessary condition that $\h_{\alpha}\subseteq\operatorname{Ann}_{\h}\g_{\beta}$, according to the root type, as listed in the above table.  As an example, if $\alpha$ is of $\h\e(0)$-type and $\beta$ is of any other type, then we always have $[\h_{\alpha},\g_{\beta}]=0$.  This is because even Heisenbergs have no finite-dimensional representations with nontrivial central character.
	
	This section seeks to answer the above question, and the answer is depicted in the picture below (this is an expanded version of Figure \ref{Fig1a} from the introduction).  Namely, in the below diagram we positioned all root types in separate places, and then we draw an arrow from root type $\alpha$ to root type $\beta$ if it is possible that $[\h_{\alpha},\g_{\beta}]\neq0$ in a Clifford Kac--Moody algebra.
	
	\
	
	\
	
	\[
	\xymatrix{
		&\s\l(2) \ar[rr] \ar@(ul,ur) \ar@{=>}[d] \ar@/^3pc/[rrrr] & & \ar[ll]  \o\s\p(1|2) \ar@(ul,ur) \ar@{=>}[d]\ar[rr] & & \ar@/_3pc/[llll]  \ar[ll] \s\l(1|1) \ar@(ul,ur) \ar@{=>}[d] & \\
		&&&&&&\\ 
		&Tak(\s\l(2)) \ar@/^3pc/@{-->}[rrrrr] \ar@(ul,ur) \ar[rr] \ar[rrdd] \ar[rrrrdd] \ar@/^2pc/[rrrr] & & \ar[ll] Tak(\o\s\p(1|2)) \ar@(ul,ur) \ar[dd] \ar@/^1pc/[ddrr]\ar[rr] & & \ar@/_2pc/[llll] \ar[ll] \ar[dd] \ar[ddll] Tak(\s\l(1|1)) \ar@(ul,ur) & \ar[l]\ar[ddl]\ar[ddlll]H_1 \\
		&&&&&&\\
		& \h\e(0) & \h\e(0)^{\Pi} & \h\e(2)^{\Pi} & & H_n, \ n\geq 3 &
	}
	\]
	
	In the above diagram, the bold-faced downward arrows $\Rightarrow$ from $\s\l(2),\o\s\p(1|2)$, and $\s\l(1|1)$ signify that if $\alpha$ is one of any of these three root types and $\beta$ is another root type, then it is possible that $[\h_{\alpha},\g_{\beta}]\neq0$.  We have used this notation to avoid an overwhelming thicket of arrows from these three root types.
	
	Further, the dashed arrow pointing to $H_1$ is meant to signify that if $\alpha$ is of type $Tak(\s\l(2))$ and $\beta$ is of type $H_1$, and if we have $[\h_{\alpha},\g_{\beta}]\neq0$, then $[\h_{\beta},\g_{\gamma}]=0$ for any simple root $\gamma$, i.e.~$\beta$ is forced to become a `sink'.  
	
	\subsection{Sinks} In light of condition (3) in our above question, an important question is when all pure, even coroots of $\g\langle\alpha\rangle$ lie in an even Heisenberg triple, i.e.~lie in a subalgebra isomorphic to $\h\e(0)$.  In this case $[\h_{\alpha},\g_{\beta}]=0$ for all simple roots $\beta$, i.e.~$(\h_{\alpha})_{\ol{0}}$ will be central in $\g(\AA)$. One of the main results of this section is the following:
	
	\begin{proposition}\label{prop_sinks}
		All pure, even coroots of $\g\langle\alpha\rangle$ lie in an even Heisenberg triple whenever $\alpha$ is one of the following types:
		\begin{enumerate}
			\item $\h\e(0)$;
			\item $\h\e(0)^{\Pi}$;
			\item $\h\e(2)^{\Pi}$;
			\item $H_n$ for $n\geq 3$.
		\end{enumerate}
		In particular, in these cases, $[\h_{\alpha},\g_{\beta}]=0$ for all simple roots $\beta$.  
	\end{proposition}
	
	The above proposition tells us that if $\alpha$ is a simple root of the above type, then $\g\langle\alpha\rangle$ generates an ideal containing no other simple roots. Thus they are `sinks' in the theory, which we see from the diagram above. 
	
	\begin{proof}[Proof of Proposition \ref{prop_sinks}]
		The statement for $\h\e(0)$ is clear, and for $\h\e(0)^{\Pi}$ there are no even coroots, so that statement is vacuously true.
		
		For $\h\e(2)^{\Pi}$, let \(e\) and \(E\) be even and odd basis vectors of \(\mathfrak{g}_{\alpha}\), and \(f\) and \(F\) even and odd basis vectors of \(\mathfrak{g}_{-\alpha}\). Then our claim reduces to showing \([E,F] \in \mathbb{C}\langle[e,f]\rangle\). 
		
		Choose $K\in\h$ such that $\alpha(K^2)\neq0$, so that $K$ defines an odd automorphism of $\g_{\pm\alpha}$.  Then \([K,e] = aE\) and \([K,F] = bf\) for non-zero \(a,b \in \mathbb{C}\). we know from Section \ref{section_dualities} that $\h_{\alpha}$ is one-dimensional and purely even, so that $[e,F]=0$.  Thus
		\begin{align*}
			0 = [K,[e,F]] = a[E,F] + b[e,f],
		\end{align*} 
		which gives our result.
		
		Finally, we deal with case of $H_n$ for $n\geq 3$.   In the proof, we will use the realization of a simple \(\mathfrak{h}\)-module given in Section \ref{subsection-realization-irreducible}. We remind here that \(\mathbb{C}[\xi_1,...,\xi_m]\) is the superalgebra of polynomial in odd variables \(\xi_1,...,\xi_m\) satisfying the relation \(\xi_i\xi_j = -\xi_j\xi_i\) for all \(i,j\). For a set \(I=\{i_1,...,i_k\} \subseteq \{1,...,m\}\), we write \(\xi_I \coloneq \xi_{i_1}\cdot...\cdot \xi_{i_k}\) for \(i_1<...<i_k\). 
		
		We first consider the case where \(\rk\alpha\) is even, let \(m \in \mathbb{Z}\) be such that \(\rk\alpha = 2m\). Then our assumption on \(\rk\alpha\) implies \(m \geq 2\). According to Section \ref{subsection-realization-irreducible}, we can identify \(\mathfrak{g}_{\alpha}\) as \(\mathbb{C}[\xi_1,...,\xi_m]\) up to parity, and choose \(H_1,...,H_m,\bar{H}_1,...,\bar{H}_m \in \mathfrak{h}_{\ol{1}}\) such that \(H_i\) acts on \(\mathfrak{g}_{\alpha}\) via multiplication by \(\xi_i\), and \(\bar{H}_i\) by the derivation \(\partial_{\xi_i}\). Similarly, we can identify \(\mathfrak{g}_{-\alpha}\) as \(\mathbb{C}[\phi_1,...,\phi_m]\) up to parity, such that \(H_i\) acts on \(\mathfrak{g}_{-\alpha}\) via multiplication by \(-\phi_i\), and \(\bar{H}_i\) be the derivation \(\partial_{\phi_i}\). 
		
		Let \(\xi_I\) and \(\phi_J\) be odd elements; we will show that \([\xi_I,\phi_J]\) is either zero or a bracket of two even elements  from \(\mathfrak{g}_{\alpha}\) and \(\mathfrak{g}_{-\alpha}\). We will use the fact that \(\h_{\alpha} = [\mathfrak{g}_{\alpha},\mathfrak{g}_{-\alpha}] \subseteq \mathfrak{h}\), and so \([\mathfrak{h}_{\ol{1}},[\mathfrak{h}_{\ol{1}},\mathfrak{h}_{\alpha}]] = 0\) because \(\mathfrak{h}\) is quasi-toral. 
		
		If \(\lvert I\cap J\rvert\geq 2\), then we can find \(i,j \in I\cap J\) such that \(i\neq j\), so
		\begin{align*}
			0 = [H_i,[H_j,[\xi_{I\setminus\{i,j\}},\phi_J]]] = \pm [\xi_I,\phi_J]. 
		\end{align*}
		If \(I=J=\{i\}\), then because \(m\geq2\), we can find \(j\notin I\) so that
		\begin{align*}
			0=[H_i,[\bar{H}_j,[\xi_j,\phi_i]]] = [H_i,[\xi_\emptyset,\phi_i]] = [\xi_i,\phi_i]. 
		\end{align*}
		If \(I=J=\emptyset\), then because \(m\geq 2\), we can find \(i,j \in \{1,...,m\}\) such that \(i\neq j\), so
		\begin{align*}
			0 = [\bar{H}_i,[\bar{H}_j,[\xi_i,\phi_j]]] = \pm [\bar{H}_i,[\xi_i,\phi_{\emptyset}]] = \pm [\xi_\emptyset,\phi_\emptyset].
		\end{align*}
		Finally, if none of the above happens, then it must be that \(I\neq J\). If \(i \in I\setminus J\), then
		\begin{align*}
			0 = [\bar{H}_i,[H_i,[\xi_I,\phi_J]]] = [\bar{H}_i,[\xi_I,-\phi_i\phi_J]] = \pm[\xi_{I\setminus\{i\}},\phi_{J\cup\{i\}}] \pm [\xi_I,\phi_J].
		\end{align*}
		If \(i \in J\setminus I\), then
		\begin{align*}
			0 = [\bar{H}_i,[H_i,[\xi_I,\phi_J]]] = [\bar{H}_i,[\xi_i\xi_I,\phi_J]] = [\xi_I,\phi_J] \pm [\xi_{I\cup\{i\}},\phi_{J\setminus\{i\}}].
		\end{align*}
		In the above two equations, we see that our term of interest, \([\xi_I,\phi_J]\), is equal to a bracket of two even elements, as desired.
		
		We now consider the case where \(\rk\alpha\) is odd, let \(m \in \mathbb{Z}\) be such that \(\rk\alpha = 2m+1\). Again, from Section \ref{subsection-realization-irreducible}, we can find \(H_1,...,H_m,\bar{H}_1,...,\bar{H}_m,\tilde{H} \in \mathfrak{h}_{\ol{1}}\) such that \(\mathfrak{g}_{\alpha}\) can be realized as \(\mathbb{C}[\xi_1,...,\xi_{m+1}]\), such that \(H_i\) acts by \(\xi_i\), \(\bar{H}_i\) acts by \(\partial_{\xi_i}\), and \(\tilde{H}\) acts by \(\xi_{m+1}+\partial_{\xi_{m+1}}\). Similarly, \(\mathfrak{g}_{-\alpha}\) can be realized as \(\mathbb{C}[\phi_1,...,\phi_{m+1}]\), such that \(H_i\) acts by \(-\phi_i\), \(\bar{H}_i\) acts by \(\partial_{\phi_i}\), and \(\tilde{H}\) acts by \(-\phi_{m+1}+\partial_{\phi_{m+1}}\). \par
		As in the previous case, let \(\xi_I\) and \(\phi_J\) be odd elements. We notice that in the current setting (\(\rk\alpha = 2m+1\)), it must be that \(\lvert I \rvert\) and \(\lvert J \rvert\) are odd integers. We will show that \([\xi_I,\phi_J]\) is either zero or a bracket of two even elements from \(\mathfrak{g}_{\alpha}\) and \(\mathfrak{g}_{-\alpha}\). The only computations left (which are not analogous to the case when \(\rk\alpha\) is even) are for \(I\setminus \{m+1\} = J\setminus \{m+1\}\) and \(\lvert I\setminus \{m+1\} \rvert \leq 1\). If \(I=J\) and \(I\setminus\{m+1\} = \{i\}\), then
		\begin{align*}
			0 = [\tilde{H},[H_i,[\xi_{m+1},\phi_i]]] = [\tilde{H},[\xi_i\xi_{m+1},\phi_i]] = \pm [\xi_i,\phi_i] \pm [\xi_i\xi_{m+1},\phi_i\phi_{m+1}].
		\end{align*}
		If \(I = J = \{m+1\}\), then
		\begin{align*}
			0 = [\tilde{H},[\bar{H}_1,[\xi_1,\phi_{m+1}]]] = [\tilde{H},[\xi_\emptyset,\phi_{m+1}]] = [\xi_{m+1},\phi_{m+1}] \pm [\xi_{\emptyset},\phi_{\emptyset}].
		\end{align*}
		This concludes our proof.
	\end{proof}

	\subsection{Connectivity properties of $Tak(\s\l(2))$, $Tak(\o\s\p(1|2))$, and $Tak(\s\l(1|1))$}
	
	In this section we assume that $\g\langle\alpha\rangle$ is of type $Tak(\s\l(2))$, $Tak(\o\s\p(1|2))$, or $Tak(\s\l(1|1))$; equivalently we assume that $\rk\alpha=2$ and $\g_{\alpha}^\vee\cong\g_{-\alpha}$.  In this case, by Section \ref{section_dualities}, $\h_{\alpha}$ is generated by an odd element which we call $H_{\alpha}$.
	\begin{lemma}
		If $\beta$ is a simple root with $\rk\beta=0$, we have $[\h_{\alpha},\g_{\beta}]=0$.  
	\end{lemma}
	\begin{proof}
		Since $[H_{\alpha},\g_{\beta}]=0$ we are done.
	\end{proof}
	
	\begin{lemma}\label{lemma_conn_tak}
		Suppose that $\beta$ is of rank 1.  Then 
		\begin{enumerate}
			\item if $\alpha$ is of type $Tak(\s\l(1|1))$ or $Tak(\o\s\p(1|2))$, then $[\h_{\alpha},\g_{\beta}]=0$;
			\item if $\alpha$ is of type $Tak(\s\l(2))$ and $[\h_{\alpha},\g_{\beta}]\neq0$, then $[\h_{\beta},\g_{\gamma}]=0$ for all simple roots $\gamma$ (i.e.~$H_1$ `becomes' a sink, justifying the dashed arrow in our diagram).
		\end{enumerate}
	\end{lemma}
	
	\begin{proof}
		For $(1)$, we have $H_{\alpha}^2=0$ for $Tak(\s\l(1|1))$, while $H_{\alpha}^2$ is the center of an even Heisenberg triple for $Tak(\o\s\p(1|2))$.  Thus we necessarily have $\beta(H_{\alpha}^2)=0$ for these two cases. Since $\beta$ is of rank one this implies that $[H_{\alpha},\g_{\beta}]=0$, and we are done.
		
		For $(2)$, if $[\h_{\alpha},\g_{\beta}]\neq0$, then clearly $H_{\alpha}$ acts by an automorphism on $\g_{\beta}$. Write $e,E$ and $f,F$ for homogeneous bases of $\g_{\beta}$ and $\g_{-\beta}$ respectively, and set $H_{\beta}=[E,f]$.  Then we see that
		\[
		[H_{\alpha},H_{\beta}]=a[e,f]+b[E,F],
		\] 
		for some nonzero $a,b\in\C$.  On the other hand we may write $\h_{\alpha}\ni[H_{\alpha},H_{\beta}]=rh_{\alpha}+sc_{\alpha}$, where $h_{\alpha}$ is the pure $\alpha$ coroot lying in an $\s\l(2)$-triple, and $c_{\alpha}$ is central in $\g\langle\alpha\rangle$. Further, because $H_{\beta}^2=0$, we see that either $r=0$ or $s=0$, and perhaps both are 0.  Now since
		\[
		a[e,f]+b[E,F]=rh_{\alpha}+sc_{\alpha},
		\]  
		and it is clear that $\beta(a[e,f]+b[E,F])=0$, we also must have $r\beta(h_{\alpha})+s\beta(c_{\alpha})=0$.  By $\s\l(2)$-representation theory we must have $\beta(h_{\alpha})\neq0$, so this forces $r=0$, which tells us that
		\[
		a[e,f]+b[E,F]=sc_{\alpha}.
		\] 
		However since $\alpha(c_{\alpha})=0$, this forces $a\alpha([e,f])+b\alpha([E,F])=0$.  Since $[e,f]$ lies in a Heisenberg triple, we must have $\alpha([e,f])=0$, so necessarily $\alpha([E,F])=0$.  Thus $\alpha((\h_{\beta})_{\ol{0}})=0$, which implies $[H_{\alpha},H_{\beta}]=0$.  It follows that
		\[
		a[e,f]+b[E,F]=0,
		\] 
		meaning that $[E,F]\in\C\langle[e,f]\rangle$, so that all pure, even coroots of  $\g\langle\beta\rangle$ lie in a Heisenberg triple.  This completes the argument.
	\end{proof}
	\subsection{Connectivity properties of $H_1$} 
	\begin{lemma}\label{lemma-rank-1}
		Let \(\alpha \in \Pi\) be such that \(\rk\alpha= 1\), and let \(\beta \in \Pi\) be any simple root with either $\rk\beta<2$ or $\rk\beta=2$ of type $Tak(\s\l(2))$ or $Tak(\o\s\p(1|2))$. Then $[\h_{\alpha},\g_{\beta}]=0$. 
	\end{lemma}
	\begin{proof}		
		For the proof, let us first realize the subalgebra of \(\mathfrak{g}(\mathcal{A})\) generated by \(\mathfrak{g}_{\alpha}\) and \(\mathfrak{g}_{-\alpha}\). Because \(\rk\alpha = 1\), we have \(\mathfrak{g}_{-\alpha} \cong\mathfrak{g}_{\alpha}^\vee\), and thus we identify $\g_{-\alpha}$ and $\g_{\alpha}^\vee$. Let \(e\) and \(E\) be even and odd basis vectors for \(\mathfrak{g}_{\alpha}\), respectively. Then \(f \coloneq e^\vee\) and \(F \coloneq \sqrt{-1}E^\vee\) form a homogeneous basis for \(\mathfrak{g}_{-\alpha}\). Let \(h \coloneq [e,f]\), \(c \coloneq [E,F]\) and \(H \coloneq [e,F]\). We notice that \(e,h,f\) form a Heisenberg subalgebra, so the integrability of \(\mathfrak{g}(\mathcal{A})\) implies \(\beta(h) = 0\). \par 
		Suppose first that \([H,\mathfrak{g}_{\beta}]=0\). The irreducibility of \(\mathfrak{g}_{\alpha}\) implies the existence of \(K\in\mathfrak{h}_{\ol{1}}\) such that \([K,\mathfrak{g}_{\alpha}]=\mathfrak{g}_{\alpha}\). In particular, \([K,H] = ah+bc\) for some non-zero \(a,b \in \mathbb{C}\). Therefore
		\begin{align*}
			0 = [[K,H],\mathfrak{g}_{\beta}] = [ah+bc,\mathfrak{g}_{\beta}].
		\end{align*}
		As \(\beta(h) = 0\), we must have \(\beta(c) = 0\). Thus we obtain \([\h_{\alpha},\mathfrak{g}_{\beta}] = 0\), as desired. We conclude:
		\begin{align}\label{equation-root-implication}
			[H,\mathfrak{g}_{\beta}]=0 \implies [\h_{\alpha},\g_{\beta}]= 0.
		\end{align}
		We now prove the lemma case by case, depending on the properties of \(\beta\). If \(\rk\beta = 0\), we must have \([H,\mathfrak{g}_{\beta}] = 0\), so we are done.  Similarly if $\rk\beta=1$, then since $H^2=0$ we must have $[H,\g_{\beta}]=0$, so we are done.
		
		Suppose now \(\beta\) is of $Tak(\s\l(2))$-type or $Tak(\o\s\p(1|2))$-type.  In the latter case, $\g\langle\alpha\rangle$ contains $(\p)\mathfrak{sq}(2)$ as a subalgebra with the same coroots, so it suffices to assume we are in this case.
		
		However by Lemma \ref{lemma_conn_tak}, if $[\h_{\alpha},\g_{\beta}]\neq0$ then necessarily $[\h_{\beta},\g_{\alpha}]=0$.  If we write $H_{\beta}$ for a nonzero odd coroot of $\h_{\beta}$, then this implies that $[H_1,H_\beta]=0$.  Thus $H,H_\beta$ have commuting actions on $\g_{\beta}$, so since $H_1$ acts nontrivially it must act nontrivially on the $\s\l(2)$ triple inside $\g\langle\beta\rangle$.  However if $h_{\beta}$ denotes the even coroot of said $\s\l(2)$ triple, then this implies $\alpha(h_{\beta})\neq0$, a contradiction. 
	\end{proof}
	
	\section{Type Q Kac--Moody algebras}\label{section-QKM}
	
	\subsection{Regularity and indecomposability} \begin{definition}\label{definition-regular}
		We say that a Cartan datum $\AA$ (and by extension, the Lie superalgebra \(\mathfrak{g}(\AA)\)) is regular if \([\h_{\alpha},\mathfrak{g}_{\beta}] = 0\) implies \([\h_{\beta},\mathfrak{g}_{\alpha}] = 0\) for any \(\alpha,\beta \in \Pi\). 
	\end{definition}
	
	From Proposition \ref{proposition-roots-connectivity} we deduce:
	
	\begin{corollary}\label{corollary-regular}
		If \(\mathfrak{g}\) is regular and \(\alpha,\beta \in \Pi\) are such that \([\h_{\alpha},\mathfrak{g}_{\beta}] = 0\), then \(\alpha+\beta\) is not a root of \(\mathfrak{g}\).
	\end{corollary}
	
	\begin{definition}
		Let $\AA$ be a Cartan datum.  We say that $\AA$ is indecomposable if there does not exist a nontrivial partition $\Pi=\Pi_1\sqcup\Pi_2$ such that $[\h_{\alpha},\g_{\beta}]=[\h_{\beta},\g_{\alpha}]=0$ for all $\alpha\in\Pi_1$, $\beta\in\Pi_2$.
	\end{definition}
	
	Our work in the previous section immediately implies:
	\begin{theorem}\label{theorem-connectivity-aKM}
		Let \(\mathfrak{g}(\mathcal{A})\) be a regular, indecomposable Clifford Kac--Moody algebra with $|\Pi|>1$.  Then one of two possibilities can occur: 
		\begin{enumerate}[label=(\roman*)]
			\item \(\rk\alpha =0\) and \(\mathfrak{g}_{-\alpha} \simeq \mathfrak{g}_{\alpha}^\vee\) for all $\alpha\in\Pi$;
			\item \(\rk\alpha =2\) and \(\mathfrak{g}_{-\alpha} \simeq \mathfrak{g}_{\alpha}^\vee\) for all $\alpha\in\Pi$.
		\end{enumerate}
	\end{theorem}
	
	\subsection{Type Q Kac--Moody superalgebras} Following Theorem \ref{theorem-connectivity-aKM}, it makes sense to consider the class of Clifford Kac--Moody algebras in which any simple root \(\alpha\in \Pi\) satisfies \(\rk\alpha = 2\) and \(\mathfrak{g}_{-\alpha} \simeq \mathfrak{g}_{\alpha}^\vee\). 
	\begin{definition}\label{definition-QKM}
		Let $\AA$ be a Cartan datum.  We say the Lie superalgebra \(\mathfrak{g}(\mathcal{A})\) is type Q Kac--Moody (QKM) if the following conditions are satisfied:
		\begin{enumerate}[label=(KM\arabic*)]
			\item\label{QKM-integrable} \(\AA\) is integrable, see Definition \ref{definition-integrable};
			\item\label{QKM-regular} \(\mathcal{A}\) is regular, see Definition \ref{definition-regular};
			\item\label{QKM-rank2} every \(\alpha \in \Pi\) is of rank \(2\);
			\item\label{QKM-irrtwist} for any $\alpha\in\Pi$, the \(\mathfrak{h}\)-module \(\mathfrak{g}_{\alpha}\) is irreducible, and \(\mathfrak{g}_{-\alpha}\simeq \mathfrak{g}_{\alpha}^\vee\);
			\item\label{QKM-linear-independence} for any \(\alpha \in \Pi\) we have \(\mathfrak{h}_{\alpha} \nsubseteq \sum_{\beta\in\Pi\setminus\{\alpha\}}\mathfrak{h}_{\beta}\).
		\end{enumerate}
		In this situation we will also say that $\AA$ is a QKM Cartan datum.
	\end{definition}
	\begin{remark}
		We notice that every QKM algebra is a Clifford Kac--Moody algebra, and the possible simple root types are $Tak(\s\l(2))$, $Tak(\o\s\p(1|2))$, and $Tak(\s\l(1|1))$. 
	\end{remark}
	
	\begin{remark}\label{remark-Chevalley}
		From Corollary \ref{corollary-Chevalley} we deduce that any QKM algebra admits a Chevalley automorphism $\omega$ extending \(\omega_{\mathfrak{h}}\) and satisfying $\omega^2=\delta$.
	\end{remark}
	
	\begin{remark}
		Condition \ref{QKM-regular} is a convenient assumption for purposes; if removed, we would obtain algebras with a nontrivial ideal inside generated by `non-regular' simple roots.  
	\end{remark}
	
	\begin{remark}
		Property \ref{QKM-linear-independence} is a condition on the linear independence of the simple coroots of \(\mathfrak{g}(\mathcal{A})\). In a Clifford Kac--Moody algebra with all simple roots of rank \(0\), this condition coincides with the usual linear independence of coroots condition as in \cite{K}.
	\end{remark}
	
	\begin{lemma}
		Let $\AA$ be a Cartan datum with simple roots $\Pi=\{\alpha_1,\dots,\alpha_n\}$ satisfying (KM1)-(KM4) of Definition \ref{definition-QKM}.  Then $\AA$ satisfies (KM5) if and only if the odd coroots in $\h_{\alpha_1},\dots,\h_{\alpha_n}$ are linearly independent.
	\end{lemma}
	
	\begin{proof}
		The backward direction is clear.  For the forward direction, suppose that $\AA$ satisfies (KM5) but the odd coroots are not linearly independent.  Since $\dim(\h_{\alpha_i})_{\ol{1}}=1$ for all $i$, let $H_i\in(\h_{\alpha_i})_{\ol{1}}$ be any nonzero element. Then by our assumption, for some $i$ we have $H_i\in\sum_{j\neq i}\h_{\alpha_j}$.  However $H_i$ generates $\h_{\alpha_i}$ as an $\h$-module; thus 
		\[
		\h_{\alpha_i}=\C\langle H_i\rangle\oplus[\h,H_i]\subseteq\sum\limits_{j\neq i}[\h,\h_{\alpha_j}]\subseteq\sum_{j\neq i}\h_{\alpha_j},
		\]
		contradiction (KM5).
	\end{proof}

	\begin{remark}
		It follows from our work that if $\AA$ is QKM, then for a simple root $\alpha$, the subalgebra $\g\langle\alpha\rangle$ is isomorphic to a central quotient of $\mathfrak{t}\s$ for a Kac--Moody Lie superalgebra $\mathfrak{s}$ with one simple root (see Example \ref{example_ts_pts}).  Thus QKM algebras are built by `glueing together' extensions of Takiffs of Kac--Moody superalgebras with one simple root.
		
		A trivial way to glue together such Takiffs is via the Takiff construction applied to an arbitrary Kac--Moody Lie superalgebra; however, in a sense, this is a less interesting type of glueing.  We will find that if we avoid such constructions, we are led to more `interesting' Lie superalgebras such as $\q(n)$, and also that the theory becomes very rigid.
	\end{remark}
	
	\subsection{Cartan datum for QKM algebras}\label{section-notation-QKM}
	
	Because QKM algebras only have simple roots of type $Tak(\s)$ for $\s=\s\l(2),\o\s\p(1|2)$, or $\s\l(1|1)$, we can more explicitly present their Cartan datum.  We do that now, while also introducing notation that will be used henceforth throughout the paper.
	
	Let $\AA$ be a QKM Cartan datum, and let $\Pi=\{\alpha_1,\dots,\alpha_n\}$ be the simple roots.  For each $i$, we have the subalgebra $\g\langle\alpha_i\rangle$ generated by $\g_{\alpha_i}$ and $\g_{-\alpha_i}$.  We choose homogeneous bases of $\g_{\pm\alpha_i}$, using the identification $\g_{-\alpha_i}=\g_{\alpha_i}^\vee$, as follows:
	\begin{enumerate}
		\item if $\alpha_i$ is of type $Tak(\s\l(2))$, set $e_i$ and $E_i$ to be nonzero even and odd elements of $\g_{\alpha_i}$, and let $f_i:=e_i^\vee$ and $F_i:=\sqrt{-1}E_i^\vee$.  Rescale the choice of $e_i$ so that $\alpha_i([e_i,f_i])=2$;
		\item if $\alpha_i$ is of type $Tak(\o\s\p(1|2))$ or $Tak(\s\l(1|1))$, set $E_i$ and $e_i$ to be nonzero even and odd elements of $\g_{\alpha_i}$ (notice the change in order), and let $f_i:=\sqrt{-1}e_i^\vee$ and $F_i:=E_i^\vee$.  If $\alpha_i$ is of type $Tak(\o\s\p(1|2))$, then rescale the choice of $e_i$ so that $\alpha_i([e_i,f_i])=1$.
	\end{enumerate}
	It follows from the above choices that:
	\begin{enumerate}
		\item if $\alpha_i$ is of type $Tak(\s\l(2))$, then $e_i,[e_i,f_i],f_i$ form an $\s\l(2)$ triple;
		\item if $\alpha_i$ is of type $Tak(\o\s\p(1|2))$, then $[e_i,e_i],e_i,[e_i,f_i],f_i,[f_i,f_i]$ form an $\o\s\p(1|2)$-quintuple;
		\item if $\alpha_i$ is of type $Tak(\s\l(1|1))$, then $e_i,[e_i,f_i],f_i$ form an $\s\l(1|1)$-triple.
	\end{enumerate}
	On should view the $e_i,E_i,f_i,$ and $F_i$ as the `Chevalley generators' for a QKM algebra.
	
	\subsubsection{Uniqueness of generators}\label{section_uniqueness_gens} Before going further, we note that the above choices are not unique: in particular we may simultaneously rescale $e_i$ and $f_i$ by $(-1)$, or simultaneously rescale $E_i$ and $F_i$ by $\lambda\in\C^\times$. For $Tak(\s\l(1|1))$ we may also rescale $e_i,f_i$ by any $\lambda\in\C^\times$.
	
	\subsubsection{Pure coroots} Set 
	\[
	H_i:=[E_i,f_i];
	\]
	one may check that $H_i\neq0$, and we also always have 
	\[
	H_i=[e_i,F_i].
	\]
	Further set:
	\[
	h_i:=[e_i,f_i] \ \ \ \ c_i:=[E_i,F_i].
	\]
	
	The following table summarizes some of the important properties, which are easy to prove, of the pure coroots we have introduced.
	\begin{center}
		\begin{tabular}{|c|c|c|c|}
			\hline
			& $Tak(\s\l(2))$ & $Tak(\o\s\p(1|2))$ & $Tak(\s\l(1|1))$\\
			\hline
			$H_i^2=$ & $c_i$ & $c_i$ & 0\\
			\hline 
			$c_i$ lies in... & $\s\l(1|1)$-triple & even Heisenberg triple & even Heisenberg triple\\
			& $\alpha_i(c_i)=0$ & $\beta(c_i)=0$ for all roots $\beta$ & $\beta(c_i)=0$ for all roots $\beta$\\
			\hline
			$h_i$ lies in... & $\s\l(2)$-triple & $\o\s\p(1|2)$ quintuple & $\s\l(1|1)$-triple\\
			& $\alpha_i(h_i)=2$ & $\alpha_i(h_i)=1$ & $\alpha_i(h_i)=0$\\
			\hline
		\end{tabular}	
	\end{center}
	
	\begin{remark}\label{remark_root_subalg_really_takiff}
		A more concrete viewpoint on the elements described above is given using the fact that each $\g\langle\alpha_i\rangle$ is a central quotient of $\mathfrak{t}\s=\s\otimes\C[\xi]\oplus\C\langle c\rangle$ for $\s=\s\l(2),\o\s\p(1|2)$, or $\s\l(1|1)$.  In these terms, $e_i,f_i$ are the Chevalley generators of $\s$, $h_i=[e_i,f_i]$, $c_i=c$, and $H_i=h_i\otimes \xi$.  
	\end{remark}
	
	\subsubsection{$X$ and $Y$ matrices}\label{section_X_Y_matrices}
	
	Let $\AA$ be a QKM Cartan datum.  We now introduce two $n\times n$ matrices $X(\AA)=X=(x_{ij})$ and $Y(\AA)=Y=(y_{ij})$ of complex numbers, which encode much of the Cartan datum. The entries are defined as follows:
		\[
		[H_j,e_i]=x_{ji}E_i, \ \text{ and } \ [H_j,E_i]=y_{ji}e_i.
		\]
		Then because $\g_{-\alpha_i}\cong\g_{\alpha_i}^\vee$, we have
		\[
		[H_i,f_j]=-(-1)^{\ol{e_i}}x_{ij}F_j, \ \text{ and } \ [H_i,F_j]=(-1)^{\ol{e_i}}y_{ij}f_j.
		\]
	
	\begin{remark}\label{remark_rescaling}
		The entries of the matrices $X(\AA)$ and $Y(\AA)$ are not unique; indeed, in addition to permuting indices, if we rescaled our choices of generators $e_i$ and $E_i$ (see \ref{section_uniqueness_gens}) we would scale certain entries in the matrices. Indeed, if we rescale $E_i$ by $\lambda$ the entries change as follows:
		\[
		y_{ij},y_{ji}\mapsto \lambda y_{ij},\lambda y_{ji}, \ \ \ \ x_{ij},x_{ji}\mapsto\lambda x_{ij},x_{ji}/\lambda.
		\]
	\end{remark}

	\subsubsection{Formulas}	
	\begin{lemma}\label{lemma_QKM_formulas}
		We have the following formulas:
		\begin{enumerate}
			\item for any $i,j$:
			\[
			[H_i,H_j]=x_{ij}c_j+y_{ij}h_j = x_{ji}c_i+y_{ji}h_i;
			\]
			\item for any $i,j$:
			\[
			\alpha_k([H_i,H_j])=x_{ik}y_{jk} + x_{jk}y_{ik},
			\]
			in particular,
			\[
			\alpha_j(H_i^2)=x_{ij}y_{ij};
			\]
			\item 
			\[
			y_{ii}=0 \ \text{ for all }i;
			\]
			\item 
			\[
			x_{ii}=\alpha_i(h_i)=\begin{cases}
				2 & \text{ if }\alpha_i\text{ of type }Tak(\s\l(2)),\\
				1 & \text{ if }\alpha_i\text{ of type }Tak(\o\s\p(1|2)),\\
				0 & \text{ if }\alpha_i\text{ of type }Tak(\s\l(1|1)).
			\end{cases}
			\]
			\item 
			\[
			\alpha_i([H_i,H_j])=y_{ij}\alpha_i(h_i).
			\]
		\end{enumerate}
	\end{lemma}
	\begin{proof}
		Formula (1) follows from applying the Jacobi identity $[H_i,H_j]$ with either $H_i=[E_i,f_i]$ or $H_j=[E_j,f_j]$.  Formula (2) holds because $\alpha_k([H_i,H_j])$ is given by the scalar action of $[H_i,H_j]$ on $\g_{\alpha_k}$; however this is the action of the operator $H_iH_j+H_jH_i$, and one sees that it acts by the given scalar.  Formulas (3) and (4) are straightforward, and formula (5) follows from (2), (3), and (4).
	\end{proof}

	\begin{lemma}\label{lemma_nontrivial_relation_on_ys}
		For two simple roots $\alpha_i,\alpha_j$ of type $Tak(\s\l(2))$ or $Tak(\o\s\p(1|2))$ we have
		\[
		x_{ij}x_{ji}y_{ji}+y_{ij}\alpha_i(h_j)=y_{ji}\alpha_i(h_i),
		\]
		and
		\[
		y_{ij}y_{ji}(\alpha_i(h_i)-\alpha_j(h_j))=y_{ij}^2\alpha_i(h_j)-y_{ji}^2\alpha_j(h_i).
		\]
		In particular if $\alpha_i(h_i)=\alpha_j(h_j)$, then
		\[
		y_{ij}^2\alpha_i(h_j)=y_{ji}^2\alpha_j(h_i).
		\]
	\end{lemma}
	\begin{proof}
		Indeed, if we apply $\alpha_i$ to formula (1) of Lemma \ref{lemma_QKM_formulas} we obtain
		\[
		x_{ij}\alpha_i(c_j)+y_{ij}\alpha_i(h_j)=y_{ji}\alpha_i(h_i).
		\]
		Now we use that for $Tak(\s\l(2))$ and $Tak(\o\s\p(1|2))$ type we have $H_i^2=c_i$, and apply (2) of Lemma \ref{lemma_QKM_formulas}, we obtain the first formula.  To obtain the second formula, swap the indices on $i$ and $j$ in the first formula, and subtract one from the other.  The third formula is immediate.
	\end{proof}

	\subsection{Dynkin diagrams}\label{section_dynkin}
	
	\begin{definition}
		Let \(\mathfrak{g}(\mathcal{A})\) be a QKM algebra, with notation as explained in Section \ref{section-notation-QKM}. We call \(A = (a_{ij})=(\alpha_j(h_i))\) the Cartan matrix of \(\mathfrak{g}(\mathcal{A})\). Define $D=D(\AA)$ to be the Dynkin diagram associated with the Cartan datum \(\AA\) as follows: draw a vertex for each $i=1,\dots,n$ according to the following rule:
		\begin{enumerate}
			\item draw \begin{tikzpicture}	\node[diamond,aspect=1,scale=0.6,draw] (A) {};	\end{tikzpicture} if $a_{ii}=2$, or equivalently $\alpha_i$ is of type $Tak(\s\l(2))$;
			\item draw \begin{tikzpicture}	\node[diamond,aspect=1,scale=0.6,draw,fill=black] (A) {};	\end{tikzpicture} if $a_{ii}=1$, or equivalently $\alpha_i$ is of type $Tak(\o\s\p(1|2))$;
			\item draw $\diamondtimes$ if $a_{ii}=0$, or equivalently $\alpha_i$ is of type $Tak(\s\l(1|1))$.
		\end{enumerate} 
		Next we draw an edge from vertex $i$ to vertex $j$ if $a_{ij}a_{ji}\neq0$, and we label the edge by the pair \(a_{ij}, a_{ji}\). We will sometimes draw an unlabeled edge between two vertices to indicate that we assume $a_{ij}a_{ji}\neq0$.  If we disregard the type of a simple root, we draw \(\square\) for its corresponding vertex.
	\end{definition}
	
	\begin{example}
		The Dynkin diagram of $\q(n)$ looks as follows:
		
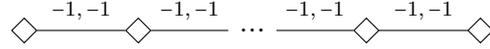
\begin{center}
			\begin{tikzpicture}
				\node[diamond,aspect=1,scale=0.6,draw] (A) at (-3,0) {};
				\node[diamond,aspect=1,scale=0.6,draw] (B) at (-1.5,0) {};
				\node (C) at (0,0) {...};
				\node[diamond,aspect=1,scale=0.6,draw] (D) at (1.5,0) {};
				\node[diamond,aspect=1,scale=0.6,draw] (E) at (3,0) {};
				\draw (A) edge node[above,font=\tiny] {\(-1,-1\)} (B);
				\draw (B) edge node[above,font=\tiny] {\(-1,-1\)} (C) ;
				\draw (C) edge node[above,font=\tiny] {\(-1,-1\)} (D) ;
				\draw (D) edge node[above,font=\tiny] {\(-1,-1\)} (E) ;
			\end{tikzpicture}
		\end{center}
		This is also the Dynkin diagram of $T\s\l(n)$.
	\end{example}
	
	\begin{remark}
		From the results of Section \ref{section-two-simple-roots} we see that two simple roots \(\alpha\) and \(\beta\) of a QKM algebra \(\mathfrak{g}(\mathcal{A})\) are connected in the associated Dynkin diagram \(S(A)\) if and only if \(\alpha+\beta\) is a root of \(\mathfrak{g}(\mathcal{A})\), that is, \(\mathfrak{g}_{\alpha+\beta}\neq 0\).
	\end{remark}
	
	\begin{definition}
		We say two distinct simple roots $\alpha,\beta$ are connected if the same is true in the Dynkin diagram.  Note that by the following lemma, this is equivalent to $\alpha_i(h_j)\alpha_j(h_i)=a_{ij}a_{ji}\neq0$.
	\end{definition}
	
		\begin{lemma}\label{lemma-QKM-notation}
		Suppose that $\AA$ is a QKM Cartan datum, and $\alpha_i,\alpha_j$ are distinct simple roots.  Then the following are equivalent:
		\begin{enumerate}
			\item $\alpha_i$ and $\alpha_j$ are connected;
			\item $\g_{\alpha_i+\alpha_j}\neq0$;
			\item $[\h_{\alpha_i},\g_{\alpha_j}]\neq0$;
			\item $[H_i,\g_{\alpha_j}]\neq0$;
			\item $x_{ij}$ and $y_{ij}$ are not both zero;
			\item $\alpha_j((\h_{\alpha_i})_{\ol{0}})\neq0$.
			\item $\alpha_{j}(h_i)\neq0$.
		\end{enumerate}
	\end{lemma}
	
	\begin{proof}
		The implications between all but the last condition a straightforward application of regularity, the equivalence of the conditions in the question at the start of Section \ref{section-connectivity}, and the fact that $H_i$ generates $\h_{\alpha_i}$ as an ideal of $\h$.  The final condition is equivalent by $\s\l(2)$-representation theory if $\alpha_i$ is of type $Tak(\s\l(2))$ or $Tak(\o\s\p(1|2))$, and for $Tak(\s\l(1|1))$ it is because $\alpha_j(c_i)=0$ for all roots $\alpha_j$ (see the table in Section \ref{section-notation-QKM}).
	\end{proof}
	
	\begin{remark}\label{remark_dynkin_diagram}
		Consider the map $\Theta$ from Dynkin diagrams of QKM algebras to Dynkin diagrams of Kac--Moody superalgebras, obtained by turning diamonds into circles (see \cite{K2} for more on Dynkin diagrams of Kac--Moody superalgebras).  We notice that if the Dynkin diagram $D$ of a QKM algebra \(\mathfrak{g}(\mathcal{A})\) has that $\Theta(D)$ is of infinite growth, then $\g(\AA)$ is also of infinite growth, by Section \ref{section_subdatum}.  
	\end{remark}
	
	\section{Two simple roots in a QKM algebra}\label{section-two-simple-roots}
	
	\begin{lemma}\label{lemma_coupled_defn_justif}
		Let $\alpha_i,\alpha_j$ be connected simple roots in a QKM algebra $\g(\AA)$.  Then either $x_{ij}x_{ji}\neq0$ or $y_{ij}y_{ji}\neq0$.
	\end{lemma}
	\begin{proof}
		Suppose that neither of the above two equations holds; because $\alpha_i$ and $\alpha_j$ are connected, we may assume WLOG that $x_{ij}=y_{ji}=0$ and $x_{ji}y_{ij}\neq0$. Then by (1) of Lemma \ref{lemma_QKM_formulas}, we learn that
		\[
		x_{ji}c_i=y_{ij}h_j.
		\]
		However since $\alpha_i(c_i)=0$, this implies $\alpha_i(h_j)=0$, contradicting Lemma \ref{lemma-QKM-notation}.
	\end{proof}
	
	\begin{definition}\label{definition_coupling}
		Let \(\mathfrak{g}(\mathcal{A})\) be a QKM algebra. We say two distinct simple roots \(\alpha_i,\alpha_j \in \Pi\) are coupled if they are connected and we have $\alpha_i(c_j)=\alpha_j(c_i)=0$.  Equivalently, by Lemma \ref{lemma_coupled_defn_justif}, either $x_{ij}=x_{ji}=0$, in which case we say $\alpha_i$ and $\alpha_j$ are $X$-coupled, or $y_{ij}=y_{ji}=0$, in which case we say $\alpha_i$ and $\alpha_j$ are $Y$-coupled.  
	\end{definition}
	
	\begin{remark}\label{remark_coroot_coupling}
		By (1) of Lemma \ref{lemma_QKM_formulas}, we have for two simple roots $\alpha_i$ and $\alpha_j$ the relation:
		\[
		x_{ij}c_j+y_{ij}h_j = x_{ji}c_i+y_{ji}h_i.
		\]
		Thus if $\alpha_i$ and $\alpha_j$ are coupled, the above formula gives a nontrivial relation between one pure even coroot on each side.  In particular, if they are $X$-coupled we obtain 
		\[
		y_{ij}h_j=y_{ji}h_i, \ \ \ \ y_{ij}y_{ji}\neq0,
		\]
		and if they are $Y$-coupled we obtain
		\[
		x_{ij}c_j=x_{ji}c_i, \ \ \ \ x_{ij}x_{ji}\neq0.	 
		\]
	\end{remark}
	
	\begin{example}
		The prototypical example of a QKM algebra with coupled simple roots is the extension of the Takiff superalgebra $T\s$ for any Kac--Moody superalgebra $\s$.  Indeed, in this case we have $Y=0$, i.e.~all connected simple roots are $Y$-coupled.
	\end{example}
	
	\begin{example}
		For $\g=\q(n)$, no distinct simple roots are coupled with one another.
	\end{example}
	
	The aim for this section is to prove the following theorem:
	
	\begin{theorem}\label{theorem_coupling}
		Suppose that $\AA$ is QKM and $\alpha_i,\alpha_j\in\Pi$ are distinct simple roots.  Then one of the following cases must occur:
		\begin{enumerate}
			\item $\alpha_i$ and $\alpha_j$ are not connected;
			\item $\alpha_i$ and $\alpha_j$ are $Y$-coupled;
			\item $\alpha_i$ and $\alpha_j$ are $X$-coupled and one of the following holds (up to swapping indices):
			\begin{enumerate}
				\item $\alpha_i$ and $\alpha_j$ are of type $Tak(\s\l(2))$ and $\alpha_i(h_j)\alpha_j(h_i)=4$; or,
				\item $\alpha_i$ is of type $Tak(\s\l(2))$, $\alpha_j$ is of type $Tak(\o\s\p(1|2))$, and we have \linebreak $\alpha_j(h_i)=-2$ and $\alpha_i(h_j)=-1$;
			\end{enumerate}
			\item $\alpha_i$ and $\alpha_j$ are not coupled, both of type $Tak(\s\l(2))$, and $y_{ij}y_{ji}\neq0$.
		\end{enumerate}
	\end{theorem}
	
	Note that the following is false for the entries of the $X$-matrix (in particular it fails for $\q^{-}_{(2,2)}$, see Section \ref{section_q_22}). 
	\begin{corollary}\label{corollary_ys}
		If $\AA$ is QKM, then $y_{ij}=0\Rightarrow y_{ji}=0$.
	\end{corollary}
	\begin{proof}
		This follows immediately from Theorem \ref{theorem_coupling}.
	\end{proof}
	
	The proof of Theorem \ref{theorem_coupling} will occupy the rest of the section, and we will also prove auxiliary results which will be important in future sections.  We begin by dealing with the case of $Tak(\s\l(1|1))$:
	
	\begin{lemma}
		Suppose that $\alpha_i$ is of type $Tak(\s\l(1|1))$ and is connected to $\alpha_j$.  Then $\alpha_i$ and $\alpha_j$ are $Y$-coupled.		
	\end{lemma}
	\begin{proof}
		Suppose that $\alpha_i$ and $\alpha_j$ are connected, so that $\alpha_i(h_j)\alpha_j(h_i)\neq0$.  If $\alpha_i$ is of type $Tak(\s\l(1|1))$, then $H_i^2=0$, implying that $x_{ij}y_{ij}=0$.  Suppose that $x_{ij}=0$, so that (1) of \ref{lemma_QKM_formulas} gives
		\[
		y_{ij}h_j=x_{ji}c_i+y_{ji}h_i.
		\]
		Applying $\alpha_i$ to the RHS we obtain 0, implying that $\alpha_i(h_j)=0$, a contradiction.  Thus we must have $y_{ij}=0$, implying that 
		\[
		x_{ij}c_j=x_{ji}c_i+y_{ji}h_i.
		\]
		If we apply $\alpha_j$, we know from the table in Section \ref{section-notation-QKM} that $\alpha_j(c_j)=\alpha_j(c_i)=0$, so we obtain $y_{ji}\alpha_j(h_i)=0$; since $\alpha_j(h_i)\neq0$, this forces $y_{ji}=0$.  Thus we have proven that $\alpha_i$ is $Y$-coupled to $\alpha_j$.
	\end{proof}

	\begin{lemma}\label{lemma-sq2-rep}
		Let \(\mathfrak{g}(\mathcal{A})\) be a QKM algebra and let \(\alpha \in \Pi\) be of type $Tak(\s\l(2))$. Let \(e,E,f,F,h,c,H\) be the spanning set for $\g\langle\alpha\rangle\cong(\p)\mathfrak{sq}(2)$ as defined in Section \ref{section-notation-QKM}. Let \(\beta \in \Pi\) be a simple root different from \(\alpha\). Then:
		\begin{enumerate}
			\item $\g_{\beta-\beta(h)\alpha}\neq0$, and $\g_{\beta+(1-\beta(h))\alpha}=0$;
			\item  if \(\beta(h)=-1\) and \(v \in \mathfrak{g}_{\beta}\), then 
			\[
			[E,v] + [e,[H,v]] =0.
			\]
		\end{enumerate}
	\end{lemma}
	\begin{proof}
		The first statement follows from the integrability of the action of $\s\l(2)$.  Now suppose that \(\beta(h) = -1\). Then $\g_{\beta+2\alpha}=0$ so that \([e,[E,v]] = 0\) . Therefore
		\begin{align*}
			0 & = -[f,[e,[E,v]]] = -[[f,e],[E,v]] - [e,[[f,E],v]] \\ &
			= [h,[E,v]] + [e,[H,v]] = [E,v] + [e,[H,v]].
		\end{align*}
	\end{proof}

	\begin{proposition}\label{propistion-two-sq2}
		Let \(\mathfrak{g}(\mathcal{A})\) be a QKM algebra and let \(\alpha_i, \alpha_j \in \Pi\) be two simple roots type $Tak(\s\l(2))$. If \(\alpha_i\) and \(\alpha_j\) are connected, then exactly one of the following happens:
		\begin{enumerate}[label=(\roman*)]
			\item $\alpha_i$ and $\alpha_j$ are coupled, in which case either:
			\begin{enumerate}
				\item $\alpha_i$ and $\alpha_j$ are $Y$-coupled; or
				\item $\alpha_i$ and $\alpha_j$ are $X$-coupled and $\alpha_i(h_j)\alpha_j(h_i)=4$.
			\end{enumerate}
			\item \(\alpha_i(h_j) = \alpha_j(h_i) = -1\) and \(\alpha_i(c_j)\alpha_j(c_i) \neq 0\);
			\item \(\alpha_i(h_j)=\alpha_j(h_i) = -2\),  and exactly one of \(x_{ij}\), \(x_{ji}\) is zero;
			\item \(\alpha_i(h_j),\alpha_j(h_i) \in \mathbb{Z}_{\leq -2}\) and \(\alpha_i(c_i)\alpha_j(c_i) \neq 0\).
		\end{enumerate}
	\end{proposition}
	
	\begin{proof}
		Because we assume $\alpha_i,\alpha_j$ are connected, we have $\alpha_i(h_j)\alpha_j(h_i)\neq0$.  
		
		Now because $\alpha_i(h_i)=\alpha_j(h_j)$, Lemma \ref{lemma_nontrivial_relation_on_ys} implies that:
		\begin{align}
			y_{ij}^2\alpha_i(h_j) = y_{ji}^2\alpha_j(h_i).
		\end{align}
		So if \(y_{ij}y_{ji} = 0\), it must be that both \(y_{ij}=y_{ji} = 0\), implying \(\alpha_i\) and \(\alpha_j\) are $Y$-coupled.
		
		If, on the other hand, \(x_{ij}x_{ji} = 0\), then Lemma \ref{lemma_nontrivial_relation_on_ys} gives:
		\begin{align*}
			2y_{ji} = y_{ij}\alpha_i(h_j), \\
			2y_{ij} = y_{ji}\alpha_j(h_i).
		\end{align*}
		Thus \(y_{ij},y_{ji}\) is a non-trivial solution to the linear system defined by the matrix 
		\[
		\begin{bmatrix} 2 & -\alpha_i(h_j)\\ -\alpha_j(h_i) & 2\end{bmatrix},
		\]
		implying that the determinant is 0, which gives
		\[
		\alpha_i(h_j)\alpha_j(h_i) = 4.
		\]
		In particular, this shows that if $\alpha_i$ and $\alpha_j$ are $X$-coupled, then $\alpha_i(h_j)\alpha_j(h_i)=4$, as desired. On the other hand, if \(\alpha_i\) and \(\alpha_j\) are not coupled but we still have $x_{ij}x_{ji}=0$, then exactly one of \(x_{ij}\), \(x_{ji}\) is zero, and we have $\alpha_i(h_j)\alpha_j(h_i)=4$.  It remains to show then that $\alpha_i(h_j)=\alpha_j(h_i)=-2$, or equivalently that $\alpha_i(h_j)\neq-1$ and $\alpha_j(h_i)\neq-1$, which will follow from our final argument.
		
		Thus suppose that \(\alpha_i(h_j) = -1\) and \(\alpha_i\), \(\alpha_j\) are not coupled. Then Lemma \ref{lemma-sq2-rep} implies \([E_j,e_i] + x_{ji}[e_j,E_i] = 0\). As a result
		\begin{align*}
			0 = [f_i,[E_j,e_i] + x_{ji}[e_j,E_i]] = [E_j,-h_i]+ x_{ji}[e_j,-H_i] = (\alpha_j(h_i)+x_{ji}x_{ij})E_j,
		\end{align*}
		and
		\begin{align*}
			0 = [F_i,[E_j,e_i] + x_{ji}[e_j,E_i]] = [E_j,H_i] + x_{ji}[e_j,c_i] = (y_{ij}-x_{ij}x_{ji}y_{ij})e_j.
		\end{align*}
		Because \(\alpha_i\) and \(\alpha_j\) are not coupled, we have \(y_{ij} \neq 0\), so \(x_{ij}x_{ji}=1\) and \(\alpha_j(h_i) = -1\).  This completes our proof.
	\end{proof}
	
	\begin{proposition}\label{prop_coupling_sl_osp}
		Let \(\mathfrak{g}(\mathcal{A})\) be a QKM algebra. Let \(\alpha_i \in \Pi\) be of $Tak(\s\l(2))$-type and \(\alpha_j \in \Pi\) be of $Tak(\mathfrak{osp}(1|2))$-type. If $\alpha_i$ is connected to $\alpha_j$, then either they are Y-coupled, or they are $X$-coupled and we have $\alpha_j(h_i)=-2$ and $\alpha_i(h_j)=-1$.
	\end{proposition}
	
	\begin{proof}
		First of all, we have \(\alpha_i(c_j) = 0\) by the table from Section \ref{section-notation-QKM}.  In particular $x_{ji}y_{ji}=0$.
		
		We may consider the \(\s\q(2)\)-subquotient of $\g\langle\alpha_j\rangle$ generated by $\g_{2\alpha_j}$ and $\g_{-2\alpha_j}$.  Then Proposition \ref{propistion-two-sq2} tells us that $y_{ji}=0$ implies $y_{ij}=0$ as well, so that $\alpha_i$ and $\alpha_j$ are $Y$-coupled.  
		
		Let us assume they are not $Y$-coupled, so that $x_{ji}=0$, $y_{ij}y_{ji}\neq0$.  Then Proposition \ref{propistion-two-sq2} tells us that $\alpha_i(h_j)(2\alpha_j)(h_i)=4$.  
		
		If $\alpha_j(h_i)=-1$, then Lemma \ref{lemma-sq2-rep} then gives \([E_i,e_j] + [e_i,[H_i,e_j]] = 0\), thus
		\begin{align}\label{equation-contradicting-tosp}
			0 = [F_j,[E_i,e_j] + [e_i,[H_i,e_j]]] = -[E_i,H_j] + [e_i,[F_j,[H_i,e_j]]].
		\end{align} 
		As \([H_i,e_j] \in \mathbb{C}\langle E_j\rangle\), we have \([F_j,[H_i,e_j]] \in \mathbb{C}\langle c_j\rangle \), so \([e_i,[F_j,[H_i,e_j]]] = 0\). Together with \eqref{equation-contradicting-tosp}, we obtain
		\begin{align*}
			0 = [E_i,H_j] = y_{ji}e_i \neq 0,
		\end{align*}
		a contradiction. Thus we must have $\alpha_j(h_i)=-2$ and $\alpha_i(h_j)=-1$.  Now if $\alpha_i$ and $\alpha_j$ were not $X$-coupled, by then Proposition \ref{propistion-two-sq2} we would have $2\alpha_j(h_i)=\alpha_i(h_j)=-2$, which is not the case, meaning instead they must be $X$-coupled, and we are done.
	\end{proof}
	
	We now finish the proof of Theorem \ref{theorem_coupling} with the following proposition.
	\begin{proposition}
		If $\alpha_i$ and $\alpha_j$ are both of type $Tak(\o\s\p(1|2))$ and are connected, then they cannot be $X$-coupled (in particular they must be $Y$-coupled).
	\end{proposition}
	\begin{proof}
		Suppose that $\alpha_i$ and $\alpha_j$ are $X$-coupled.   By considering $2\alpha_i$ and $\alpha_j$ as $X$-coupled simple roots, Proposition \ref{prop_coupling_sl_osp} tells us that we have $\alpha_j(h_i)=-2$ and $(2\alpha_i)(h_j)=-1$. However this implies that $\alpha_i(h_j)=-1/2$, which is impossible.
	\end{proof}
	
	\section{Completely coupled vs.~completely uncoupled}\label{section_coupled_v_uncoupled}
	
	It is interesting to understand what are the possible configurations of simple roots of a QKM algebra. If we consider QKM algebras of finite growth, the coupling property is very well behaved. First, a definition.
	
	\begin{definition}
		Let $\AA$ a be QKM Cartan datum with matrices $X=X(\AA)$ and $Y=Y(\AA)$ as in Section \ref{section-notation-QKM}.  We say that $\AA$, and also $\g(\AA)$, is 
		\begin{enumerate}
			\item completely coupled if any two simple roots that are connected are coupled;
			\item completely $X$-coupled if $x_{ij}=0$ for all $i\neq j$;
			\item completely $Y$-coupled if $Y=0$, i.e.~$y_{ij}=0$ for all $i,j$;
			\item completely uncoupled if for all $i\neq j$, either $\alpha_i$ and $\alpha_j$ are not connected, or they are connected and uncoupled.
		\end{enumerate}
	\end{definition}
	
	The main theorem to be proven in this section is:
	\begin{theorem}
		Suppose that $\g(\AA)$ is an indecomposable, finite growth QKM algebra.  Then $\AA$ is either completely $X$-coupled, completely $Y$-coupled, or completely uncoupled. Further, their Cartan data are completely classified up to isomorphism.
	\end{theorem}  
	
	\subsection{$X$-couplings}
	
	\begin{lemma}\label{lemma_X_coupling_isolated}
		Suppose that $\AA$ is an indecomposable QKM Cartan datum. Then if $\alpha_i,\alpha_j$ are distinct, connected, $X$-coupled simple roots, then $\Pi=\{\alpha_i,\alpha_j\}$, i.e.~there are no other simple roots.
	\end{lemma}
	\begin{proof}
		By Theorem \ref{theorem_coupling}, our assumption implies that $\alpha_k(h_k)=1$ or $2$ for $k=i,j$.  On the other hand, since $x_{ij}=x_{ji}=0$ by assumption, (1) of Lemma \ref{lemma_QKM_formulas} implies that $h_i=\lambda h_j$ for $\lambda=y_{ij}/y_{ji}\neq0$.  We see that
		\[
		0<\alpha_i(h_i)=\lambda\alpha_i(h_j).
		\]
		Since $\alpha_i(h_j)\in\Z_{<0}$, this forces $\lambda\in\R_{<0}$.  If $\gamma$ were any other simple root in $\AA$, then we must have $\gamma(h_i),\gamma(h_j)\leq0$; on the other hand this means that $0\leq \lambda\gamma(h_j)=\gamma(h_i)$.  Thus $\gamma(h_i)=0$, and similarly $\gamma(h_j)=0$.  Thus $\gamma$ is not connected to either $\alpha_i$ or $\alpha_j$, which by indecomposability implies that $\Pi=\{\alpha_i,\alpha_j\}$, and we are done.
	\end{proof}

\begin{theorem}\label{theorem_X_coupled}
		If $\g(\AA)$ is an indecomposable, completely $X$-coupled QKM algebra, then its Cartan datum satisfies one of the following up to isomorphism, and $\g(\AA)$ is finite growth:
		\begin{enumerate}
			\item $\Pi=\{\alpha_1,\alpha_2\}$ with $\alpha_1,\alpha_2$ both of type $Tak(\s\l(2))$, and $h_1=-h_2$.  We have 
			\[
			X=\begin{bmatrix}
				2 & 0\\
				0 & 2
			\end{bmatrix}, \ \ \ \ 
			Y=\begin{bmatrix}
				0 & -1 \\ 1 & 0
			\end{bmatrix}.
			\]
			The Dynkin diagram is given by:
			\begin{center}
				\begin{tikzpicture}
					\node[diamond,aspect=1,scale=0.6,draw] (A) at (-0.5,0) {};
					\node[diamond,aspect=1,scale=0.6,draw] (B) at (0.5,0) {};
					\draw (A) edge node[above,font=\tiny] {\(-2,-2\)} (B);
				\end{tikzpicture}
			\end{center}
			\item $\Pi=\{\alpha_1,\alpha_2\}$ with $\alpha_1,\alpha_2$ both of type $Tak(\s\l(2))$, and $h_1=-2h_2$.  We have 
			\[
			X=\begin{bmatrix}
				2 & 0\\
				0 & 2
			\end{bmatrix}, \ \ \ \ 
			Y=\begin{bmatrix}
				0 & -2\\ 1 & 0
			\end{bmatrix}.
			\]
			The Dynkin diagram is given by:
			\begin{center}
				\begin{tikzpicture}
					\node[diamond,aspect=1,scale=0.6,draw] (A) at (-0.5,0) {};
					\node[diamond,aspect=1,scale=0.6,draw] (B) at (0.5,0) {};
					\draw (A) edge node[above,font=\tiny] {\(-4,-1\)} (B);
				\end{tikzpicture}
			\end{center}
			\item $\Pi=\{\alpha_1,\alpha_2\}$ with $\alpha_1$ of type $Tak(\s\l(2))$, $\alpha_2$ of type $Tak(\o\s\p(1|2))$, and $h_1=-2h_2$.  We have 
			\[
			X=\begin{bmatrix}
				2 & 0\\
				0 & 1
			\end{bmatrix}, \ \ \ \ 
			Y=\begin{bmatrix}
				0 & -2\\ 1 & 0
			\end{bmatrix}.
			\]
			The Dynkin diagram is given by:
			\begin{center}
				\begin{tikzpicture}
					\node[diamond,aspect=1,scale=0.6,draw] (A) at (-0.5,0) {};
					\node[diamond,aspect=1,scale=0.6,draw,fill=black] (B) at (0.5,0) {};
					\draw (A) edge node[above,font=\tiny] {\(-2,-1\)} (B);
				\end{tikzpicture}
			\end{center}
		\end{enumerate}
	\end{theorem}
    
	\begin{proof}[Proof of Theorem \ref{theorem_X_coupled}]
		It remains to check that the only possible Cartan data are the ones listed.  However this follows from (3) of Theorem \ref{theorem_coupling}, which tells us the possible simple root types and the values of $\alpha_i(h_j)$, and we can rescale $y_{ij}$ and $y_{ji}$ to make it as above, by rescaling our choice of generators (see Remark \ref{remark_rescaling}).  The relations on $h_1$ and $h_2$ come from (1) of Lemma \ref{lemma_QKM_formulas}.
	\end{proof}

	\subsection{Complete coupling for finite growth QKM algebras}	
	\begin{lemma}
		Let \(\mathfrak{g}(\mathcal{A})\) be a QKM algebra. Let \(\alpha_1\), \(\alpha_2\) and \(\alpha_3\) be simple roots, such that their corresponding full subdiagram of the Dynkin diagram of \(\mathfrak{g}(\mathcal{A})\) is:
		\begin{center}
			\begin{tikzpicture}[square/.style={regular polygon,regular polygon sides=4}]
				\node[square,aspect=1,scale=0.6,draw] (A) at (0,0) {};
				\node[square,aspect=1,scale=0.6,draw] (B) at (180:1cm) {};
				\node[square,aspect=1,scale=0.6,draw] (C) at (0:1cm) {};
				\draw (A) -- (B);
				\draw (A) -- (C);
			\end{tikzpicture}
		\end{center}
        (Here, each symbol $\square$ could be either \begin{tikzpicture}	\node[diamond,aspect=1,scale=0.6,draw] (A) {};	\end{tikzpicture}, \begin{tikzpicture}	\node[diamond,aspect=1,scale=0.6,draw,fill=black] (A) {};	\end{tikzpicture}, or $\diamondtimes$.)
		If \(\alpha_1\) and \(\alpha_2\) are coupled, then \(\alpha_2\) and \(\alpha_3\) are coupled. 
	\end{lemma}
	\begin{proof}		
		Since we may assume by Lemma \ref{lemma_X_coupling_isolated} that $\alpha_1$ and $\alpha_2$ are $Y$-coupled, (1) of Lemma \ref{lemma_QKM_formulas} implies that 
		\[
		x_{12}c_2=x_{21}c_1,
		\]
		where $x_{12},x_{21}\neq0$.  Thus we obtain that $\alpha_1(c_2)=0$, since the same is true of $c_1$. Applying $\alpha_1$ to the formula
		\[
		x_{23}c_3+y_{23}h_3=x_{32}c_2+y_{32}h_2
		\]
		gives 
		\[
		0=x_{32}\alpha_1(c_2)+y_{32}\alpha_1(h_2)=y_{32}\alpha_1(h_2).
		\]
		Since $\alpha_1(h_2)\neq0$, this forces $y_{32}=0$, so by Lemma \ref{lemma_coupled_defn_justif} we obtain $y_{23}=0$ also, and we are done.
	\end{proof}
	
	\begin{lemma}
		Let \(\mathfrak{g}(\mathcal{A})\) be a QKM algebra of finite growth. Let \(\alpha_1\), \(\alpha_2\) and \(\alpha_3\) be simple roots, such that their corresponding full subdiagram of the Dynkin diagram of \(\mathfrak{g}(\mathcal{A})\) is:		
		
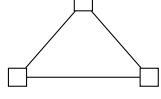
\begin{center}
			\begin{tikzpicture}[square/.style={regular polygon,regular polygon sides=4}]
				\node[square,scale=0.6,draw] (A) at (90:0.5cm) {};
				\node[square,scale=0.6,draw] (B) at (210:1cm) {};
				\node[square,scale=0.6,draw] (C) at (-30:1cm) {};
				\draw (A) edge node[left,font=\tiny] {} (B);
				\draw (B) edge node[below,font=\tiny] {} (C);
				\draw (C) edge node[right,font=\tiny] {} (A);
			\end{tikzpicture}
		\end{center}
		If \(\alpha_1\) and \(\alpha_2\) are coupled, then \(\alpha_3\) is coupled with \(\alpha_1\) and \(\alpha_2\).
	\end{lemma}
	
	\begin{proof}
		We again may assume by Lemma \ref{lemma_X_coupling_isolated} that any coupling that occurs here is $Y$-coupling.
		
		If $\alpha_3$ is of type $Tak(\s\l(1|1))$ or $Tak(\o\s\p(1|2))$, then we are done by Theorem \ref{theorem_coupling}.  If $\alpha_1$ is of type $Tak(\s\l(1|1))$ or $Tak(\o\s\p(1|2))$, then it must be $Y$-coupled to $\alpha_2$ and $\alpha_3$, so by Remark \ref{remark_coroot_coupling} we learn that $c_2$ and $c_3$ are proportional.  This implies that $\alpha_2(c_3)=\alpha_3(c_2)=0$, so $\alpha_2$ and $\alpha_3$ are coupled.  Similar considerations apply if $\alpha_2$ is of type $Tak(\s\l(1|1))$ or $Tak(\o\s\p(1|2))$.
		
		Thus we may assume that \(\alpha_1\), \(\alpha_2\) and \(\alpha_3\) are all of type $Tak(\s\l(2))$, and we know that $y_{12}=y_{21}=0$, since they are $Y$-coupled.  If either $\alpha_1$ or $\alpha_2$ were coupled to $\alpha_3$, then we would have that $c_1$ is a nonzero multiple of $c_2$ is a nonzero multiple of $c_3$, meaning that $\alpha_i(c_j)=0$ for all $i,j$, so all roots are coupled.  
		
		Thus we assume $\alpha_3$ is not coupled to either $\alpha_1$ or $\alpha_2$; in particular $y_{13}y_{31}y_{23}y_{32}\neq0$.  Now suppose that \(x_{13} = 0\); then (1) of Lemma \ref{lemma_QKM_formulas} gives 
		\begin{align*}
			y_{13}h_3 = x_{31}c_1 + y_{31}h_1.
		\end{align*}
		Since \(\alpha_2(c_1)=\alpha_1(c_1) = 0\), so evaluating \(\alpha_1\) and \(\alpha_2\) on the above equation, we obtain
		\begin{align*}
			&y_{13}a_{31} = 2y_{31}, \\
			&y_{13}a_{32} = y_{31}a_{12}.
		\end{align*}
		However this implies $a_{31}/a_{32}=2/a_{12}$, which is impossible, as \(a_{31},a_{32},a_{12} < 0\). Thus we must have $x_{13}\neq0$, and for similar reasons $x_{23}\neq0$.
		
		By (1) or Lemma \ref{lemma_QKM_formulas} we have \([H_1,H_2] = x_{21}c_1 = x_{12}c_2\), and by (2) of the same lemma we obtain:
		\begin{align*}
			x_{13}y_{23}+x_{23}y_{13} = x_{21}x_{13}y_{13} = x_{12}x_{23}y_{23}.
		\end{align*}
		We may rescale \(E_1\) and \(E_2\) such that \(y_{13}=y_{23}= 1\), so from the above equation we obtain
		\begin{align*}
			x_{13}(1-x_{21})+x_{23} = x_{13} + x_{23}(1-x_{12}) = 0.
		\end{align*}
		As \(x_{13}x_{23} \neq 0\),  the above equations imply that \((1-x_{21})(1-x_{12}) = 1\), so 
		\begin{align}\label{equation-x-relations}
			x_{12}+x_{21} = x_{12}x_{21}.
		\end{align}
		We again use equations (1) and (2) of Lemma \ref{lemma_QKM_formulas} to deduce that
		\begin{align*}
			& x_{12}y_{32} = \alpha_2([H_1,H_3]) = y_{31}a_{12}, \\
			& x_{21}y_{31} = \alpha_1([H_2,H_3]) = y_{32}a_{21}.
		\end{align*}
		From these equations we obtain \(x_{12}x_{21} = a_{12}a_{21}\) and \(\frac{x_{12}}{x_{21}} = \frac{y_{31}^2a_{12}}{y_{32}^2a_{21}}\). As in the proof of Proposition \ref{propistion-two-sq2}, we must have \(y_{31}^2a_{13} = y_{13}^2a_{31} = a_{31}\) and \(y_{32}^2a_{23} = y_{23}^2a_{32} = a_{32}\), so
		\begin{align*}
			\frac{x_{12}}{x_{21}} = \frac{a_{12}a_{23}a_{31}}{a_{21}a_{13}a_{32}} > 0. 
		\end{align*}
		Squaring \eqref{equation-x-relations} and dividing by \(x_{12}x_{21}\), we get
		\begin{align*}
			\frac{x_{12}}{x_{21}} + \frac{x_{21}}{x_{12}} + 2 = x_{12}x_{21}
		\end{align*}
		But \(\frac{x_{12}}{x_{21}} > 0\) and \(x_{12}x_{21} = a_{12}a_{21}\), so \(a_{12}a_{21} > 2\). It is now easy to see that $A$ is not a finite type Cartan matrix, meaning \(\mathfrak{g}(\mathcal{A})\) is not of finite growth by Remark \ref{remark_dynkin_diagram}.	
	\end{proof}
	
	We immediately obtain the following corollary:
	
	\begin{corollary}
		If \(\mathfrak{g}(\mathcal{A})\) is an indecomposable QKM algebra of finite growth, then it is either completely coupled or completely uncoupled.
	\end{corollary}
	
	\subsection{$Y$-couplings: realization via the Takiff construction}
	
	\begin{theorem}\label{theorem-Takiffs}
		Let \(\mathfrak{g}(\mathcal{A})\) be an indecomposable, completely $Y$-coupled QKM algebra (i.e.~$Y=0$).
		Let \(A=(a_{ij})=(\alpha_j(h_i))\) be the Cartan matrix of \(\mathfrak{g}(\mathcal{A})\), and set \(\s=\mathfrak{s}(A)\) to be the corresponding contragredient, Kac--Moody Lie superalgebra, with root lattice $Q_{\s}$.  Finally, let  \(T\mathfrak{s}\) be as in Example \ref{example-takiff}. Set
		\[
		J_{\AA}:=\bigcap\limits_{i=1}^{n}\operatorname{Ann}_{\h}\g_{\alpha_i},  \ \ \ \ J_{\s}:=\bigcap\limits_{\alpha\in Q_{\s}}\operatorname{Ann}_{T\s}(\s_{\alpha}\otimes\C[\xi]).
		\] 
		Then:
		\begin{enumerate}
			\item Up to rescaling of generators (see Remark \ref{remark_rescaling}), we have an equality of matrices $X=A$;
			\item we have an embedding
			\[
			\mathfrak{g}(\mathcal{A})/J_\AA\hookrightarrow T\mathfrak{s}(A)/J_{\mathfrak{s}},
			\]
			which is an isomorphism on all root spaces.
		\end{enumerate}
	\end{theorem}
	
	Thus Theorem \ref{theorem-Takiffs} says if $Y=0$, $\g(\AA)$ is, up to extensions and derivations in $\h$, a Takiff superalgebra.
	
	\begin{proof}[Proof of Theorem \ref{theorem-Takiffs}]	
	
		Fix \(\alpha_i\) a simple root of \(\mathfrak{g}(\mathcal{A})\).  We first show that after rescaling $E_1,\dots,E_n$, we obtain that $X=A$.
		
		The irreducibility of \(\mathfrak{g}_{\alpha_i}\) as an \(\mathfrak{h}\)-module implies the existence of \(I \in \mathfrak{h}_{\ol{1}}\) such that \([I,E_i] = e_i\). For each \(j\), let \(y_j \in \mathbb{C}\) be the scalar satisfying \([I,E_j] = y_je_j\). We compute:
		\begin{align*}
			x_{ij}y_j = \alpha_j([I,H_i]) = \alpha_j(h_i) = a_{ij}
		\end{align*}
		Thus it suffices to show that we can rescale $E_1,\dots,E_n$ so that $y_j=1$ for all $j$; or equivalently that $y_j\neq0$ for all $j$.  However the above formula clearly shows that if \(\alpha_j\) is connected to \(\alpha_i\), then $y_i\neq0\Rightarrow y_j\neq0$.  Since the Dynkin diagram is connected, we are done, and we obtain, after rescaling, $x_{ij}=a_{ij}$ for all $i,j$.
		
		For any \(H \in \mathfrak{h}_{\ol{1}}\) define the complex numbers \(\langle H , e_i \rangle \) and \(\langle H , E_i \rangle\) by the formula:
		\[
		[H,e_i+E_i] = \langle H , e_i \rangle E_i + \langle H , E_i \rangle e_i.
		\]
		Repeating the argument given for  $I$ and the fact that $X=A$, we have \(\langle H , E_i \rangle = \langle H , E_j \rangle\) for any \(i,j\). 
		
		Write $\mathfrak{t}_{\s}$ for the Cartan subalgebra of $\s$, and let $t_1,\dots,t_n\in\mathfrak{t}_{\s}$ be such that $\alpha_i(t_j)=\delta_{ij}$.  Define a map \(\tilde{\phi} : \tilde{\mathfrak{g}}(\mathcal{A})\rightarrow T\mathfrak{s}(A)/J_\mathfrak{s}\) as follows:
		\begin{align*}
			e_i & \mapsto \tilde{e}_i \\
			f_i & \mapsto \tilde{f}_i \\
			E_i & \mapsto \tilde{e}_i\otimes \xi\\
			F_i & \mapsto \tilde{f}_i\otimes \xi \\
			\mathfrak{t}\ni h& \mapsto \sum\limits_{i}\alpha_i(h)t_i \\
			\h_{\ol{1}}\ni H & \mapsto \langle H,E_1 \rangle\partial_\xi + \sum_{i=1}^n \langle H,e_i \rangle t_i\otimes\xi
		\end{align*}
		A straightforward verification shows that the above assignment determines a well defined homomorphism of Lie superalgebras. We further notice that \((\tilde{\phi}(\tilde{\mathfrak{g}}(\mathcal{A})))_{\tilde{\alpha}} = (T\mathfrak{s}(A)/Z_\mathfrak{s})_{\tilde{\alpha}}\) for any nonzero  \(\tilde{\alpha} \in \mathbb{Z}\langle \tilde{\alpha}_1,...,\tilde{\alpha}_n \rangle\). As \(T\mathfrak{s}(A)/J_\mathfrak{s}\) does not admit an ideal intersecting its Cartan subalgebra trivially, we see that the map \(\tilde{\phi}\) factors through a map \(\phi : \mathfrak{g}(\mathcal{A}) \rightarrow T\mathfrak{s}(A)/J_\mathfrak{s}\), which is an isomorphism on all root spaces. \par
		Finally, it is easy to check that \(\phi(J_\mathcal{A}) = 0\), so that \(\phi\) factors through a map \newline \(\mathfrak{g}(\mathcal{A})/J_{\mathcal{A}}\rightarrow T\mathfrak{s}(A)/J_\mathfrak{s}\), which is easily checked to be injective.
	\end{proof}

	\section{Completely uncoupled QKM algebras}
	
	In this section, we consider the structure of completely uncoupled QKM algebras \(\mathfrak{g}(\mathcal{A})\), by which we mean that \(\alpha\) and \(\beta\) are uncoupled for any distinct \(\alpha,\beta\in \Pi\).   In particular, we will be interested in the possible Dynkin diagrams, and a classification of those with finite growth.  By Theorem \ref{theorem_coupling}, we may assume all simple roots are of $Tak(\s\l(2))$-type. 	
	
	One should view this as the more nontrivial setting, that is far away from Takiff superalgebras. We will see in fact that being completely uncoupled makes a QKM algebra very rigid.
	
	We begin with a lemma stating some of the properties we will use of completely uncoupled algebras.
	\begin{lemma}\label{lemma_uncoupled_facts}
		Let $\AA$ be a completely uncoupled QKM Cartan datum.  Then:
		\begin{enumerate}
			\item any full Cartan subdatum (see Section \ref{section_subdatum}) $\AA'$ of a completely uncoupled QKM Cartan datum $\AA$ is also completely uncoupled;
			\item two distinct simple roots $\alpha_i$ and $\alpha_j$  are connected if and only if $[H_i,H_j]\neq0$;
			\item we have, for all $i,j$, $y_{ij}y_{ji}\neq0$, and:
			\[
			y_{ij}^2\alpha_i(h_j)=y_{ji}^2\alpha_j(h_i).
			\]
		\end{enumerate} 
	\end{lemma}
	\begin{proof}
		Part (1) is clear, and (3) is an immediate application of Lemma \ref{lemma_nontrivial_relation_on_ys} and Corollary \ref{corollary_ys}.
		
		For part (2), the backward direction is obvious.  For the forward direction, (1) of Lemma \ref{lemma_QKM_formulas} tells us that 
		\[
		[H_i,H_j]=x_{ij}c_j+y_{ij}h_j.
		\]
		However we must have $y_{ij}\neq0$, and since $h_j$ cannot be a multiple of $c_j$, we must have $[H_i,H_j]\neq0$.
	\end{proof}

	\begin{theorem}\label{theorem-root-system}
		Let \(\mathfrak{g}(\mathcal{A})\) be a completely uncoupled QKM algebra and let \(A\) be its Cartan matrix. Let \(\mathfrak{s} = \mathfrak{s}(A)\) be the (central quotient of the) Kac--Moody Lie algebra corresponding to \(A\), with Cartan subalgebra \(\mathfrak{t}\) \footnote{this is a Kac--Moody Lie algebra in the sense of \cite{GHS}, with the condition on the dimension of the Cartan subalgebra \(\mathfrak{t}\) more relaxed.}, set of simple roots \(\Pi\) and simple coroots \(h_1,...,h_n\). Then the root system \(\Delta\) of \(\mathfrak{g}(\mathcal{A})\) and the root system \(\Delta_{\mathfrak{s}}\) of \(\mathfrak{s}\) coincide (as subsets of \(\mathfrak{t}^*\)). Moreover, \(\mathfrak{g}(\mathcal{A})\) is finite-dimensional if and only if \(\Delta\) is finite.
	\end{theorem}
	\begin{proof}
		It is immediate that \(\Delta_\mathfrak{s} \subseteq \Delta\). Suppose \(\Delta_\mathfrak{s} \neq \Delta\); then the Chevalley automorphism on \(\mathfrak{g}\) (see Remark \ref{remark-Chevalley}) implies that \(\Delta^+ \setminus \Delta_{\mathfrak{s}} \neq \emptyset\). Let \(\beta \in \Delta^+ \setminus \Delta_{\mathfrak{s}}\) be of minimal height. Then \(\beta\) is not proportional to any simple root. Since \(\beta\) is a root of \(\mathfrak{g}(\mathcal{A})\), there exists \(\alpha \in \Pi\) such that \(\beta - \alpha \in \Delta^+\), so the minimality of the height of \(\beta\) implies \(\beta - \alpha \in \Delta_\mathfrak{s}^+\). Let \(\mathfrak{s_{\alpha}} \subseteq \mathfrak{s}\) be the \(\s\l(2)\) subalgebra corresponding to \(\alpha\). From the integrability of \(\mathfrak{g}(\mathcal{A})\) we obtain that \(\bigoplus_{k\in\mathbb{Z}} \mathfrak{g}_{\beta + k \alpha}\) is a finite-dimensional \(\mathfrak{s}_{\alpha}\)-module. But then \(r_{\alpha}(\beta)\) is a positive root of \(\mathfrak{g(\mathcal{A})}\), of smaller height than \(\beta\), so \(r_{\alpha}(\beta) \in \Delta_{\mathfrak{s}}^+\). This implies that \(\beta \in \Delta_{s}\), a contradiction. Hence \(\Delta = \Delta_{\mathfrak{s}}\). \par

		If \(\Delta\) is finite then because \(\dim \mathfrak{g}_{\beta}<\infty\) for all \(\beta \in \Delta_\mathfrak{g}\), we have \(\mathfrak{g}\) is finite-dimensional.
	\end{proof}

	\begin{lemma}\label{lemma-Dynkin-D}
		There does not exist a completely uncoupled QKM algebra \(\mathfrak{g}(\mathcal{A})\) with the following Dynkin diagram:
		
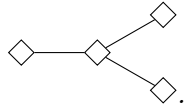
\begin{center}
			\begin{tikzpicture}
				\node[diamond,aspect=1,scale=0.6,draw] (A) at (0,0) {};
				\node[diamond,aspect=1,scale=0.6,draw] (B) at (180:1cm) {};
				\node[diamond,aspect=1,scale=0.6,draw] (C) at (30:1cm) {};
				\node[diamond,aspect=1,scale=0.6,draw] (D) at (-30:1cm) {};
				\draw (A) -- (B);
				\draw (A) -- (C);
				\draw (A) -- (D);
			\end{tikzpicture}.
		\end{center}
	\end{lemma}
	\begin{proof}
		We write \(\alpha_1,\alpha_2,\alpha_3,\alpha_4\) for the simple roots and assume \(\alpha_1\) corresponds to the middle vertex of the Dynkin diagram.
		Lemma \ref{lemma_uncoupled_facts} implies that \([H_1,H_2] \neq 0\).
		As \([H_1,H_2] \in \h_{\alpha_1}\cap\h_{\alpha_2}\), we have \([H_1,H_2] \in \ker \alpha_4\). So
		\begin{align*}
			[H_1,H_2] \in \C\langle c_1,h_1\rangle\cap\ker\alpha_4.
		\end{align*}
		A symmetric argument implies
		\begin{align*}
			[H_1,H_3] \in\mathbb{C}\langle c_1,h_1 \rangle  \cap \ker \alpha_4.
		\end{align*}
		However, since $\alpha_4(h_1)\neq0$, the above subspaces are all one-dimensional and thus equal.  Hence:  
		\begin{align*}
			\mathbb{C}\langle[H_1,H_2]\rangle = \mathbb{C}\langle c_1,h_1 \rangle \cap \ker \alpha_4 = \mathbb{C}\langle[H_1,H_3]\rangle.
		\end{align*}
		
		Therefore $\mathbb{C}\langle [H_1,H_2] \rangle\in(\h_{\alpha_3})_{\ol{0}}\subseteq\ker\alpha_2$.  Thus by Lemma \ref{lemma_QKM_formulas} we have \newline $0=\alpha_2([H_1,H_2])=y_{12}\alpha_2(h_2)$, i.e.~$y_{12}=0$, a contradiction of being uncoupled.
	\end{proof}

	\begin{lemma}\label{lemma-A3-diagrams}
		Suppose \(\mathfrak{g}(\mathcal{A})\) is a completely uncoupled QKM algebra with the following Dynkin diagram:
		\begin{center}
			\begin{tikzpicture}
				\node[diamond,aspect=1,scale=0.6,draw] (A) at (0,0) {};
				\node[diamond,aspect=1,scale=0.6,draw] (B) at (180:1cm) {};
				\node[diamond,aspect=1,scale=0.6,draw] (C) at (0:1cm) {};
				\draw (A) -- (B);
				\draw (A) -- (C);
			\end{tikzpicture}
		\end{center}
		Then \(\mathfrak{g}(\mathcal{A})\) is necessarily of type \(A_3\), that is, its Cartan matrix is \(A = \begin{pmatrix}
			2 & -1 & 0 \\
			-1 & 2 & -1 \\
			0 & -1 & 2
		\end{pmatrix}\).
	\end{lemma}
	\begin{proof}
		Using the notation of Section \ref{section-notation-QKM}, we have, by the proof of Proposition \ref{propistion-two-sq2},
		\begin{align}
			&[H_i,H_j] = x_{ij}c_j + y_{ij}h_j = x_{ji}c_i + y_{ji}h_i, \label{equation-commutator-odd-coroots-general}\\
			&y_{ij}^2\alpha_i(h_j) = y_{ji}^2 \alpha_j(h_i) \label{equation-y-dependence-general}
		\end{align}
		As \(\mathfrak{g}(\mathcal{A})\) is completely uncoupled, we have \(y_{12}y_{21}y_{23}y_{32} \neq 0\). We also have
		\begin{align*}
			\alpha_2([H_1,H_3]) e_2 = [[H_1,H_3],e_2] = (x_{32}y_{12} + x_{12}y_{32})e_2.
		\end{align*}
		By Lemma \ref{lemma-QKM-notation} we have \([H_1,H_3] = 0\), so 
		\begin{align}\label{equation-non-touching-eval}
			0=\alpha_2([H_1,H_3]) = x_{32}y_{12} + x_{12}y_{32}.
		\end{align}
		For \(i=2\), \(j=3\), equation \eqref{equation-commutator-odd-coroots-general} evaluated on \(\alpha_1\) gives
		\begin{align}\label{equation-touching-eval}
			0 = x_{23}\alpha_1(c_3) + y_{23}\alpha_1(h_3) =x_{32}\alpha_1(c_2) + y_{32}\alpha_1(h_2).
		\end{align}
		As \(y_{32} \neq 0\), the linear system \eqref{equation-non-touching-eval} and \eqref{equation-touching-eval} in the variables \(x_{32}\) and \(y_{32}\) admits a non-trivial solution. Therefore
		\begin{align}\label{equation-determinant-relation}
			0 = \begin{vmatrix}
				\alpha_1(c_2) & \alpha_1(h_2) \\ y_{12} & x_{12} 
			\end{vmatrix} = x_{12}\alpha_1(c_2) - y_{12}\alpha_1(h_2).
		\end{align}
		Using \eqref{equation-commutator-odd-coroots-general} again for \(i=1\), \(j=2\) and evaluating on \(\alpha_1\) we obtain
		\begin{align*}
			x_{12}\alpha_1(c_2) + y_{12}\alpha_1(h_2) = x_{21}\alpha_1(c_1) + y_{21}\alpha_1(h_1) = 2y_{21}.
		\end{align*}
		Together with \eqref{equation-determinant-relation}, we deduce \(y_{12}\alpha_1(h_2) = y_{21}\). Finally, \eqref{equation-y-dependence-general} implies
		\begin{align*}
			\alpha_1(h_2) (y_{21}^2\alpha_2(h_1)) = \alpha_1(h_2)(y_{12}^2\alpha_1(h_2)) = (y_{12}\alpha_1(h_2))^2 = y_{21}^2,
		\end{align*}
		hence \(\alpha_1(h_2)\alpha_2(h_1) = 1\). A symmetric argument gives \(\alpha_3(h_2)\alpha_2(h_3) = 1\), and our claim follows.
	\end{proof}
	
	Using Remark \ref{remark_dynkin_diagram}, we obtain the following as an immediate corollary to Lemmas \ref{lemma-Dynkin-D} and \ref{lemma-A3-diagrams}, and Proposition \ref{propistion-two-sq2}:
	
	\begin{corollary}\label{corollary-fg-cc-QKM}
		Let \(\mathfrak{g}(\mathcal{A})\) be a completely uncoupled indecomposable QKM algebra (with at least two simple roots). If \(\mathfrak{g}(\mathcal{A})\) is of finite growth, then its Dynkin diagram is one of the following types:
		\begin{enumerate}
			\item Type \(A(n)\) for \(n \in \mathbb{Z}_{\geq 2}\), which corresponds to
			
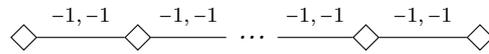
\begin{center}
				\begin{tikzpicture}
					\node[diamond,aspect=1,scale=0.6,draw] (A) at (-3,0) {};
					\node[diamond,aspect=1,scale=0.6,draw] (B) at (-1.5,0) {};
					\node (C) at (0,0) {...};
					\node[diamond,aspect=1,scale=0.6,draw] (D) at (1.5,0) {};
					\node[diamond,aspect=1,scale=0.6,draw] (E) at (3,0) {};
					\draw (A) edge node[above,font=\tiny] {\(-1,-1\)} (B);
					\draw (B) edge node[above,font=\tiny] {\(-1,-1\)} (C) ;
					\draw (C) edge node[above,font=\tiny] {\(-1,-1\)} (D) ;
					\draw (D) edge node[above,font=\tiny] {\(-1,-1\)} (E) ;
				\end{tikzpicture}
			\end{center}
			with \(n\) vertices.
			\item Type \(A(n)^{(1)}\) for \(n \in \mathbb{Z}_{\geq 2}\), which corresponds to
			
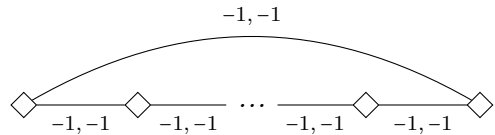
\begin{center}
				\begin{tikzpicture}
					\node[diamond,aspect=1,scale=0.6,draw] (A) at (-3,0) {};
					\node[diamond,aspect=1,scale=0.6,draw] (B) at (-1.5,0) {};
					\node (C) at (0,0) {...};
					\node[diamond,aspect=1,scale=0.6,draw] (D) at (1.5,0) {};
					\node[diamond,aspect=1,scale=0.6,draw] (E) at (3,0) {};
					\draw (A) edge node[below,font=\tiny] {\(-1,-1\)} (B);
					\draw (B) edge node[below,font=\tiny] {\(-1,-1\)} (C) ;
					\draw (C) edge node[below,font=\tiny] {\(-1,-1\)} (D) ;
					\draw (D) edge node[below,font=\tiny] {\(-1,-1\)} (E) ;
					\draw (A) edge [bend left] node[above,font=\tiny] {\(-1,-1\)} (E);
				\end{tikzpicture}
			\end{center}
			with \(n+1\) vertices.
			\item Type \(A(1)^{(1)}\), which corresponds to
			
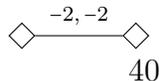
\begin{center}
				\begin{tikzpicture}
					\node[diamond,aspect=1,scale=0.6,draw] (A) at (-0.75,0) {};
					\node[diamond,aspect=1,scale=0.6,draw] (B) at (0.75,0) {};
					\draw (A) edge node[above,font=\tiny] {\(-2,-2\)} (B) ;
				\end{tikzpicture}
			\end{center}
		\end{enumerate}
	\end{corollary}
	\begin{remark}
		In the next section we will construct finite growth QKM algebras associated to each of the above diagrams, and see they are of finite growth.   
	\end{remark}
	
	\subsection{Beyond finite growth}
	        So far we have seen that completely uncoupled QKM algebras are very rigid structure.  We do not yet understand which Dynkin diagrams have a completely uncoupled QKM algebra associated to them.

            \textbf{Question} Do the Dynkin diagrams listed in Corollary \ref{corollary-fg-cc-QKM}, along with $	\begin{tikzpicture}
		\node[diamond,aspect=1,scale=0.6,draw] (A) at (-0.75,0) {};
		\node[diamond,aspect=1,scale=0.6,draw] (B) at (0.75,0) {};
		\draw (A) edge node[above,font=\tiny] {\(-m,-n\)} (B) ;
	\end{tikzpicture}$ with $m,n\in\Z_{\geq2}$, exhaust all connected Dynkin diagrams of completely uncoupled QKM algebras?
	
        In order to show that the above question has a positive answer, it suffices to show a having a triangle as a Dynkin diagram is not possible.  Some evidence for this is provided by the following lemma, for which we omit the proof in the interest of space.
        
        \begin{lemma}\label{lemma_triangle}
    
		Suppose that $\g(\AA)$ is a completely uncoupled QKM algebra with Dynkin diagram:
		
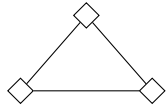
\begin{center}
			\begin{tikzpicture}
				\node[diamond,scale=0.6,draw] (A) at (90:0.5cm) {};
				\node[diamond,scale=0.6,draw] (B) at (210:1cm) {};
				\node[diamond,scale=0.6,draw] (C) at (-30:1cm) {};
				\draw (A) edge node[left,font=\tiny] {} (B);
				\draw (B) edge node[below,font=\tiny] {} (C);
				\draw (C) edge node[right,font=\tiny] {} (A);
			\end{tikzpicture}
		\end{center}
		If $a_{12}a_{21}\neq1$ and $a_{23}a_{32}a_{31}a_{13}\neq1$, then $a_{12}a_{23}a_{31}=a_{21}a_{32}a_{13}$.
	\end{lemma}

	The upshot of Lemma \ref{lemma_triangle} is that a Dynkin diagram with underlying graph of a triangle generically cannot correspond to a completely uncoupled QKM algebra. 	.

	\section{Finite growth, completely uncoupled QKM algebras}
	
	In this section we give explicit presentations for the superalgebras with the Dynkin diagrams of Corollary \ref{corollary-fg-cc-QKM}, i.e.~the finite growth Dynkin diagrams.
	
	\subsection{Type \(A(n)\)}\label{subsection-An}
	The proof for the following theorem can be found in \cite{Si}.   There, an explicit list of possibilities for $\g(\AA)$ is given, under the assumption that 
	\[
	\bigcap_{\alpha \in \Pi} \mathrm{Ann}_{\mathfrak{h}}\mathfrak{g}_{\alpha} \subseteq \sum_{\alpha\in \Pi} \mathfrak{h}_{\alpha}.
	\]
	\begin{theorem}\label{theorem-finite-dimensional-QKM}
		Let \(\mathfrak{g}(\mathcal{A})\) be an uncoupled QKM algebra Dynkin diagram of type \(A(n)\) for \(n \in \mathbb{Z}_{\geq 2}\).  Then the derived subalgebra of $\g(\AA)$ is isomorphic to a trivial central extension of either $\s\q(n)$ or $\p\s\q(n)$.  
	\end{theorem}
	
	\subsection{Type \({A(n)}^{(1)}\)}
	
	Let \(\mathfrak{g}(\mathcal{A})\) be a completely coupled QKM algebra with Dynkin diagram \({A(n)}^{(1)}\). We aim to show that \(\mathfrak{g}(\mathcal{A})\) is an affinization of a QKM algebra with a Dynkin diagram of type \(A(n)\). By an odd affinization we shall mean the following:
	
	\begin{definition}[Odd Affinization]
		Let \(\mathfrak{g}\) be a Lie superalgebra that admits an odd, invariant, super-symmetric bilinear form \((-,-)\). Then the Lie superalgebra
		\begin{align*}
			\hat{\mathfrak{g}} \coloneq \mathfrak{g}\otimes \mathbb{C}[t,t^{-1}] \oplus \mathbb{C}\langle t\partial_t, K  \rangle
		\end{align*}
		with \(t\) even, \(K\) odd and central, and bracket given by
		\begin{align*}
			[x\otimes t^m, y\otimes t^n] = [x,y]\otimes t^{m+n} + m \delta_{m+n,0} (x,y) K,
		\end{align*}
		is called the odd affinization of \(\mathfrak{g}\) with respect to the bilinear form \((-,-)\).
	\end{definition}
	
	Now let $\AA$ be a completely uncoupled Cartan datum with Dynkin diagram of type $A(n)^{(1)}$.  Denote the simple roots of \(\mathfrak{g}(\mathcal{A})\) by \(\alpha_0,...,\alpha_n\); we will consider all indices mod \(n+1\).  A direct computation (see \cite{Si}) shows that we can rescale \(E_i\) such that \(y_{i+1,i} = - y_{i,i+1} = 1\) and \(x_{i+1,i} = x_{i,i+1} = -1\). In particular \(x_{ij} = a_{ij}\) for any \(i,j\). Moreover, it is shown there that \(\sum_{i=0}^n h_i = 0\). \par
	We notice that \(\sum_{i=0}^n H_i \in \bigcap_{\alpha\in\Pi}\mathrm{Ann}_{\mathfrak{h}}\mathfrak{g}_{\alpha}\), and that any other linear combination of \(H_0,\dots,H_n\) that lies in $\bigcap_{\alpha\in\Pi}\mathrm{Ann}_{\mathfrak{h}}\mathfrak{g}_{\alpha}$ must be proportional to \(\sum_{i=0}^n H_i\). 	We denote \(K \coloneq \sum_{i=0}^n H_i\) and notice that \([K,K] = 0\). 
	
	Part (1) of Lemma \ref{lemma_QKM_formulas} and the above values of \(x_{ij}\) and \(y_{ij}\) imply:
	\begin{align*}
		&-c_0 - h_0 = -c_n + h_n, \\
		&-c_0 + h_0 = -c_1 - h_1,
	\end{align*}
	so that
	\begin{align}\label{equation-psqn-realtions}
		2c_0 = c_1 + h_1 + c_n - h_n. 
	\end{align}
	Finally, let \(\alpha_0,...,\alpha_n,\beta_1,...,\beta_m\) be a basis for \(\mathfrak{h}_{\ol{0}}^*\), we let \(\alpha_0^*,...\alpha_n^*,\beta_1^*,...\beta_m^*\) be the dual basis for \(\mathfrak{h}_{\ol{0}}\).
	
	\begin{theorem}
		Let \(\mathfrak{g}(\mathcal{A})\) be a QKM algebra with Dynkin diagram of type \({A(n)}^{(1)}\) for \(n \in \mathbb{Z}_{\geq 2}\). Then \(\hat{\mathfrak{psq}_n}\) can be identified with a subquotient of \(\mathfrak{g}(\mathcal{A})\), where \(\hat{\mathfrak{psq}}_n\) is the odd affinization of \(\mathfrak{psq}_n\) with respect to its non-degenerate form. \par 
		Concretely, there exists a QKM algebra \(\mathfrak{g}(\mathcal{A}')\) that can be identified with an ideal of \(\mathfrak{g}(\mathcal{A})\) such that \(\mathfrak{g}(\mathcal{A}) =  \mathfrak{g}(\mathcal{A}') + \mathfrak{h}_{\ol{1}}\), and \(\hat{\mathfrak{psq}_n} \simeq \mathfrak{g}(\mathcal{A}')/\mathfrak{c}'_{\ol{0}}\), where \(\mathfrak{c}'_{\ol{0}}\) is the even part of the center of \(\mathfrak{g}(\mathcal{A}')\). 
	\end{theorem}
	
	\begin{proof}
		We use the notation of \cite{G} for a basis of \(\mathfrak{psq}_n\). \par	
		Let \(\mathfrak{h}'\) be a subalgebra of \(\mathfrak{h}\) given by \(\mathfrak{t}' = \mathfrak{t}\) and \(\mathfrak{h}' = \mathbb{C}\langle K,H_1,...,H_n \rangle\) (as in the notation above).
		Let \(\mathcal{A}'\) be the Cartan datum obtain from \(\mathcal{A}\) by replacing \(\mathfrak{h}\) by \(\mathfrak{h}'\) (and restricting the \(\mathfrak{h}\)-modules \(\mathfrak{g}_{\alpha}\) to \(\mathfrak{h}'\)). Clearly \(\mathfrak{g}(\mathcal{A}')\) is a QKM algebra and can be embedded naturally as an ideal of \(\mathfrak{g}(\mathcal{A})\) satisfying \(\mathfrak{g}(\mathcal{A}) = \mathfrak{g}(\mathcal{A}')+\mathfrak{h}_{\ol{1}}\). \par
		Now define a map from \(\tilde{\mathfrak{g}}(\mathcal{A}')\) onto \(\hat{\mathfrak{psq}}_n\) as follows:
		\begin{align*}
			&e_i \mapsto X_{E_{i,i+1}, 0}\otimes 1 \text{ for } 1\leq i \leq n, \\
			&E_i \mapsto X_{0, E_{i,i+1}}\otimes 1 \text{ for } 1\leq i \leq n, \\
			&f_i \mapsto X_{E_{i+1,i}, 0}\otimes 1 \text{ for } 1\leq i \leq n, \\
			&F_i \mapsto X_{0, E_{i+1,i}}\otimes 1 \text{ for } 1\leq i \leq n, \\
			&e_0 \mapsto X_{E_{n,0},0}\otimes t, \\
			&E_0 \mapsto X_{0,E_{n,0}}\otimes t, \\
			&f_0 \mapsto X_{E_{0,n},0}\otimes t^{-1}, \\ 
			&F_0 \mapsto X_{0,E_{0,n}}\otimes t^{-1}, \\
			&\alpha_0^* \mapsto t\delta_t, \\
			&h_i \mapsto X_{E_{i,i}-E_{i+1,i+1},0}\otimes 1 \text{ for } 1 \leq i \leq n, \\
			&\beta_i^* \mapsto 0, \\
			&K \mapsto K, \\
			&H_i \mapsto X_{0,E_{i,i}-E_{i+1,i+1}}\otimes 1 \text{ for } 1 \leq i \leq n.
		\end{align*}
		It is straightforward yet tedious to verify that the above assignment determines a well defined homomorphism of Lie superalgebras, which surjects onto \(\hat{\mathfrak{psq}}_n\). As \(\hat{\mathfrak{psq}}_n\) has no non-trivial ideals intersecting \(\mathfrak{h}\) trivially, this homomorphism factors through a map \(\mathfrak{g}(\mathcal{A})\rightarrow \hat{\mathfrak{psq}}_n\). It is clear that the kernel of this lies entirely in the even part of the center.
	\end{proof}
	
	It should be noted that the odd part of the Cartan subalgebra of \(\mathfrak{g}(\mathcal{A})\) cannot be much more complicated than that of \(\hat{\mathfrak{psq}}_n\). Indeed let \(Y \in \mathfrak{h}_{\ol{1}}\). Because \(x_{ij} = a_{ij}\) and the submatrix of \(A\) obtained by removing the first row and the first column is non-degenerate, there is a unique linear combination \(\sum_{i=1}^n t_i H_i\) such that \([Y-\sum_{i=1}^n t_i H_i,e_j] = 0\) for \(1\leq j \leq n\). So we assume \([Y,e_j] = 0\) for all \(1\leq j \leq n\). Let \(r \in \mathbb{C}\) be the scalar satisfying \([Y,E_1] = re_1\). Then \([Y,H_1] = rh_1\), so
	\begin{align*}
		ra_{12}e_2 = [rh_1, e_2] = [[Y,H_1],e_2] = x_{12}[Y,E_2] = a_{12}[Y,E_2],
	\end{align*}
	implying \([Y,E_2]=re_2\). Repeating this argument, we obtain \([Y,E_j] = re_j\) for all \(1\leq j \leq n\). For \(j=0\), we obtain the following:
	\begin{align*}
		-re_0 &= ra_{10}e_0 = [rh_1, e_0] = [[H_1,Y], e_0] = [H_1,[Y,e_0]] + a_{10}[Y,E_0] = [H_1,[Y,e_0]] - [Y,E_0],\\
		-re_0 & = ra_{n0}e_0 = [rh_n, e_0] = [[H_n,Y], e_0] = [H_n,[Y,e_0]] + a_{n0}[Y,E_0] = [H_n,[Y,e_0]] - [Y,E_0].
	\end{align*}
	These equations imply that $[H_1,[Y,e_0]]=[H_n,[Y,e_0]]$; however since $y_{n,0}=-y_{1,0}$, this implies that $[Y,e_0]=0$.  Thus we obtain further that \([Y,E_0] = re_0\). It is immediate that \([Y,Y]\) is central.

	\subsection{Type \({A(1)}^{(1)}\)}\label{section_q_22}
	Let \(s \in \{\pm 1\}\) and \(x_{12},x_{21},y_{12},y_{21} \in \mathbb{C}\) satisfy \(\frac{y_{12}}{y_{21}} = s\sqrt{\frac{a_{12}}{a_{21}}}\) and \(x_{12}x_{21} = 2 + s\sqrt{a_{12}a_{21}}=2(1+s)\). We assume without loss of generality that \(x_{12}\neq 0\) and define a datum \(\mathcal{A}^{s}_{(2,2)}\). \par
	Let \(\mathfrak{h}\) be a \((3|2)\)-dimensional quasi-toral superalgebra, defined as follows. Let \(\{h_1,h_2,h_3\}\) be a basis for \(\mathfrak{t}\) and \(\{H_1,H_2\}\) a basis for \(\mathfrak{h}_{\ol{1}}\). Let \(\alpha_1,\alpha_2,\alpha_3\) be a basis of \(\mathfrak{t}^*\), defined by 
	\begin{align*}
		(\alpha_j(h_i))_{i,j} = \begin{pmatrix} 2 & -2 & 0 \\ -2 & 2 & 1 \\ 0 & 1 & 0 \end{pmatrix}.
	\end{align*}
	We also write \(\delta = \alpha_1+\alpha_2\). We let \(c_1 = x_{12}y_{12} h_3\) and \(c_2 = \frac{1}{x_{12}}(x_{21}c_1+y_{21}h_1 - y_{12}h_2)\). We define the relations in \(\mathfrak{h}\) by
	\begin{align*}
		&[H_i,H_i] = 2c_i, \\
		&[H_1,H_2] = x_{12}c_2+ y_{12}h_2 = x_{21}c_1 + y_{21}h_1.
	\end{align*}
	We set \(\Pi = \{\alpha_1,\alpha_2\}\). It is clear that \(\alpha_1\) and \(\alpha_2\) are linearly independent in \(\mathfrak{t}^*\). 
	
	We let \(\mathfrak{g}_{\alpha_i}\) be an irreducible \(\mathfrak{h}\)-module with \(e_i\) and \(E_i\) even and odd basis vectors, the action of \(\mathfrak{h}\) given by
	\begin{align*}
		& H_i\cdot(e_i+E_i) = 2E_i, \\
		& H_j\cdot (e_i+E_i) = x_{ji}E_i + y_{ji}e_i.
	\end{align*}
	Our definition of \(\mathfrak{h}\) ensures this is a well defined action. Now we let \(\mathfrak{g}_{-\alpha_i} \coloneq \mathfrak{g}_{\alpha_i}^{\vee}\) and let \(f_i \coloneq e_i^\vee\) and \(F_i \coloneq \sqrt{-1} \cdot E_i^\vee\) be a basis for it. Finally, we let \([-,-]_{\alpha_i}:\mathfrak{g}_{\alpha_i}\otimes \mathfrak{g}_{-\alpha_i} \rightarrow \mathfrak{h}\) be the \(\mathfrak{h}\)-module homomorphism defined by \([e_i,F_i]_{\alpha_i} = H_i\). All this information defines a Cartan datum \(\mathcal{A}_{(2,2)}^s\). It then follows that \(\mathfrak{g}(\mathcal{A}_{(2,2)}^s)\) is a QKM algebra. We write \(\mathfrak{q}^+_{(2,2)} \coloneq \mathfrak{g}(\mathcal{A}_{(2,2)}^{1})\) and \(\mathfrak{q}^-_{(2,2)} \coloneq \mathfrak{g}(\mathcal{A}_{(2,2)}^{-1})\). 
	
	\begin{theorem}
		\(\mathfrak{q}^\pm_{(2,2)}\) are of finite growth, their real root spaces are of dimension \((1|1)\) and their imaginary root space is of dimension \((2|2)\).
	\end{theorem}

	\begin{proof}
		Let \(\mathfrak{g}\) denote either \(\mathfrak{q}^+_{(2,2)}\) or \(\mathfrak{q}^-_{(2,2)}\). Let \(\mathfrak{s}\) be the Kac--Moody Lie algebra corresponding to \(\mathfrak{g}\) as in Theorem \ref{theorem-root-system}. Then \(\mathfrak{s} \simeq \s\l(2)^{(1)}\) and is of finite growth. 
		
		We can associate to \(\mathfrak{g}\) the same Weyl group \(W\) as of \(\mathfrak{s}\), and every reflection of \(W\) can be lifted to an automorphism of \(\mathfrak{g}\). As \(\dim \mathfrak{g}_{\alpha} = (1|1)\) for any \(\alpha \in \Pi\), we have \(\dim \mathfrak{g}_{\beta} = (1|1)\) for any real root \(\beta\in\Delta\). If we write \(\delta = \alpha_1+\alpha_2\) then the last result can be stated as \(\dim \mathfrak{g}_{\alpha+k\delta} = (1|1)\) for any \(\alpha \in \Pi\), \(k \in \mathbb{Z}\).\par
		Finally, Proposition \ref{propistion-two-sq2} implies, without loss of generality, that \(\alpha_2(c_1) \neq 0\). We consider the \(\s\q(2)\) subalgebra of \(\mathfrak{g}\) corresponding to \(\alpha_1\). The representation theory of \(\s\q(2)\) implies that \(\dim \mathfrak{g}_{k\delta} \geq (2|2)\), for any \(k \in \mathbb{Z}\setminus\{0\}\). We notice that \(\delta(c_1) \neq 0\). If \(\mathfrak{g}_{k\delta} > (2|2)\), then there exists \(\theta \in \mathfrak{g}_{k\delta}\) such that \([\mathfrak{g}_{-\alpha_1},\theta] = 0\). This implies \(0 = [c_1,\theta] = k\delta(c_1)\theta\), so \(\delta(c_1) = 0\), a contradiction. We conclude that \(\dim \mathfrak{g}_{k\delta}= (2|2)\) for any \(k \in \mathbb{Z}\setminus\{0\}\). This concludes the proof.
	\end{proof}
	
	As stated in Theorem \ref{thm superconformal}, $\q_{(2,2)}^{\pm}$ are the $d=2$, $\NN=3,4$ twisted superconformal algebras.  This will be shown explicitly in a future work.

	\textsc{\footnotesize Alexander Sherman, Dept. of Mathematics, University of Sydney, Camperdown, Australia} 
	
	\textit{\footnotesize Email address:} \texttt{\footnotesize xandersherm@gmail.com}
	
	\textsc{\footnotesize Lior Silberberg, Dept. of Mathematics, Weizmann Institute of Science, Rehovot, Israel} 
	
	\textit{\footnotesize Email address:} \texttt{\footnotesize lior.silberberg@imj-prg.fr }
\end{document}